\documentclass[a4paper, 10pt, twoside]{article}

\usepackage{amsmath, amscd, amsfonts, amssymb, amsthm, latexsym, url, color, todonotes, bm, framed, rotating, enumerate} 
\usepackage{graphicx}
\usepackage[left=1in, right=1in, top=1.2in, bottom=1in, includefoot, headheight=13.6pt]{geometry}
\usepackage{mathtools}
\input{xy}
\xyoption{all}

\usepackage[ps2pdf=true,colorlinks,breaklinks=true]{hyperref}
\usepackage[figure,table]{hypcap}
\hypersetup{
	bookmarksnumbered,
	pdfstartview={FitH},
	citecolor={black},
	linkcolor={black},
	urlcolor={black},
	pdfpagemode={UseOutlines}
}
\makeatletter
\newcommand\org@hypertarget{}
\let\org@hypertarget\hypertarget
\renewcommand\hypertarget[2]{%
  \Hy@raisedlink{\org@hypertarget{#1}{}}#2%
} 
\makeatother 

\newtheorem{theorem}{Theorem}[section]
\newtheorem{lemma}[theorem]{Lemma}
\newtheorem{corollary}[theorem]{Corollary}
\newtheorem{proposition}[theorem]{Proposition}

\theoremstyle{definition}
\newtheorem{definition}[theorem]{Definition}
\newtheorem{remark}[theorem]{Remark}
\newtheorem{example}[theorem]{Example}

\newcommand{\xysquare}[8]{
\[\xymatrix{
#1 \ar@{#5}[r] \ar@{#6}[d] & #2 \ar@{#7}[d]\\
#3 \ar@{#8}[r] & #4
}\]
}

\DeclareMathOperator*{\projlimf}{``\varprojlim''}

\newcommand{\al}{\alpha}
\newcommand{\bb}{\mathbb}

\newcommand{\blob}{\bullet}

\newcommand{\comment}[1]{}

\newcommand{\comp}{{\hat{\phantom{o}}}}

\newcommand{\ep}{\varepsilon}

\newcommand{\into}{\hookrightarrow}
\newcommand{\isofrom}{\stackrel{\simeq}{\leftarrow}}
\newcommand{\isoto}{\stackrel{\simeq}{\to}}
\newcommand{\Isoto}{\stackrel{\simeq}{\longrightarrow}}

\newcommand{\op}{\operatorname}
\newcommand{\pid}[1]{\langle #1 \rangle}
\renewcommand{\phi}{\varphi}
\newcommand{\quis}{\stackrel{\sim}{\to}}
\newcommand{\res}{\overline}
\newcommand{\roi}{\mathcal{O}}

\newcommand{\sub}[1]{{\mbox{\rm\scriptsize #1}}}

\newcommand{\To}{\longrightarrow}
\newcommand{\ul}[1]{\underline{#1}}

\newcommand{\xto}{\xrightarrow}

\renewcommand{\cal}{\mathcal}
\renewcommand{\hat}{\widehat}
\renewcommand{\frak}{\mathfrak}
\newcommand{\indlim}{\varinjlim}
\renewcommand{\tilde}{\widetilde}
\renewcommand{\Im}{\operatorname{Im}}
\renewcommand{\ker}{\operatorname{Ker}}
\renewcommand{\projlim}{\varprojlim}

\DeclareMathOperator{\Ann}{Ann}

\DeclareMathOperator{\dlog}{dlog}

\DeclareMathOperator{\Ext}{Ext}
\DeclareMathOperator{\Frac}{Frac}
\DeclareMathOperator{\Gal}{Gal}
\DeclareMathOperator{\Hom}{Hom}

\DeclareMathOperator{\Spec}{Spec}
\DeclareMathOperator{\Spa}{Spa}
\DeclareMathOperator{\Spf}{Spf}
\DeclareMathOperator{\Sp}{Sp}

\DeclareMathOperator{\Tor}{Tor}

\newcommand{\lb}[1]{}

\newcommand{\dotimes}{\otimes^{\bb L}}
\newcommand{\xTo}[1]{\stackrel{#1}{\To}}

\usepackage{fancyhdr}

\usepackage{sectsty}
\sectionfont{\Large\sc\centering}
\chapterfont{\large\sc\centering}
\chaptertitlefont{\LARGE\centering}
\partfont{\centering}

\begin{document}
\itemsep0pt

\title{Notes on the $\bb A_\sub{inf}$-cohomology of {\em Integral $p$-adic Hodge theory}}

\author{M.~Morrow}

\date{}

\maketitle

\begin{abstract}
\noindent We present a detailed overview of the construction of the $\bb A_\sub{inf}$-cohomology theory from the preprint ``Integral $p$-adic Hodge theory'', joint with B.~Bhatt and P.~Scholze. We focus particularly on the $p$-adic analogue of the Cartier isomorphism via relative de Rham--Witt complexes.
\end{abstract}

\tableofcontents

\newpage\noindent These are expanded notes of a mini-course, given at l'Institut des Math\'ematiques de Jussieu, 15 Jan.~-- 1 Feb., 2016,  detailing some of the main results of the preprint
\begin{quote}
[BMS] B.~Bhatt, M.~Morrow, P.~Scholze, {\em Integral $p$-adic Hodge theory}, \href{http://arxiv.org/abs/1602.03148}{\tt arXiv:1602.03148}~(2016). \nocite{Bhatt_Morrow_Scholze2}
\end{quote}
More precisely, the goal of these notes is to give a detailed, and largely self-contained, presentation of the construction of the  $\bb A_\sub{inf}$-cohomology theory from [BMS], focussing on the $p$-adic analogue of the Cartier isomorphism via relative de Rham--Witt complexes. By restricting attention to this particular aspect of [BMS], we hope to have made the construction more accessible. However, the reader should only read these notes in conjunction with [BMS] itself and is strongly advised also to consult the surveys \cite{Bhatt2016, Scholze2016} by the other authors, which cover complementary aspects of the theory. In particular, in these notes we do not discuss $q$-de Rham complexes, cotangent complex calculations, Breuil--Kisin(--Fargues) modules, or the crystalline and de Rham comparison theorems of [BMS, \S12--14], as these topics are not strictly required for the construction of the $\bb A_\sub{inf}$-cohomology theory.\footnote{To be precise, there is one step in the construction, namely the equality (dim$_\frak X$) in the proof of Theorem \ref{theorem_main_with_proofs2}, where we will have to assume that the $p$-adic scheme $\frak X$ is defined over a discretely valued field; this assumption can be overcome using the crystalline comparison theorems of [BMS].}


We now offer a brief layout of the notes. Section \ref{section_intro} recalls some classical problems and results of $p$-adic Hodge theory before stating the main theorem of the course, namely the existence of a new cohomology theory for $p$-adic schemes which integrally interpolates \'etale, crystalline and de Rham cohomologies.

Section \ref{section_decalage} introduces the d\'ecalage functor, which modifies a given complex by a small amount of torsion. This functor is absolutely essential to our constructions, as it kills the ``junk torsion'' which so often appears in $p$-adic Hodge theory and thus allows us to establish results integrally. An example of this annihilation of torsion, in the context of Faltings' almost purity theorem, is given in \S\ref{subsection_faltings_purity}.

Section \ref{section_perfectoid} develops the elementary theory of perfectoid rings, emphasising the importance of certain maps $\theta_r,\tilde\theta_r$ which generalise Fontaine's usual map $\theta$ of $p$-adic Hodge theory and are central to the later constructions.

Section \ref{section_pro_etale} is a minimal summary of Scholze's theory of pro-\'etale cohomology for rigid analytic varieties. In particular, in \S\ref{subsection_CL} we explain the usual technique by which the pro-\'etale manifestation of the almost purity theorem allows the pro-\'etale cohomology of ``small'' rigid affinoids to be (almost) calculated in terms of group cohomology related to perfectoid rings; this technique is fundamental to our constructions.

In Section \ref{section_main} we revisit the main theorem and define the new cohomology theory as the hypercohomology of a certain complex $\bb A\Omega_{\frak X/\roi}$. In Theorem \ref{theorem_p_adic_Cartier} we state the {\em $p$-adic Cartier isomorphism}, which identifies the cohomology sheaves of the base change of $\bb A\Omega_{\frak X/\roi}$ along $\tilde\theta_r$ with Langer--Zink's relative de Rham--Witt complex of the $p$-adic scheme $\frak X$. We then deduce all main properties of the new cohomology theory from this $p$-adic Cartier isomorphism.

Section \ref{section_constructing_Witt} reviews Langer--Zink's theory of the relative de Rham--Witt complex, which may be seen as the initial object in the category of Witt complexes, i.e., families of differential graded algebras over the Witt vectors which are equipped with compatible Restriction, Frobenius, and Verschiebung maps. In \S\ref{subsection_constructing_Witt} we present one of  our main constructions, namely building Witt complexes from the data of a commutative algebra (in a derived sense), equipped with a Frobenius, over the infinitesimal period ring $\bb A_\sub{inf}$. In  \S\ref{subsection_local_Cartier} we apply this construction to the group cohomology of a Laurent polynomial algebra and prove that the result is precisely the relative de Rham--Witt complex itself; this is the local calculation which underlies the $p$-adic Cartier isomorphism.

Finally, Section \ref{section_Cartier} sketches the proof of the $p$-adic Cartier isomorphism by reducing to the final calculation of the previous paragraph. This reduction is based on various technical lemmas that the d\'ecalage functor behaves well under base change and taking cohomology, and that it transforms certain almost quasi-isomorphisms into quasi-isomorphisms.
\
\\
\

\noindent{\bf Acknowledgements.}
It is a pleasure to take this chance to thank my coauthors Bhargav Bhatt and Peter Scholze for the discussions and collaboration underlying [BMS], from which all results in these notes are taken. I am also grateful to the participants of the mini-course at l'IMJ on which these notes are based, including J.-F.~Dat, C.~Cornut, L.~Fargues, J.-M.~Fontaine, M.-H.~Nicole, and B.~Klingler, for their many helpful comments and insightful questions.

\newpage
\section{Introduction}\label{section_intro}
\subsection{Mysterious functor and Crystalline Comparison}\label{subsection_mysterious}
Here in \S\ref{subsection_mysterious} we consider the following common situation:
\begin{itemize}\itemsep0pt
\item $K$ a complete discrete valuation field of mixed characteristic; ring of integers $\roi_K$; perfect residue field $k$.
\item $\frak X$ a proper, smooth scheme over $\roi_K$.
\end{itemize}

For $\ell\neq p$, proper base change in \'etale cohomology gives a canonical isomorphism \[H^i_\sub{\'et}(\frak X_{\res k},\bb Z_\ell)\cong H^i_\sub{\'et}(\frak X_{\res K},\bb Z_\ell)\] which is compatible with Galois actions.\footnote{To be precise, the isomorphism depends only on a choice of specialisation of geometric points of $\Spec \roi_K$. A consequence of the compatibility with Galois actions is that the action of $G_K$ on $H^i_\sub{\'et}(\frak X_{\res K},\bb Z_\ell)$ is unramified.} Grothendieck's question of the mysterious functor is often now interpreted as asking what happens in the case $\ell =p$. More precisely, how are $H^i_\sub{\'et}(\frak X_{\res K}):=H^i_\sub{\'et}(\frak X_{\res K}, \bb Z_p)$ and $H^i_\sub{crys}(\frak X_k):=H^i_\sub{crys}(\frak X_k/W(k))$ related? In other words, how does $p$-adic cohomology of $\frak X$ degenerate from the generic to the special fibre?

Grothendieck's question is answered after inverting $p$ by the Crystalline Comparison Theorem (Fontaine--Messing, Bloch--Kato, Faltings, Tsuji, Nizio\l,...), stating that there are natural isomorphisms \[H^i_\sub{crys}(\frak X_k)\otimes_{W(k)}\bb B_\sub{crys}\cong H^i_\sub{\'et}(\frak X_{\res K})\otimes_{\bb Z_p}\bb B_\sub{crys}\] which are compatible with Galois and Frobenius actions (and filtrations after base changing to $\bb B_\sub{dR}$), where $\bb B_\sub{crys}$ and $\bb B_\sub{dR}$ are Fontaine's period rings (which we emphasise contain $1/p$). Hence general theory of period rings implies that \[H^i_\sub{crys}(\frak X_k)[\tfrac1p]=(H^i_\sub{\'et}(\frak X_{\res K})\otimes_{\bb Z_p}\bb B_\sub{crys})^{G_K}\] (i.e., the crystalline Dieudonn\'e module of $H^i_\sub{\'et}(\frak X_{\res K})[\tfrac1p]$, by definition) with $\phi$ on the left induced by $1\otimes\phi$ on the right. In summary, $(H^n_\sub{\'et}(\frak X_{\res K})[\tfrac1p], G_K)$ determines $(H^n_\sub{crys}(\frak X_k)[\tfrac1p], \phi)$. Similarly, in the other direction, $(H^n_\sub{\'et}(\frak X_{\res K})[\tfrac1p], G_K)$ is determined by $(H^n_\sub{crys}(\frak X_k)[\tfrac1p], \phi, \text{Hodge fil.})$.

But what if we do not invert $p$? There are many partial results in the literature, under various hypotheses, which we do not attempt to summarise here, except by offering the simplification that ``everything seems to work integrally if $ie<p-1$'',\footnote{Our results can presumably make this more precise.} where $e$ is the absolute ramification degree of $K$. With no assumptions on ramification degree, dimension, value of $p$, etc., we prove in [BMS] results of the following form:

\begin{enumerate}
\item The torsion in $H^i_\sub{\'et}(\frak X_{\res K})$ is ``less than'' that of $H^i_\sub{crys}(\frak X_k)$. To be precise, \[\op{length}_{\bb Z_p}H^i_\sub{\'et}(\frak X_{\res K})/p^r\le \op{length}_{W(k)}H^i_\sub{crys}(\frak X_k)/p^r\] for all $r\ge1$, as one would expect for a degenerating family of cohomologies. In particular, if $H^i_\sub{crys}(\frak X_k)$ is torsion-free then so is $H^i_\sub{\'et}(\frak X_{\res K})$.
\item If $H^*_\sub{crys}(\frak X_k)$ is torsion-free for $*=i,i+1$, then $(H^i_\sub{\'et}(\frak X_{\res K}), G_K)$ determines $(H^i_\sub{crys}(\frak X_k), \phi)$.
\end{enumerate}

It really is possible that additional torsion appears when degenerating the $p$-adic cohomology from the generic fibre to the special fibre, as the following example indicates (which is labeled a theorem as there seems to be no case of an $\frak X$ as above for which $H^i_\sub{\'et}(\frak X_{\res K})\otimes_{\bb Z_p}W(k)$ and $H^i_\sub{crys}(\frak X)$ were previously known to have non-isomorphic torsion submodules):

\setcounter{theorem}{-1}
\begin{theorem}
There exists a smooth projective relative surface $\frak X$ over $\bb Z_2$ such that $H^i_\sub{\'et}(\frak X_{\res K})$ is torsion-free for all $i\ge0$ but such that $H^2_\sub{crys}(\frak X_k)$ contains non-trivial $2$-torsion.\footnote{In [BMS, Thm.~2.10] we also give an example for which $H^2_\sub{\'et}(\frak X_{\res K})_\sub{tors}=\bb Z/p^2\bb Z$ and $H^2_\sub{crys}(\frak X_k)_\sub{tors}=k\oplus k$.}
\end{theorem}
\begin{proof}[Sketch of proof]
Step 1: Let $S_0$ be any ``singular'' Enriques surface over $\bb F_2$. Here ``singular'' means that the torsion subgroup of the Picard scheme of $S_0$ is isomorphic to the group scheme $\bm\mu_2$, or equivalently that the universal cover of $S_0$ is provided by a $\bb Z/2\bb Z$-Galois cover $\tilde S_0\to S_0$ where $\tilde S_0$ is a $K3$ surface. It is in fact not clear that singular Enriques surfaces exist over $\bb F_2$, but we construct an example. A theorem of Lang--Ogus \cite{Lang1983} states that $S_0$ lifts to a smooth projective scheme $S$ over $\bb Z_2$, and $\tilde S_0$ then of course lifts to a $\bb Z/2\bb Z$-Galois cover $\tilde S\to S$.

The reason we start with an Enriques surface is that Illusie \cite[Prop.~7.3.5]{Illusie1979} has calculated that $H^2_\sub{crys}(S_0/W(\bb F_2))=\bb Z_2^{10}\oplus\bb Z/2\bb Z$, which will be the source of the torsion in crystalline cohomology.

Step 2: Let $E$ be an ordinary elliptic curve over $\bb Z_2$. There is a canonical copy of $\bm \mu_2$ inside the $2$-torsion $E[2]$ (Proof: ordinarity of $E$ implies that $E[2]$ is a twisted form of the group scheme $\bm\mu_2\oplus \bb Z/2\bb Z$; normally $\bm\mu_p$ has $p-1$ twisted forms, but here $p=2$ and so $E[2]\supset\bm\mu_2$). Let $\eta:\bb Z/2\bb Z\to\bm\mu_2$ be the morphism of group schemes over $\bb Z_2$ which is generically an isomorphism but is zero on the special fibre, and consider the composition \[\bb Z/2\bb Z\to\bm\mu_2\into E.\] Now push out the $\bb Z/2\bb Z$-Galois cover $\tilde S\to S$ along this map to form an $E$-torsor \[D:=E\times_{\bb Z/2\bb Z}\tilde S\To S,\] where $D$ is a smooth projective relative $3$-fold over $\bb Z_2$.

Step 3: Calculate the $p$-adic cohomologies of $D$. On the special fibre the torsor becomes trivial by construction, i.e., $D_{\bb F_p}=E_{\bb F_p}\times_{\bb F_p}S_0$, and so the K\"unneth formula implies that $H^2_\sub{crys}(D_{\bb F_2}/W(\bb F_2))$ contains a copy of $H^2_\sub{crys}(S_0/W(\bb F_2))\supset \bb Z/2\bb Z$. Conversely, on the generic fibre, it can be shown that $H^*_\sub{\'et}(D_{\res{\bb Q}_2},\bb Z_2)$ is free for $*=0,1,2$ (the cases $*=0,1$ are automatic, but the case $*=2$ is probably the hardest part of the whole construction, for which we refer to [Prop.~2.2(i), BMS]).

Step 4: Let $\frak X\subseteq D$ be any sufficiently ample smooth hypersurface (this is possible over $\bb F_2$ by Bertini theorems of Gabber and Poonen). The sufficiently ampleness condition ensures that the weak Lefschetz theorem in crystalline cohomology holds integrally, so that $\bb Z/2\bb Z\subset H^2_\sub{crys}(D_{\bb F_2}/W(\bb F_2))\into H^2_\sub{crys}(\frak X_{\bb F_2}/W(\bb F_2))$. To complete the proof it remains only to show that $H^*_\sub{\'et}(\frak X_{\res{\bb Q}_2},\bb Z_2)$ is free for $*=0,1,2,3,4$; the cases $*=0,1$ are standard, and we will check $*=2$ in a moment, whence the cases $*=3,4$ follow from Poincar\'e duality. It remains only to check $*=2$, for which we let $U = D\setminus X$, which is affine, and consider the localisation sequence
\[
H_{c,\sub{\'et}}^2(U_{\res{\bb Q}_2},\bb Z_2)\to H^2_\sub{\'et}(D_{\res{\bb Q}_2},\bb Z_2)\to H^2_\sub{\'et}(\frak X_{\res{\bb Q}_2},\bb Z_2)\to H_{c,\sub{\'et}}^3(U_{\res{\bb Q}_2},\bb Z_2)\to \ldots\ .
\]
Since $U$ is affine, smooth and $3$-dimensional, Artin's cohomological bound (and Poincar\'e duality) implies that $H_{c,\sub{\'et}}^*(U_{\res{\bb Q}_2},\bb Z_2)$ and $H_{c,\sub{\'et}}^*(U_{\res{\bb Q}_2},\bb Z/2\bb Z)$ vanish for $*<3$. It follows that $H_{c,\sub{\'et}}^3(U_{\res{\bb Q}_2},\bb Z_2)$ is free over $\bb Z_2$, whence the displayed localisation sequence and the proved freeness of $H^2_\sub{\'et}(D_{\res{\bb Q}_2},\bb Z_2)$ imply the freeness of $H^2_\sub{\'et}(\frak X_{\res{\bb Q}_2},\bb Z_2)$, as desired.
\end{proof}


\subsection{Statement of main theorem, and outline}\label{subsection_intro}
The following notation will be used repeatedly in these notes:
\begin{itemize}\itemsep0pt
\item $\bb C$ is a complete, non-archimedean, algebraically closed field of mixed characteristic;\footnote{More general, most of the theory which we will present works for any perfectoid field of mixed characteristic which contains all $p$-power roots of unity.} ring of integers~$\roi$; residue field $k$. 
\item $\roi^\flat:=\projlim_\phi\roi/p\roi$ is the {\em tilt} (using Scholze's language \cite{Scholze2012} -- or $\cal R_\roi$ in Fontaine's original notation \cite{Fontaine1994}) of $\roi$; so $\roi^\flat$ is a perfect ring of characteristic $p$ which is the ring of integers of $\bb C^\flat:=\Frac\roi$, which is a complete, non-archimedean, algebraically closed field with residue field $k$.
\item $\bb A_\sub{inf}:=W(\roi^\flat)$ is the first period ring of Fontaine;\footnote{A brief introduction to $\roi^\flat$ and $\bb A_\sub{inf}$ may be found at the beginning of Appendix \ref{appendix_A_inf}.} it is equipped with the usual Witt vector Frobenius $\phi$. There are three key specialisation maps:
\[\xymatrix@C=2cm{
&W(\bb C^\flat)&\\
\roi&\bb A_\sub{inf}\ar@{->>}[l]^{\sub{Fontaine's map }\theta}_{\sub{de Rham}}\ar@{->>}[r]^{\sub{crystalline}}\ar@{^(->}[u]^{\sub{\'etale}}&W(k)
}\]
where Fontaine's map $\theta$ will be discussed in detail, and in greater generality, in Section \ref{section_perfectoid}.
\end{itemize}
The goal of this notes is to provide a detailed overview of the proof of the following theorem, proving the existence of a cohomology theory, taking values in $\bb A_\sub{inf}$-modules, which integrally interpolates the \'etale, crystalline, and de Rham cohomologies of a $p$-adic scheme:

\begin{theorem}
\label{thm:IntroGlobalCompThm}
For any proper, smooth (possibly formal) scheme $\frak X$ over $\mathcal{O}$, there is a perfect complex $R\Gamma_{\bb A}(\frak X)$ of $\bb A_\sub{inf}$-modules, functorial in $\frak X$ and equipped with a $\phi$-semi-linear endomorphism $\phi$, with the following specialisations (which are compatible with Frobenius actions where they exist):
\begin{enumerate}
\item \'Etale: $R\Gamma_{\bb A}(\frak X) \dotimes_{\bb A_\sub{\rm inf}} W(\bb C^\flat) \simeq R\Gamma_\sub{\rm\' et}(X,\mathbb{Z}_p) \dotimes_{\mathbb{Z}_p} W(\bb C^\flat)$, where $X:=\frak X_\bb C$ is the generic fibre of $\frak X$ (viewed as a rigid analytic variety over $\bb C$ in the case that $\frak X$ is a formal scheme)
\item Crystalline: $R\Gamma_{\bb A}(\frak X) \dotimes_{\bb A_\sub{\rm inf}} W(k) \simeq R\Gamma_{\mathrm{crys}}(\mathfrak{X}_k/W(k))$.
\item de Rham: $R\Gamma_{\bb A}(\frak X) \dotimes_{\bb A_\sub{\rm inf}} \mathcal{O} \simeq R\Gamma_{\mathrm{dR}}(\mathfrak{X}/\mathcal{O})$.
\end{enumerate}
The individual cohomology groups \[H^i_\bb A(\frak X):=H^i(R\Gamma_\bb A(\frak X))\] have the following properties:
\begin{enumerate}\setcounter{enumi}{3}
\item $H^i_\bb A(\frak X)$ is a finitely presented $\bb A_\sub{\rm inf}$-module;
\item $H^i_\bb A(\frak X)[\tfrac1p]$ is finite free over $\bb A_\sub{\rm inf}[\tfrac1p]$;
\item $H^i_\bb A(\frak X)$ is equipped with a Frobenius semi-linear endomorphism $\phi$ which becomes an isomorphism after inverting any generator $\xi\in\bb A_\sub{\rm inf}$ of $\ker\theta$, i.e., $\phi:H^i_\bb A(\frak X)[\tfrac1\xi]\isoto H^i_\bb A(\frak X)[\tfrac1{\phi(\xi)}]$;
\item \'Etale: $H^i_\bb A(\frak X)\otimes_{\bb A_\sub{\rm inf}}W(\bb C^\flat)\cong H^i_\sub{\rm \'et}(X,\bb Z_p)\otimes_{\bb Z_p}W(\bb C^\flat)$, whence \[(H^i_\bb A(\frak X)\otimes_{\bb A_\sub{\rm inf}}W(\bb C^\flat))^{\phi=1}\cong H^i_\sub{\rm \'et}(X,\bb Z_p).\]
\item Crystalline: there is a short exact sequence \[0\To H^i_\bb A(\frak X)\otimes_{\bb A_\sub{\rm inf}}W(k)\to H^i_\sub{\rm crys}(\frak X_k/W(k))\To\Tor_1^{\bb A_\sub{\rm inf}}(H^{i+1}_\bb A(\frak X), W(k))\To 0,\] where the $\Tor_1$ term is killed by a power of $p$.
\item de Rham: there is a short exact sequence \[0\To H^i_\bb A(\frak X)\otimes_{\bb A_\sub{\rm inf}}\roi\to H^i_\sub{\rm dR}(\frak X/\roi)\To H^{i+1}_\bb A(\frak X)[\xi]\To 0,\] where the third term is again killed by a power of $p$.
\item If $H^i_\sub{\rm crys}(\frak X_k/W(k))$ or $H^i_\sub{\rm dR}(\frak X/\roi)$ is torsion-free, then $H^i_\bb A(\frak X)$ is a finite free $\bb A_\sub{inf}$-module.

\end{enumerate}

\end{theorem}

\begin{corollary}
Let $\frak X$ be as in the previous theorem, fix $i\ge 0$, and assume $H^i_\sub{crys}(\frak X_k/W(k))$ is torsion-free. Then $H^i_\sub{\rm \'et}(X,\bb Z_p)$ is also torsion-free. If we assume further that $H^{i+1}_\sub{\rm crys}(\frak X_k/W(k))$ is torsion-free, then
\[H^i_\bb A(\frak X)\otimes_{\bb A_\sub{\rm inf}}W(k)=H^i_\sub{\rm crys}(\frak X_k/W(k))\qquad\text{and}\qquad
H^i_\bb A(\frak X)\otimes_{\bb A_\sub{\rm inf}}\roi=H^i_\sub{\rm dR}(\frak X/\roi).\]
\end{corollary}
\begin{proof}
We first mention that the ``whence'' assertion of part (vii) of the previous theorem is the following general, well-known assertion: if $M$ is a finitely generated $\bb Z_p$-module and $K$ is any field of characteristic $p$, then $(M\otimes_{\bb Z_p}W(K))^{\phi=1}=M$ (where $\phi$  really means $1\otimes\phi$).

Now assume $H^i_\sub{crys}(\frak X_k/W(k))$ is torsion-free. Then part (x) of the previous theorem implies that $H_\bb A^i(\frak X)$ is finite free; so from part (vii) we see that $H^i_\sub{\'et}(X,\bb Z_p)$ cannot have torsion. If we also assume $H^{i+1}_\sub{crys}(\frak X_k/W(k))$ is torsion-free, then $H_\bb A^{i+1}(\frak X)$ is again finite free by (x), and so no torsion terms appear in the short exact sequences in parts (viii) and (ix) of the previous theorem.
\end{proof}





Having stated the main theorem, we now give a very brief outline of the ideas which will be used to construct the $\bb A_\sub{inf}$-cohomology theory.
\begin{enumerate}
\item We will define $R\Gamma_\bb A(\frak X)$ to be the Zariski hypercohomology of the following complex of sheaves of $\bb A_\sub{inf}$-modules on the formal scheme $\frak X$: \[\bb A\Omega_{\frak X/\roi}:=L\eta_\mu (R\nu_*(\bb A_{\sub{inf},X}))\] where:
\begin{itemize}\itemsep0pt
\item $\bb A_{\sub{inf},X}$ is a certain period sheaf of $\bb A_\sub{inf}$-modules on the pro-\'etale site $X_\sub{pro\'et}$ of the rigid analytic variety $X$ (note that even if $\frak X$ is an honest scheme over $\roi$, we must view its generic fibre as a rigid analytic variety);
\item $\nu:X_\sub{pro\'et}\to \frak X_\sub{Zar}$ is the projection map to the Zariski site of $\frak X$;
\item $L\eta$ is the d\'ecalage functor which modifies a given complex by a small amount of torsion (in this case with respect to a prescribed element $\mu\in\bb A_\sub{inf}$).
\end{itemize}
\item Parts (ii) and (iii) of Theorem \ref{thm:IntroGlobalCompThm} are proved simultaneously by relating $\bb A\Omega_{\frak X/\roi}$ to Langer--Zink's relative de Rham--Witt complex $W_r\Omega_{\frak X/\roi}^\blob$; indeed, this equals $\Omega_{\frak X/\roi}^\blob$ if $r=1$ (which computes de Rham cohomology of $\frak X$) and satisfies $W_r\Omega_{\frak X/\roi}^\blob\otimes_{W_r(\roi)}W_r(k)=W_r\Omega_{\frak X_k/k}^\blob$ (where $W_r\Omega_{\frak X_k/k}^\blob$ is the classical de Rham--Witt complex of Bloch--Deligne--Illusie computing crystalline cohomology of $\frak X_k$).
\item If $\Spf R$ is an affine open of $\frak X$ (so $R$ is a $p$-adically complete, formally smooth $\roi$-algebra\footnote{Throughout these notes we follow the convention that {\em formally smooth/\'etale} includes the condition of being topologically finitely presented, i.e., a quotient of $\roi\pid{T_1,\dots,T_N}$ by a finitely generated ideal. Under this convention formal smoothness implies flatness. In fact, according to a result of Elkik \cite[Thm~7]{Elkik1973} (see Rmq.~2 on p.~587 for elimination of the Noetherian hypothesis), a $p$-adically complete $\roi$-algebra is formally smooth if and only if it is the $p$-adic completion of a smooth $\roi$-algebra.
\comment{The following three sets of hypotheses on a $p$-adically complete $\roi$-algebra $R$ are equivalent:
\begin{enumerate}[(i)]\itemsep0pt
\item $R$ is the $p$-adic completion of a smooth $\roi$-algebra;
\item $R$ is flat over $\roi$, and $R/p$ is smooth over $\roi/p$;
\item $R$ is formally smooth and topologically of finite type over $\roi$.
\end{enumerate}
The implication (iii)$\Rightarrow$(i) is due to 
}
}) which is {\em small}, i.e., formally \'etale over $\roi\pid{T_1^{\pm1},\dots,T_d^{\pm1}}$ ($:=$ the $p$-adic completion of $\roi[T_1^{\pm1},\dots,T_d^{\pm1}]$), then we will use the almost purity theorem to explicitly calculate $R\Gamma_\sub{Zar}(\Spf R,\bb A\Omega_{\frak X/\roi})$ in terms of group cohomology and Koszul complexes. These calculations can be rephrased using ``$q$-de Rham complexes'' over $\bb A_\sub{inf}$ (=deformations of the de Rham complex), but we do not do so in these notes.
\comment{
\item The following is therefore the rough plan of the notes:
\begin{itemize}\itemsep0pt
\item The d\'ecalage functor $L\eta$.
\item Perfectoid rings.
\item Pro \'etale cohomology.
\item Construction of $\bb A\Omega_{\frak X/\roi}$ and $R\Gamma_\bb A(\frak X)$.
\item Relative de Rham--Witt complexes.
\item q-de Rham complexes.
\end{itemize}
}
\item Two notes on the history of the results: \begin{itemize} \item Early motivation for the existence of $R\Gamma_\bb A(\frak X)$ (e.g., as discussed by Scholze at Harris' 2014 MSRI birthday conference) came from topological cyclic homology. These notes say nothing about that point of view.
\item At the time of writing the announcement of our results \cite{Bhatt_Morrow_Scholze}, we only knew that the definition of $R\Gamma_\bb A(\frak X)$ in part (i) of the remark almost (in the precise sense of Faltings' almost mathematics) had the desired properties of Theorem \ref{thm:IntroGlobalCompThm}, so it was necessary to modify the definition slightly; this modification is no longer necessary.
\end{itemize}
\end{enumerate}

\section{The d\'ecalage functor $L\eta$: modifying torsion}\label{section_decalage}
For a ring $A$ and non-zero divisor $f\in A$, we define the {\em d\'ecalage functor} which was introduced first by Berthelot--Ogus \cite[Chap.~8]{BerthelotOgus1978} following a suggestion of Deligne. It will play a fundamental role in our constructions.

\begin{definition}\label{definition_decalage}
Suppose that $C$ is a cochain complex of $f$-torsion-free $A$-modules. Then we denote by $\eta_fC$ the subcomplex of $C[\tfrac1f]$ defined as \[(\eta_fC)^i:=\{x\in f^iC^i:dx\in f^{i+1}C^{i+1}\}\]
i.e., $\eta_fC$ is the largest subcomplex of $C[\tfrac1f]$ which in degree $i$ is contained in $f^iC^i$ for all $i\in\bb Z$.
\end{definition}

It is easy to compute the cohomology of $\eta_fC$:

\begin{lemma}\label{lemma_calculation_of_cohomol}
The map $C^i\to (\eta_f C)^i$ given by $m\to f^im$ induces a natural isomorphism \[H^i(C)/H^i(C)[f]\isoto H^i(\eta_fC).\]
\end{lemma}
\begin{proof}
It is easy to see that the map induces $H^i(C)\to H^i(\eta_fC)$, and the kernel corresponds to those $x\in C^i$ such that $dx=0$ and $fx\in d(C^{i-1})$, i.e., $H^i(C)[f]$.
\end{proof}

\begin{corollary}
If $C\quis C'$ is a quasi-isomorphism of complexes of $f$-torsion-free $A$-modules, then the induced map $\eta_fC\to\eta_fC'$ is also a quasi-isomorphism.
\end{corollary}
\begin{proof}
Immediate from the previous lemma.
\end{proof}

We may now derive $\eta_f$. There is a well-defined endofunctor $L\eta_f$ of the derived category $D(A)$ defined as follows: if $D\in D(A)$ then pick a quasi-isomorphism $C\quis D$ where $C$ is a cochain complex of $f$-torsion-free $A$-modules (e.g., pick a projective resolution, at least if $D$ is bounded above) and set \[L\eta_f D:=\eta_f C.\] This is well-defined by the previous corollary and standard formalism of derived categories.\footnote{
{\bf Warning}. $L\eta_f$ does not preserve distinguished triangles! For example, if $A=\bb Z$ then $L\eta_p(\bb Z/p\bb Z)=0$ but $L\eta_p(\bb Z/p^2\bb Z)=\bb Z/p\bb Z$.}
The general theory of the functor $L\eta_f$ will be spread out through the notes (see especially Remarks \ref{remark_more_on_Leta} and \ref{remark_base_change}); now we proceed to two important examples.

\subsection{Example 1: Crystalline cohomology}\label{subsection_Leta_crys}
The following proposition is the origin of the d\'ecalage functor, in which $A=W(k)$ and $f=p$; it is closely related to the Cartier isomorphism for the de Rham--Witt complex.

\begin{proposition}
Let $k$ be a perfect field of characteristic $p$ and $R$ a smooth $k$-algebra. Then
\begin{enumerate}
\item (Illusie 1979) The absolute Frobenius $\phi:W\Omega_{R/k}^\blob\to W\Omega_{R/k}^\blob$ is injective and has image $\eta_pW\Omega_{R/k}^\blob$, thus inducing a Frobenius semi-linear isomorphism \[\Phi:W\Omega_{R/W(k)}^\blob\isoto \eta_pW\Omega_{R/W(k)}^\blob.\] 
\item (Berthelot--Ogus 1978) There exists a Frobenius semi-linear quasi-isomorphism \[\Phi:R\Gamma_\sub{crys}(R/W(k))\quis L\eta_pR\Gamma_\sub{crys}(R/W(k)).\]
\end{enumerate}
\end{proposition}
\begin{proof}
Obviously (i)$\Rightarrow$(ii), but (ii) was proved earlier and is the historical origin of $L\eta$: see \cite[Thm.~8.20]{BerthelotOgus1978} (with the zero gauge). Berthelot--Ogus applied it to study the relation between the Newton and Hodge polygons associated to a proper, smooth variety over $k$.

(i) is a trivial consequence of the following standard de Rham--Witt identities:
\begin{itemize}\itemsep0pt
\item $\phi$ has image in $\eta_pW\Omega_{R/k}^\blob$ since $\phi=p^iF$ on $W\Omega_{R/k}^i$ and $d\phi=\phi d$.
\item $\phi$ is injective since $FV=VF=p$.
\item the image of $\phi$ is exactly $\eta_pW\Omega_{R/k}^\blob$ since $d^{-1}(pW\Omega_{R/k}^{i+1})=F(W\Omega_{R/k}^i)$ \cite[Eqn.~I.3.21.1.5]{Illusie1979}.\qedhere
\end{itemize}
\end{proof}


\subsection{``Example 2'': An integral form of Faltings' almost purity theorem}\label{subsection_faltings_purity}
We now present an integral form of (the main consequence of) Faltings' almost purity theorem; we do not need this precise result, but we will make use of Lemma \ref{lemma_technical_0} and the ``goodness'' of the group cohomology established in the course of the proof of Theorem \ref{theorem_integral_purity}. Moreover, readers familiar with Faltings' approach to $p$-adic Hodge theory may find this result quite motivating. To recall Faltings' almost purity theorem we consider the following situation:
\begin{itemize}\itemsep0pt
\item $\bb C$ is a complete, non-archimedean, algebraically closed field of mixed characteristic; ring of integers~$\roi$.
\item $R$ is a $p$-adically complete, formally smooth $\roi$-algebra, which we further assume is connected and {\em small}, i.e., formally \'etale over $\roi\pid{\ul T^{\pm1}}:=\roi\pid{T_1^{\pm1},\dots,T_d^{\pm1}}$. As usual in Faltings' theory, we associate to this the following two rings:
\item $R_\infty:=R\hat\otimes_{\roi\pid{\ul T^{\pm1}}}\roi\pid{\ul T^{\pm1/p^\infty}}$ -- this is acted on by $\Gamma:=\bb Z_p(1)^d$ via $R$-algebra automorphisms in the usual way: given $\gamma\in\Gamma=\Hom_{\bb Z_p}((\bb Q_p/\bb Z_p)^d,\mu_{p^\infty})$ and $k_1,\dots,k_d\in\bb Z[\tfrac1p]$, the action is $\gamma\cdot T_1^{k_1}\cdots T_d^{k_d}:=\gamma(k_1,\dots,k_d)T_1^{k_1}\cdots T_d^{k_d}$.
\item $\res R:=$ the $p$-adic completion of the normalisation of $R$ in the maximal (ind)\'etale extension of $R[\tfrac1p]$ -- this is acted on by $\Delta:=\Gal(R[\tfrac1p])$ via $R$-algebra automorphisms, and its restriction to $R_\infty$ gives the $\Gamma$-action there.
\end{itemize}

Faltings' almost purity theorem states $\res R$ is an ``almost \'etale'' $R_\infty$-algebra, and the main consequence of this is that the resulting map on continuous group cohomology \[R\Gamma_\sub{cont}(\Gamma,R_\infty)\To R\Gamma_\sub{cont}(\Delta,\res R)\] is an almost quasi-isomorphism (i.e., all cohomology groups of the cone are killed by the maximal ideal $\frak m\subset\roi$). This is his key to calculating \'etale cohomology in terms of de Rham cohomology; indeed, $R\Gamma_\sub{cont}(\Delta,\res R)$ is a priori hard to calculate and encodes Galois/\'etale cohomology, while $R\Gamma_\sub{cont}(\Gamma,R_\infty)$ is easy to calculate using Koszul complexes (as we will see in the proof of Theorem \ref{theorem_integral_purity}) and differential forms.

The following is our integral form of this result, in which we apply $L\eta$ with respect to any element $f\in\frak m\subset\roi$:

\begin{theorem}\label{theorem_integral_purity}
Under the above set-up, the induced map \[L\eta_{f}R\Gamma_\sub{cont}(\Gamma,R_\infty)\To L\eta_{f}R\Gamma_\sub{cont}(\Delta,\res R)\] is a quasi-isomorphism (not just an almost quasi-isomorphism!) for any non-zero $f\in\frak m$.
\end{theorem}
\comment{
\begin{remark}\label{remark_integral_purity_Witt_promise}
Once we have developed some of the theory of perfectoid rings, their Witt rings, etc, we will prove in Theorem \ref{theorem_integral_purity_W_r} the more general statement that \[L\eta_{[\zeta_{p^r}]-1}R\Gamma_\sub{cont}(\Gamma,W_r(R_\infty))\To L\eta_{[\zeta_{p^r}]-1}R\Gamma_\sub{cont}(\Delta,W_r(R_\sim))\] is a quasi-isomorphism for each $r\ge1$. As we will explain in Remark \ref{remark_integral_purity_W_r}, it is possible to recover Faltings' almost purity theorem from this family of quasi-isomorphisms; thus we have an integral precise strengthening of almost purity.
\end{remark}
}
\begin{remark}
\begin{enumerate}
\item The proof of Theorem \ref{theorem_integral_purity} requires knowing nothing new about $R\Gamma_\sub{cont}(\Delta,\res R)$: a key remarkable property of $L\eta$ is that it can transform almost quasi-isomorphisms into actual quasi-isomorphisms, having only imposed hypotheses on the domain, not the codomain, of the morphism; this will be explained in the next lemma.
\item The theorem implies that the kernel and cokernel of $H^i_\sub{cont}(\Gamma,R_\infty)\to H^i_\sub{cont}(\Delta,\res R)$ are killed by $f$; since $f$ is any element of $\frak m$, the kernel and cokernel are killed by $\frak m$. Thus Theorem \ref{theorem_integral_purity} is a family of on-the-nose integral results which recovers Faltings' almost quasi-isomorphism $R\Gamma_\sub{cont}(\Gamma,R_\infty)\to R\Gamma_\sub{cont}(\Delta,\res R)$.
\end{enumerate}
\end{remark}

\begin{lemma}\label{lemma_technical_0}
Let $\frak M\subseteq A$ be an ideal of a ring and $f\in \frak M$ a non-zero-divisor. Say that an $A$-module $M$ is ``good'' if and only if both $M$ and $M/fM$ contain no non-zero elements killed by $\frak M$. Then the following statements hold:
\begin{enumerate}\itemsep0pt
\item If $M\to N$ is a homomorphism of $A$-modules with kernel and cokernel killed by $\frak M$, and $M$ is good, then $M/M[f]\to N/N[f]$ is an isomorphism.
\item If $C\to D$ is a morphism of complexes of $A$-modules whose cone is killed by $\frak M$, and all cohomology groups of $C$ are good, then $L\eta_fC\to L\eta_fD$ is a quasi-isomorphism.
\end{enumerate}
\end{lemma}
\begin{proof}
Clearly (ii) is a consequence of (i) and Lemma \ref{lemma_calculation_of_cohomol}. So we must prove (i).

Since the kernel of $M$ is killed by $\frak M$, but $M$ contains no non-zero elements killed by $\frak M$, we see that $M\to N$ is injective, and we will henceforth identify $M$ with a submodule of $N$. Then $M[f]=M\cap N[f]$ and so $M/M[f]\to N/N[f]$ is also injective.

Since the quotient $N/M$ is killed by $\frak M$, there is a chain of inclusions $\frak M fN\subseteq fM\subseteq fN\subseteq M$. But $M/fM$ contains no non-zero elements killed by $\frak M$, so $fM=fN$, and this completes the proof: any $n\in N$ satisfies $fn=fm$ for some $m\in M$, whence $n\equiv m$ mod $N[f]$.
\end{proof}

\begin{proof}[Proof of Theorem \ref{theorem_integral_purity}]
To prove Theorem \ref{theorem_integral_purity} we use Faltings' almost purity theorem and Lemma \ref{lemma_technical_0} (in the context $A=\roi$, $f\in\frak M=\frak m$): so it is enough to show that $H^i_\sub{cont}(\Gamma,R_\infty)$ is ``good'' for all $i\ge0$. This is a standard type of explicit calculation of $H^i_\sub{cont}(\Gamma,R_\infty)$ in terms of Koszul complexes. For the sake of the reader unfamiliar with this type of calculation, the special case that $R=\roi\pid{T^{\pm1}}$ is presented in a footnote;\footnote{\label{footnote_special_case_torus}
In this footnote we carry out the calculation of the proof of Theorem \ref{theorem_integral_purity} when $R=\roi\pid{T^{\pm1}}$, in which case $R_\infty=\roi\pid{T^{\pm1/p^\infty}}$. To reiterate, we must show that $H^i_\sub{cont}(\Gamma,R_\infty)$ is good for all $i\ge0$.

First note that $R_\infty$ admits a $\Gamma$-equivariant decomposition into $\roi$-submodules \[R_\infty=\hat\bigoplus_{k\in\bb Z[\tfrac1p]}\roi T^k\] (where the hat denotes $p$-adic completion of the sum), with the generator $\gamma\in\Gamma$ acting on the rank-one free $\roi$-module $\roi T^k$ as multiplication by $\zeta^k$. Thus $R\Gamma_\sub{cont}(\bb Z_p,\roi T^k)\simeq[\roi\xto{\zeta^k-1}\roi]$ (since the group cohomology of an infinite cyclic group with generator $\gamma$ is computed by the invariants and coinvariants of $\gamma$, and similarly in the case of continuous group cohomology), and so \[R\Gamma\sub{cont}(\bb Z_p,R_\infty)\simeq\hat\bigoplus_{k\in\bb Z[\tfrac1p]}[\roi\xto{\zeta^k-1}\roi]\] (where the hat now denotes the derived $p$-adic completion of the sum of complexes), which has cohomology groups \[H^0_\sub{cont}(\bb Z_p,R_\infty)\cong\hat\bigoplus_{k\in\bb Z}\roi\oplus 0,\qquad H^1_\sub{cont}(\bb Z_p,R_\infty)\cong\hat\bigoplus_{k\in\bb Z}\roi\oplus\bigoplus_{k\in\bb Z[\tfrac1p]\setminus\bb Z}\roi/(\zeta^k-1)\roi\] (once some care is taken regarding the $p$-adic completions: see footnote \ref{footnote_completing_sums}).

We claim that both cohomology groups are good. Since $\roi$ has no non-zero elements killed by $\frak m$, it remains only to prove that the same is true of $\roi/a\roi$, where $a=f$ or $\zeta^k-1$ for some $k\in\bb Z[\tfrac1p]\setminus\bb Z$. But this is an easy argument with valuations: if $x\in\roi$ is almost a multiple of $a$, then $\nu_p(x)+\ep\ge\nu_p(a)$ for all $\ep>0$, whence $\nu_p(x)\ge\nu_p(a)$ and so $x$ is actually a multiple of $a$.
}
here in the main text we will prove the general case. Both there and here we pick a compatible sequence $\zeta_p,\zeta_{p^2},\dots\in\roi$ of $p$-power roots of unity to get a generator $\gamma\in\bb Z_p(1)$ and an identification $\Gamma\cong\bb Z_p^d$; as a convenient abuse of notation, we write $\zeta^k:=\zeta_{p^j}^a$ when $k=a/p^j\in\bb Z[\tfrac 1p]$.

First note that $\roi\pid{\ul T^{\pm1/p^\infty}}$ admits a $\Gamma$-equivariant decomposition into $\roi\pid{\ul T^{\pm1}}$-modules: \[\roi\pid{\ul T^{\pm1/p^\infty}}=\roi\pid{\ul T^{\pm1}}\oplus \roi\pid{\ul T^{\pm1}}^\sub{non-int},\] where \[\roi\pid{\ul T^{\pm1}}^\sub{non-int}:=\hat{\bigoplus_{\substack{k_1,\dots,k_d\in \bb Z[\tfrac 1p]\cap[0,1)\\\sub{not all zero}}}}\roi\pid{\ul T^{\pm1}}T_1^{k_1}\cdots T_d^{k_d}\] (where the hat denotes $p$-adic completion of the sum), with the generators $\gamma_1,\dots,\gamma_d\in\Gamma$ acting on the rank-one free $\roi$-module $\roi T_1^{k_1}\cdots T_d^{k_d}$ respectively as multiplication by $\zeta^{k_1},\dots,\zeta^{k_d}$.

Base changing to $R$ we obtain a similar $\Gamma$-equivariant decomposition of $R_\infty$ into $R$-modules \[R_\infty=R\oplus R_\infty^\sub{non-int},\qquad R_\infty^\sub{non-int}:=\hat{\bigoplus_{\substack{k_1,\dots,k_d\in \bb Z[\tfrac 1p]\cap[0,1)\\\sub{not all zero}}}}RT_1^{k_1}\cdots T_d^{k_d},\] and so $R\Gamma_\sub{cont}(\bb Z_p^d,R_\infty)\simeq R\Gamma_\sub{cont}(\bb Z_p^d,R)\oplus R\Gamma_\sub{cont}(\bb Z_p^d,R_\infty^\sub{non-int})$, where \[R\Gamma_\sub{cont}(\bb Z_p^d,R_\infty^\sub{non-int})\simeq\hat{\bigoplus_{\substack{k_1,\dots,k_d\in \bb Z[\tfrac 1p]\cap[0,1)\\\sub{not all zero}}}}R\Gamma_\sub{cont}(\bb Z_p^d,RT_1^{k_1}\cdots T_d^{k_d})\] (where the hat now denotes the derived $p$-adic completion of the sum of complexes). Now we must calculate $H^i_\sub{cont}(\bb Z_p,?)$ for $?=R$ or $RT_1^{k_1}\cdots T_d^{k_d}$.

In the first case, the action of $\bb Z_p^d$ on $R$ is trivial and so a standard group cohomology fact says that $H^i_\sub{cont}(\bb Z_p^d,R) \cong\bigwedge_R^iR^d$. In the second case, another standard group cohomology fact says that $R\Gamma_\sub{cont}(\bb Z_p^d,RT_1^{k_1}\cdots T_d^{k_d})$ can be calculated by the Koszul complex $K_r(\zeta^{k_1}-1,\dots, \zeta^{k_d}-1)$; then Lemma \ref{lemma_on_Koszul_2} reveals (crucially using that not all $k_i$ are zero) that \[H^i_\sub{cont}(\bb Z_p^d,RT_1^{k_1}\cdots T_d^{k_d})\cong R/(\zeta_{p^r}-1)R^{\binom{d-1}{i-1}}\] where $r:=-\min_{1\le i\le d}\nu_p(k_i)\ge1$ is the smallest integer such that $\zeta_{p^r}-1|\zeta^{k_i}-1$ for all $i=1,\dots,d$.

Assembling\footnote{\label{footnote_completing_sums}This step requires some care about $p$-adic completions: the following straightforward result is sufficient. Suppose $(C_\lambda)_\lambda$ is a family of complexes satisfying the following for all $i\in\bb Z$: the group $H^i(C_\lambda)$ is $p$-adically complete and separated for all $\lambda$, with a bound on its $p$-power-torsion which is independent of $\lambda$. Then $H^i(\hat\bigoplus_\lambda C_k)=\hat\bigoplus_\lambda H^i(C_\lambda)$, where the left hat is the derived $p$-adic completion of the sum of complexes, and the right hat is the usual $p$-adic completion of the sum of cohomology groups. {\em Proof}. Set $C_\sub{disc}:=\bigoplus_\lambda C_\lambda$ and $C=\hat{C_\sub{disc}}$ (derived $p$-adic completion); then the usual short exact sequences associated to a derived $p$-adic completion are
\[\xymatrix@R=4mm{
&0\ar[d]&&&\\
&\projlim_r^1 H^i(C_\sub{disc})[p^r]\ar[d]&&&\\
0\ar[r] & \Ext^1_{\bb Z_p}(\bb Q_p/\bb Z_p,H^i(C_\sub{disc}))\ar[r]\ar[d] & H^i(C)\ar[r] & \Hom_{\bb Z_p}(\bb Q_p/\bb Z_p, H^{i+1}(C_\sub{disc}))\ar[r] & 0\\
&\hat{H^i(C_\sub{disc})}=\hat\bigoplus_\lambda H^i(C_\lambda)\ar[d]&&&\\
&0&&&
}\]
Our assumption that $\bigoplus_\lambda H^i(C_\lambda)$ has bounded $p$-power-torsion implies that the right and top terms vanish. $\square$}
these calculations yields isomorphisms \[H^i_\sub{cont}(\Gamma, R_\infty)\cong \bigwedge_R^iR^d\oplus \bigoplus_{\substack{k_1,\dots,k_d\in \bb Z[\tfrac 1p]\cap[0,1)\\\sub{not all zero}}} R/(\zeta_{p^{-\min_{1\le i\le d}\nu_p(k_i)}}-1)R^{\binom{d-1}{i-1}},\] which we claim is good for each $i\ge0$. That is, we must show that $R$, $R/fR$, and $R/(\zeta_{p^r}-1)R$, for $r\ge1$, contain no non-zero elements killed by $\frak m$. This is trivial for $R$ itself since it is a torsion-free $\roi$-algebra, so it remains to show, for each non-zero $a\in\frak m$, that $R/aR$ contains no non-zero elements killed by $\frak m$; but $R$ is a topologically free $\roi$-module [BMS, Lem.~8.10] and so $R/aR$ is a free $\roi/a\roi$-module, thereby reducing the problem to the analogous assertion for $\roi/a\roi$, which was proved in the final paragraph of footnote \ref{footnote_special_case_torus}.
\end{proof}

\section{Algebraic preliminaries on perfectoid rings}\label{section_perfectoid}
Fix a prime number $p$, and let $A$ be a commutative ring which is $\pi$-adically complete (and separated) for some element $\pi\in A$ dividing $p$. Denoting by $\phi:A/pA\to A/pA$ the absolute Frobenius, we have:
\begin{itemize}\itemsep0pt
\item the {\em tilt} $A^\flat:=\projlim_\phi A/pA$ of $A$, which is a perfect $\bb F_p$-algebra, on which we also denote the absolute Frobenius by $\phi$. We sometimes write elements of $A^\flat$ as $x=(x_0,x_1,\dots)$, where $x_i\in A/pA$ and $x_i^p=x_{i-1}$ for all $i\ge1$, and unless indicated otherwise the ``projection $A^\flat \to A/pA$'' refers to the map $x\mapsto x_0$.
\item the associated ``infinitesimal period ring'' $W(A^\flat)$ of Fontaine, which is denoted by $\bb A_\sub{inf}(A)$ in [BMS]. Note that, since $A^\flat$ is a perfect ring, $W(A^\flat)$ behaves just like the ring of Witt vectors of a perfect field of characteristic $p$: in particular $p$ is a non-zero divisor of $W(A^\flat)$, each element has a unique expansion of the form $[x]+p[y]+p^2[z]+\cdots$, and $W(A^\flat)/p^r=W_r(A^\flat)$ for any $r\ge1$.
\end{itemize}
The goal of this section is to study these construction in more detail, in particular to introduce ring homomorphisms \[\tilde\theta_r,\theta_r:W(A^\flat)\To W_r(A)\] which play a fundamental role in the paper, and to define perfectoid rings. 

\subsection{The maps $\theta_r$, $\tilde\theta_r$}\label{subsection_theta}
The following lemma is helpful in understanding $A^\flat$ and will be used several times; we omit the proof since it is relatively well-known and based on standard $p$-adic or $\pi$-adic approximations:

\begin{lemma}\label{lemma_monoids}
The canonical maps \[\projlim_{x\mapsto x^p}A\To A^\flat=\projlim_\phi A/p A\To\projlim_\phi A/\pi A\] are isomorphisms of monoids (resp.~rings).
\end{lemma}

Before stating the main lemma which permits us to define the maps $\theta_r$, we recall that if $B$ is any ring, then the associated rings of Witt vectors $W_r(B)$ are equipped with three operators:
\[R,F:W_{r+1}(B)\to W_r(B)\qquad V:W_r(B)\to W_{r+1}(B),\] where $R,F$ are ring homomorphisms, and $V$ is merely additive. Therefore we can take the limit over $r$ in two ways (of which the second is probably more familiar): \[\projlim_{r\sub{ wrt } F}W_r(B)\qquad\text{or}\qquad W(B)=\projlim_{r\sub{ wrt } R}W_r(B).\]
\begin{lemma}\label{lemma_witt_alg_1}
Let $A$ be as above, i.e., a ring which is $\pi$-adically complete with respect to some element $\pi\in A$ dividing $p$. Then the following three ring homomorphisms are isomorphisms:
\[\xymatrix{
W(A^\flat)=\projlim_{r\sub{ wrt }R}W_r(A^\flat) & \projlim_{r\sub{ wrt }F}W_r(A^\flat)\ar[l]_-{\phi^\infty}^-{(i)}\ar[d]_-{(ii)}\\
\projlim_{r\sub{ wrt }F}W_r(A)\ar[r]_-{(iii)} & \projlim_{r\sub{ wrt }F}W_r(A/\pi A)
}\]
where
\begin{enumerate}\itemsep0pt
\item $\phi^\infty$ is induced by the homomorphisms $\phi^r:W_r(A^\flat)\to W_r(A^\flat)$ for $r\ge1$;
\item the right vertical arrow is induced by the projection $A^\flat \to A/p A\to A/\pi A$;
\item the bottom horizontal arrow is induced by the projection $A\to A/\pi A$.
\end{enumerate}
There is therefore an induced isomorphism \[W(A^\flat)\Isoto \projlim_{r\sub{ wrt }F}W_r(A)\] making the diagram commute.
\end{lemma}
\begin{proof}
(i) is a formal consequence of $A^\flat$ being a perfect ring and commutativity of 
\[\xymatrix{
W_{r+1}(A^\flat)\ar[r]_\phi^\cong\ar[dr]_F & W_{r+1}(A^\flat)\ar[d]^R\\
& W_r(A^\flat)
}\]
(ii) Since $W_r$ commutes with inverse limits of rings we have \[\projlim_{r\sub{ wrt }F}W_r(A^\flat)=\projlim_{r\sub{ wrt }F}\projlim_\phi W_r(A/\pi A)=\projlim_\phi \projlim_{r\sub{ wrt }F}W_r(A/\pi A)\isoto  \projlim_{r\sub{ wrt }F}W_r(A/\pi A),\] where the first equality uses Lemma \ref{lemma_monoids}, and the final projection is an isomorphism since $\phi$ induces an automorphism of the ring $\projlim_{r\sub{ wrt }F}W_r(A/\pi A)$ (thanks again to the formulae $R\phi=\phi R=F$ in characteristic $p$).

(iii) is the most subtle part. Firstly, for any fixed $s\ge1$ we claim that the canonical morphism of pro-rings
\[
\{W_r(A/\pi^sA)\}_{r\sub{ wrt }F}\to \{W_r(A/\pi A)\}_{r\sub{ wrt }F}
\]
is an isomorphism. As it is level-wise surjective, it is enough to show that the kernel $\{W_r(\pi A/\pi^sA)\}_s$ is pro-isomorphic to zero; fix $r\ge 1$. There exists $c\ge 1$ such that $p^c$ is zero in $W_r(A/\pi^sA)$,\footnote{
{\em Proof}. Since $p^r$ vanishes in the $W_r(\bb F_p)$-algebra $W_r(A/pA)$, it also vanishes in $W_r(A/\pi A)$, and therefore its Witt vector expansion has the form $p^r=(a_0\pi,\dots,a_{r-1}\pi)=\sum_{i=0}^{r-1}V^i[a_i\pi]$ for some $a_i\in A$. Using the identity $V^i[a]V^j[b]=p^iV^j[ab^{p^{j-i}}]$ for $a,b\in A$ and $j\ge i$, it is now easy to see that if $c\gg0$ (depending on $r$ and $s$) then $p^c$ vanishes in $W_r(A/\pi^sA)$. $\square$} and we claim that $F^{s+c}:W_{r+s+c}(A/\pi^sA)\to W_r(A/\pi^sA)$ kills the kernel $W_{r+s+c}(\pi A/\pi^sA)$. Each element in the kernel can be written as $\sum_{i=0}^{r+s+c-1}V^i[\pi a_i]$ for some $a_i\in A/p^sA$, and indeed
\[
F^{s+c}\left(\sum_{i=0}^{r+s+c-1}V^i[\pi a_i]\right)=\sum_{i=0}^{c-1}p^i[\pi a_i]^{p^{s+c-i}}+p^cF^s\left(\sum_{i=c}^{r+s+c-1}V^{i-c}[\pi a_i]\right)=0.
\]
This proves the desired pro-isomorphism, from which it follows that \[\projlim_{r\sub{ wrt }F}W_r(A/\pi^s A)\Isoto \projlim_{r\sub{ wrt }F}W_r(A/\pi A).\] Taking the limit over $s\ge 1$, exchanging the order of the limits, and using the isomorphism $W_r(A)\isoto\projlim_sW_s(A/\pi^sA)$ (again since $W_r$ commutes with inverse limits) completes the proof.
\end{proof}

\begin{definition}
Continue to let $A$ be as in the previous lemma, and $r\ge1$. Define $\tilde\theta_r:W(A^\flat)\to W_r(A)$ to be the composition
\[\boxed{\tilde\theta_r:W(A^\flat)\Isoto \projlim_{r\sub{ wrt }F}W_r(A)\To W_r(A),}\] where the first map is the isomorphism of the previous lemma, and the second map is the canonical projection. Also define \[\boxed{\theta_r:=\tilde\theta_r\circ\phi^r:W(A^\flat)\To W_r(A).}\]

We stress that the Frobenius maps $F:W_{r+1}(A)\to W_r(A)$ need not be surjective, and thus $\theta_r,\tilde\theta_r$ need not be surjective; indeed, such surjectivity will be part of the definition of a perfectoid ring (see Lemma \ref{lemma_frobenius_surjectivity}).
\end{definition}

To explicitly describe the maps $\theta_r$ and $\tilde \theta_r$, we follow the usual convention of exploiting the isomorphism of monoids of Lemma \ref{lemma_monoids} to denote an element $x\in A^\flat$ either as $x=(x_0,x_1,\dots)\in \projlim_\phi A/pA$ or $x=(x^{(0)},x^{(1)},\dots)\in \projlim_{x\mapsto x^p}A$:

\begin{lemma}\label{lemma_theta}
For any $x\in A^\flat$ we have $\theta_r([x])=[x^{(0)}]\in W_r(A)$ and $\tilde\theta_r([x]) = [x^{(r)}]$ for $r\ge 1$.
\end{lemma}
\begin{proof}
The formula for $\tilde\theta_r$ follows from a straightforward chase through the above isomorphisms, and the corresponding formula for $\theta_r$ is an immediate consequence.
\end{proof}

In particular, Lemma \ref{lemma_theta} implies that $\theta:=\theta_1:W(A^\flat)\to A$ is the usual map of $p$-adic Hodge theory as defined by Fontaine \cite[\S1.2]{Fontaine1994}, and also shows that the diagram \[\xymatrix{
W(A^\flat)\ar[r]^{\theta_r}\ar[d] & W_r(A)\ar[d]\\
W_r(A^\flat)\ar[r] & W_r(A/pA)
}\]
commutes, where the left arrow is the canonical restriction map and the bottom arrow is induced by the projection $A^\flat\to A/p A$.

The following records the compatibility of the maps $\theta_r$ and $\tilde\theta_r$ with the usual operators on the Witt groups:

\begin{lemma}\label{lemma_theta_r_diagrams}
Continue to let $A$ be as in the previous two lemmas. Then the following diagrams commute:
\[\xymatrix{
W(A^\flat)\ar[d]_{\op{id}}\ar[r]^{\theta_{r+1}}& W_{r+1}(A)\ar[d]^R\\
W(A^\flat)\ar[r]^{\theta_r}& W_r(A)\\
}\quad
\xymatrix{
W(A^\flat)\ar[d]_{\phi}\ar[r]^{\theta_{r+1}}& W_{r+1}(A)\ar[d]^F\\
W(A^\flat)\ar[r]^{\theta_r}& W_r(A)\\
}\quad
\xymatrix{
W(A^\flat)\ar[r]^{\theta_{r+1}}& W_{r+1}(A)\\
W(A^\flat)\ar[r]^{\theta_r}\ar[u]^{\lambda_{r+1}\phi^{-1}}& W_r(A)\ar[u]_V\\
}
\]
where the third diagram requires an element $\lambda_{r+1}\in W(A^\flat)$ satisfying $\theta_{r+1}(\lambda_{r+1})=V(1)$ in $W_{r+1}(A)$. Equivalently, the following diagrams commute:
\[\xymatrix{
W(A^\flat)\ar[d]_{\phi^{-1}}\ar[r]^{\tilde\theta_{r+1}}& W_{r+1}(A)\ar[d]^R\\
W(A^\flat)\ar[r]^{\tilde\theta_r}& W_r(A)\\
}\quad
\xymatrix{
W(A^\flat)\ar[d]_{\op{id}}\ar[r]^{\tilde\theta_{r+1}}& W_{r+1}(A)\ar[d]^F\\
W(A^\flat)\ar[r]^{\tilde\theta_r}& W_r(A)\\
}\quad
\xymatrix{
W(A^\flat)\ar[r]^{\tilde\theta_{r+1}}& W_{r+1}(A)\\
W(A^\flat)\ar[r]^{\tilde\theta_r}\ar[u]^{\times\phi^{r+1}(\lambda_{r+1})}& W_r(A)\ar[u]_V\\
}
\]

\end{lemma}
\begin{proof}
We check the second set of diagrams, since it is these we will use when constructing Witt complexes (even though it is the first set of diagrams which initially appear more natural). Under the chain of isomorphisms of Lemma \ref{lemma_witt_alg_1} defining $W(A^\flat)\isoto\projlim_{r\sub{ wrt }F}W_r(A)$ it is easy to check that the action of $\phi^{-1}$ on $W(A^\flat)$ corresponds to that of the restriction map $R$ on $\projlim_{r\sub{ wrt }F}W_r(A)$. That is, the diagram
\[\xymatrix{
W(A^\flat)\ar[d]_{\phi^{-1}}\ar[r]^{\tilde\theta_{r+1}}& W_{r+1}(A)\ar[d]^R\\
W(A^\flat)\ar[r]^{\tilde\theta_r}& W_r(A)\\
}\]
commutes. Commutativity of the second diagram follows from the definition of the maps $\tilde\theta_r$. Finally, using commutativity of the second diagram, the commutativity of the third diagram follows from the fact that $VF$ is multiplication by $V(1)$ on $W_{r+1}(A)$.
\end{proof}

\subsection{Perfectoid rings}\label{subsection_perfectoid}
The next goal is to define what it means for $A$ to be perfectoid, which requires discussing surjectivity and injectivity of the Frobenius on $A/pA$. We do this in greater generality than we require, but this greater generality reveals the intimate relation to the map $\theta$ and its generalisations $\theta_r$, $\tilde\theta_r$.

\begin{lemma}\label{lemma_frobenius_surjectivity}
Let $A$ be a ring which is $\pi$-adically complete with respect to some element $\pi\in A$ such that $\pi^p$ divides $p$. Then the following are equivalent:
\begin{enumerate}\itemsep0pt
\item Every element of $A/\pi pA$ is a $p^\sub{th}$-power.
\item Every element of $A/pA$ is a $p^\sub{th}$-power.
\item Every element of $A/\pi^pA$ is a $p^\sub{th}$-power.
\item The Witt vector Frobenius $F:W_{r+1}(A)\to W_r(A)$ is surjective for all $r\ge1$.
\item $\theta_r:W(A^\flat)\to W_r(A)$ is surjective for all $r\ge1$.
\item $\theta:W(A^\flat)\to A$ is surjective.
\end{enumerate}
Moreover, if these equivalent conditions hold then there exist $u, v\in A^\times$ such that $u\pi$ and $vp$ admit systems of $p$-power roots in $A$.
\end{lemma}

\begin{proof}
The implications (i)$\Rightarrow$(ii)$\Rightarrow$(iii) are trivial since $\pi pA\subseteq pA\subseteq\pi^pA$. (v)$\Rightarrow$(vi) is also trivial since $\theta=\theta_1$.

(iii)$\Rightarrow$(i): a simple inductive argument allows us to write any given element $x\in A$ as an infinite sum $x=\sum_{i=0}^\infty x_i^p\pi^{pi}$ for some $x_i\in A$; but then $x\equiv(\sum_{i=0}^\infty x_i\pi^i)^p$ mod $p\pi A$.

(iv)$\Rightarrow$(ii): Clear from the fact that the Frobenius $F:W_2(A)\to W_1(A)=A$ is explicitly given by $(\al_0,\al_1)\mapsto \al_0^p+p\al_1$.

(iv)$\Rightarrow$(v): The hypothesis states that the transition maps in the inverse system $\projlim_{r\sub{ wrt }F}W_r(A)$ are surjective, which implies that each map $\tilde\theta_r$ is surjective, and hence that each map $\theta_r$ is surjective.

(vi)$\Rightarrow$(ii): Clear since any element of $A$ in the image of $\theta$ is a $p^\sub{th}$-power mod $p$.

It remains to show that (ii)$\Rightarrow$(iv), but we will first prove the ``moreover'' assertion using only (i) (which we have shown is equivalent to (ii)). Applying Lemma \ref{lemma_monoids} to both $A$ and $A/\pi pA$ implies that the canonical map $\projlim_{x\mapsto x^p}A\to \projlim_{x\mapsto x^p}A/\pi pA$ is an isomorphism. Applying (i) repeatedly, there therefore exists $\omega\in\projlim_{x\mapsto x^p}A$ such that $\omega^{(0)}\equiv \pi$ mod $\pi pA$ (resp.~$\equiv p$ mod $\pi pA$). Writing $\omega^{(0)}=\pi+\pi px$ (resp.~$\omega^{(0)}=p+\pi px$) for some $x\in A$, the proof of the ``moreover'' assertion is completed by noting that $1+p x\in A^\times$ (resp.~$1+\pi x\in A^\times$).

(ii)$\Rightarrow$(iv): By the ``moreover'' assertion, there exist $\pi'\in A$ and $v\in A^\times$ satisfying $\pi'^p=vp$. Note that $A$ is $\pi'$-adically complete, and so we may apply the implication (ii)$\Rightarrow$(i) for the element $\pi'$ to deduce that every element of $A/\pi'pA$ is a $p^\sub{th}$-power; it follows that every element of $A/Ip$ is a $p^\sub{th}$-power, where $I$ is the ideal $\{a\in A:a^p\in pA\}$. Now apply implication ``(xiv)$^\prime\Rightarrow$(ii)'' of Davis--Kedlaya \cite{DavisKedlaya2014}.
\end{proof}

\begin{lemma}\label{lemma_inj_of_Frob}
Let $A$ be a ring which is $\pi$-adically complete with respect to some element $\pi\in A$ such that $\pi^p$ divides $p$, and assume that the equivalent conditions of the previous lemma are true.
\begin{enumerate}
\item If $\ker\theta$ is a principal ideal of $W(A^\flat)$, then 
\begin{enumerate}\itemsep0pt
\item $\Phi:A/\pi A\to A/\pi^pA$, $a\mapsto a^p$, is an isomorphism;
\item any generator of $\ker\theta$ is a non-zero-divisor;\footnote{\label{footnote_int_dom}In all our cases of interest the ring $A$ will be an integral domain, in which case it may be psychologically comforting to note that $A^\flat$ and $W(A^\flat)$ are also integral domains. {\em Proof}. The ring $W(A^\flat)$ is $p$-adically separated, satisfies $W(A^\flat)/p=A^\flat$, and $p$ is a non-zero-divisor in it (these properties all follow simply from $A^\flat$ being perfect). So, once we show that $A^\flat$ is an integral domain, it will easily follow that $W(A^\flat)$ is also an integral domain. But the fact that $A^\flat$ is an integral domain follows at once from the same property of $A$ using the isomorphism of monids $\projlim_{x\mapsto x^p}A\isoto A^\flat$ which already appeared in Lemma \ref{lemma_witt_alg_1}. $\square$
}
\item an element $\xi\in\ker\theta$ is a generator if and only if it is ``distinguished'', i.e., its Witt vector expansion $\xi=(\xi_0,\xi_1,\dots)$ has the property that $\xi_1$ is a unit of $A^\flat$.
\item any element $\xi\in\ker\theta$ satisfying $\theta_r(\xi)=V(1)\in W_r(A)$ for some $r>1$ is distinguished (and such an element exists for any given $r>1$).
\end{enumerate}
\item Conversely, if $\pi$ is a non-zero-divisor and $\Phi:A/\pi A\to A/\pi^pA$ is an isomorphism (which is automatic if $A$ is integrally closed in $A[\tfrac1\pi]$), then $\ker\theta$ is a principal ideal.
\end{enumerate}
\end{lemma}
\begin{proof}
Rather than copying the proof here, we refer the reader to Lem.~3.10 and Rmk.~3.11 of [BMS]. The only assertion which is not proved there is the paranthetical assertion in (ii), for which we just note that if $A$ is integrally closed in $A[\tfrac1\pi]$, then $\Phi$ is automatically injective: indeed, if $a^p$ divides $\pi^p$, then $(a/\pi)^p\in A$ and so $a/\pi\in A$.
\end{proof}

We can now define a perfectoid ring:

\begin{definition}
A ring $A$ is {\em perfectoid} if and only if the following three conditions hold:
\begin{itemize}\itemsep0pt
\item $A$ is $\pi$-adically complete for some element $\pi\in A$ such that $\pi^p$ divides $p$;
\item the Frobenius map $\phi: A/pA\to A/pA$ is surjective (equivalently, $\theta:W(A^\flat)\to A$ is surjective);
\item the kernel of $\theta: W(A^\flat)\to A$ is principal.
\end{itemize}
\end{definition}

\begin{remark}
The first condition of the definition could be replaced by the seemingly stronger, but actually equivalent and perhaps more natural, condition that ``$A$ is $p$-adically complete and there exists a unit $u\in A^\times$ such that $pu$ is a $p^\sub{th}$-power.'' Indeed, this follows from the final assertion of Lemma \ref{lemma_frobenius_surjectivity}.
\end{remark}

We return to the maps $\theta_r$, describing their kernels in the case of a perfectoid ring:

\begin{lemma}\label{lemma_ker_theta_r}
Suppose that $A$ is a perfectoid ring, and let $\xi\in W(A)$ be any element generating $\ker\theta$ (this exists by Lemma \ref{lemma_frobenius_surjectivity}). Then $\ker\theta_r$ is generated by the non-zero-divisor
\[
\xi_r:=\xi\phi^{-1}(\xi)\cdots\phi^{-(r-1)}(\xi)
\]
for any $r\ge1$, and so $\ker\tilde\theta_r$ is generated by the non-zero-divisor \[\tilde\xi_r:=\phi^r(\xi_r)=\phi(\xi)\cdots\phi^r(\xi).\]
\end{lemma}

\begin{proof} 
It is enough to prove the claim about $\xi_r$, since the claim about $\tilde\xi_r$ then follows by applying $\phi^r$. We prove the result by induction on $r\ge1$, the case $r=1$ being covered by the hypotheses; so fix $r\ge1$ for which the result is true. By Lemma \ref{lemma_inj_of_Frob}(i)(d) we may, after multiplying $\xi$ by a unit (depending on the fixed $r\ge1$), assume that $\theta_{r+1}(\xi)=V(1)$. Hence Lemma \ref{lemma_theta_r_diagrams} implies that there is a commutative diagram
\[\xymatrix{
0\ar[r]&W(A^\flat)\ar[r]^{\xi\phi^{-1}}\ar[d]^{\theta_r}&W(A^\flat)\ar[r]^\theta\ar[d]^{\theta_{r+1}} & A\ar@{=}[d]\ar[r] & 0\\
0\ar[r] & W_r(A)\ar[r]_V & W_{r+1}(A)\ar[r]_{R^r} & A\ar[r]&0
}\]
in which both rows are exact. Since $\ker\theta_r$ is generated by $\xi\phi^{-1}(\xi)\cdots\phi^{-(r-1)}(\xi)$, it follows that $\ker\theta_{r+1}$ is generated by $\xi\phi^{-1}(\xi)\cdots\phi^{-r}(\xi)$, as desired.
\end{proof}

\lb{
We use this to show that Witt rings of perfectoid rings behave well under base change (as though the rings were perfect):

\begin{corollary}\label{corollary_tor_independence_1}
Let $A\to B$ be a homomorphism between perfectoid rings. Then the canonical maps
\[W_j(A)\dotimes_{\theta_r,W(A^\flat)}W(B^\flat)\To W_j(B),\qquad W_j(A)\dotimes_{W_r(A)}W_r(B)\to W_j(B)\] are quasi-isomorphisms for all $1\le j\le r$.
\end{corollary}
\begin{proof} Let $\xi\in W(A^\flat)$ be a generator of $\ker\theta_A$, and let $\xi_j$ be as in Lemma~\ref{lemma_ker_theta_r} , which is a non-zero divisor of $W(A^\flat)$. The condition of being distinguished passes through ring homomorphisms, so the image of $\xi$ in $W(B^\flat)$ is still a distinguished, non-zero-divisor which generates of $\ker\theta_B$, by Lemma \ref{lemma_inj_of_Frob}. Thus we may apply Lemma~\ref{lemma_ker_theta_r} to both $A$ and $B$ to see that
\[
W_j(A)\dotimes_{W(A^\flat)}W(B^\flat)=W(A^\flat)/\xi_j\dotimes_{W(A^\flat)}W(B^\flat)=W(B^\flat)/\xi_j=W_j(B).
\]
Then same is true with $r$ in place of $j$, and so
\[
W_j(A)\dotimes_{W_r(A)}W_r(B)=W_j(A)\dotimes_{W_r(A)}W_r(A)\dotimes_{W(A^\flat)}W(B^\flat)
=W_j(A)\dotimes_{W(A^\flat)}W(B^\flat)=W_j(B),
\]
as required.
\end{proof}
}

We finish this introduction to perfectoid rings with some examples:

\begin{example}[Perfect rings of characteristic $p$]\label{example_perfect}
Suppose that $A$ is a ring of characteristic $p$. Then $A$ is perfectoid if and only if it is perfect. Indeed, if $A$ is perfect, then it is $0$-adically complete, the Frobenius is surjective, and the kernel of $\theta: W(A)\to A$ is generated by $p$. Conversely, if $A$ is perfectoid, then Lemma \ref{lemma_inj_of_Frob}(i)(c) implies that the distinguished element $p\in \ker (\theta: W(A^\flat)\to A)$ must be a generator, whence $W(A^\flat)/p\cong A$; but $W(A^\flat)/p=A^\flat$ is perfect.

In particular, in this case $A^\flat=A$ and the maps $\theta_r:W(A^\flat)\to W_r(A)$ are the canonical Witt vector restriction maps.
\end{example}

\begin{example}\label{example_O}
If $\bb C$ is a complete, non-archimedean algebraic closed field of residue characteristic $p>0$, then its ring of integers $\roi$ is a perfectoid ring. Indeed, if $\bb C$ has equal characteristic $p$ then $\roi$ is perfect and we may appeal to the previous lemma. If $\bb C$ has mixed characteristic (our main case of interest), then $\roi$ is $p^{1/p}$-adically complete, integrally closed in $\roi[\tfrac1{p^{1/p}}]=\bb C$, and every element of $\roi/p\roi$ is a $p^\sub{th}$-power since $\bb C$ is algebraically closed, so we may appeal to Lemma \ref{lemma_inj_of_Frob}(ii); in this situation the ring $W(\roi^\flat)$ will always be denoted by $\bb A_\sub{inf}$.
\end{example}

\begin{example}\label{example_torus}
Let $A$ be a perfectoid ring which is $\pi$-adically complete with respect to some non-zero-divisor $\pi\in A$ such that $\pi^p$ divides $p$. Here we offer some constructions of new perfectoid rings from $A$:
\begin{enumerate}\itemsep0pt
\item The rings $A\langle T_1^{1/p^\infty},\dots,T_d^{1/p^\infty}\rangle$ and $A\langle T_1^{\pm1/p^\infty},\dots,T_d^{\pm1/p^\infty}\rangle$, which are by definition the $\pi$-adic completions of $A[T_1^{1/p^\infty},\dots,T_d^{1/p^\infty}]$ and $A[T_1^{\pm1/p^\infty},\dots,T_d^{\pm1/p^\infty}]$ respectively, are also perfectoid.
\item Any $\pi$-adically complete, formally \'etale $A$-algebra is also perfectoid.
\end{enumerate}
{\em Proof}. Since the $\pi$-adic completeness of the given ring is tautological in each case, we only need to check that $\Phi:B/\pi B\to B/\pi^pB$, $b\mapsto b^p$ is an isomorphism in each case. This is clear for $B=A\langle\ul T^{\pm1/p^\infty}\rangle$ and $A\langle\ul T^{1/p^\infty}\rangle$, and it hold for and $A$-algebra $B$ as in (ii) since the square
\[\xymatrix{
B/\pi\ar[r]^\phi & B/\pi \\
A/\pi\ar[u]\ar[r]_\phi & A/\pi\ar[u]
}\]
is a pushout diagram (the base change of the Frobenius along an \'etale morphism in characteristic $p$ is again the Frobenius). \hfill $\square$
\end{example}

\subsection{Main example: perfectoid rings containing enough roots of unity}\label{subsection_roots_of_unity}
Here in \S\ref{subsection_roots_of_unity} we fix a perfectoid ring $A$ which has no $p$-torsion and which contains a compatible system $\zeta_p,\zeta_{p^2},\dots$ of primitive 
 $p$-power roots of unity (to be precise, since $A$ is not necessarily an integral domain, this means that $\zeta_{p^r}$ is a root of the $p^{r\sub{th}}$ cyclotomic polynomial), which we fix. The simplest example is $\roi$ itself, but we also need the theory for perfectoid algebras containing $\roi$ such as $\roi\pid{T_1^{\pm1/p^\infty},\dots,T_d^{\pm1/p^\infty}}$.

In particular we define particular elements $\ep,\xi,\mu,\dots$, which will be used repeatedly in our main constructions, and so we highlight (or rather box) the primary definitions and relations. Firstly, set
\[\boxed{\ep:=(1,\zeta_p,\zeta_{p^2},\dots)\in A^\flat,\qquad \mu:=[\ep]-1\in W(A^\flat),}\]
and
\[
\boxed{\xi:=1+[\ep^{1/p}]+[\ep^{1/p}]^2+\cdots+[\ep^{1/p}]^{p-1}\in W(A^\flat).}
\]

\begin{lemma}
$\xi$ is a generator of $\ker\theta$ satisfying $\theta_r(\xi)=V(1)$ for all $r\ge1$.
\end{lemma}
\begin{proof}
By Lemma \ref{lemma_inj_of_Frob}(i)(d) it is sufficient to show that $\theta_r(\xi)=V(1)$ for all $r\geq 1$. The ghost map $\op{gh}:W_r(A)\to A^r$ is injective since $A$ is $p$-torsion-free, and so it is sufficient to prove that $\op{gh}(\theta_r(\xi))=\op{gh}(V(1))$. But it follows easily from Lemma \ref{lemma_theta} that the composition $\op{gh}\circ\theta_r:W(A^\flat)\to A^r$ is given by $(\theta,\theta\phi,\dots,\theta\phi^{r-1})$, and so in particular that \[\op{gh}(\theta_r(\xi))=(\theta(\xi),\theta\phi(\xi),\dots,\theta\phi^{r-1}(\xi)).\] Since $\theta(\xi)=0$ and $\op{gh}(V(1))=(0,p,p,p,\dots)$, it remains only to check that $\theta\phi^i(\xi)=p$ for all $i\ge 1$, which is straightforward: \[\theta\phi^i(\xi)=\theta(1+[\ep^{p^{i-1}}]+[\ep^{p^{i-1}}]^2+\cdots+[\ep^{p^{i-1}}]^{p-1})=1+1+\cdots+1=p.\qedhere\]
\end{proof}

It now follows from Lemma \ref{lemma_ker_theta_r} that $\ker\theta_r$ is generated by
\[
\boxed{\xi_r:=\xi\phi^{-1}(\xi)\cdots\phi^{-(r-1)}(\xi)=\sum_{i=0}^{p^r-1}[\ep^{1/p^r}]^i,}
\]
and that $\ker\tilde\theta_r$ is generated by \[\boxed{\tilde\xi_r:=\phi^r(\xi_r)=\phi(\xi)\cdots\phi^r(\xi)}.\]

\begin{proposition}\label{proposition_roots_of_unity}
$\mu$ is a non-zero divisor of $W(A^\flat)$ which satisfies \[\boxed{\mu=\xi_r\phi^{-r}(\mu),\qquad \phi^r(\mu)=\tilde\xi_r\mu,\qquad \tilde\theta_r(\mu)=[\zeta_{p^r}]-1\in W_r(A)}\] for all $r\ge 1$.
\end{proposition}
\begin{proof}
The final identity is immediate from Lemma \ref{lemma_theta}. It is clear that $\mu=\xi\phi^{-1}(\mu)$, whence the identity $\mu=\xi_r\phi^{-r}(\mu)$ follows by a trivial induction on $r$, and the central identity then follows by applying $\phi^r$. To prove that $\mu$ is a non-zero-divisor, it suffices to show that $\tilde\theta_r(\mu)=[\zeta_{p^r}]-1$ is a non-zero-divisor of $W_r(A)$ for all $r\ge 1$ (since $W(A^\flat)=\projlim_{r\sub{ wrt }F}W_r(A)$). Since $A$ is $p$-torsion-free the ghost map is injective and so we may check this by proving that \[\op{gh}([\zeta_{p^r}]-1)=(\zeta_{p^r}-1,\zeta_{p^{r-1}}-1,\dots,\zeta_p-1)\] is a non-zero-divisor of $A^r$; i.e., we must show that $\zeta_{p^r}-1$ is a non-zero-divisor in $A$ for all $r\ge 1$. But $\zeta_{p^r} - 1$ divides $p$, and $A$ is assumed to be $p$-torsion-free.
\end{proof}

\begin{remark}\label{remark_non_primitive}
The reader may wish to note that the Teichm\"uller lifts $[\zeta_p], [\zeta_{p^2}],\dots$ are not primitive $p$-power roots unity in $W_r(A)$ in any reasonable sense. Indeed, it follows from its ghost components $\op{gh}([\zeta_p])=(\zeta_p,1,1,\dots,1)$ that $[\zeta_p]$ is not a root of $X^{p-1}+\cdots+X+1$ when $r>1$.

However, the element $[\zeta_{p^r}]-1\in W_r(A)$ will play a distinguished role in our constructions and so we point out that it is a non-zero-divisor whose powers define the $p$-adic topology. Indeed, it follows from the ghost component calculation of the previous proposition that $[\zeta_{p^r}]-1$ is a root of the polynomial \[((X+1)^{p^r}-1)/X=X^{p^r-1}+pX(\cdots)+p^r,\] whence $p$ divides $([\zeta_{p^r}]-1)^{p^r-1}$, and $[\zeta_{p^r}]-1$ divides $p^r$. A particularly important consequence of this is that $L\eta_{[\zeta_{p^r}]-1}$ commutes with derived $p$-adic completion, by [BMS, Lem.~6.20].
\end{remark}

\section{The pro-\'etale site and its sheaves}\label{section_pro_etale}
In this section we review aspects of pro-\'etale cohomology following \cite[\S3--4]{Scholze2013}, working under the following set-up:
\begin{itemize}\itemsep0pt
\item $\bb C$ is a complete, non-archimedean, algebraically closed field of mixed characteristic; ring of integers~$\roi$ with maximal ideal $\frak m$; residue field $k$.
\item $X$ is a quasi-separated rigid analytic variety over $\bb C$.
\end{itemize}
In particular, we will introduce various pro-\'etale sheaves on $X$ which will play an essential role in our constructions, and explain how to calculate their cohomology via affinoid perfectoids and almost purity theorems.

\subsection{The pro-\'etale site $X_\sub{pro\'et}$}\label{subsection_pro_etale}
We will take for granted that the reader is either familiar with, or can reasonably imagine, \'etale morphisms and coverings of rigid analytic varieties, and we let $X_\sub{\'et}$ denote the associated \'etale site of $X$. To define coverings in $X_\sub{\'et}$ (and soon in $X_\sub{pro\'et}$) it is useful to view $X$ as an adic space,\footnote{There is an equivalence of categories between quasi-separated rigid analytic varieties over $\bb C$ and those adic spaces over $\Spa(\bb C,\roi)$ whose structure map is quasi-separated and locally of finite-type. A collection of \'etale maps $\{f_\lambda:U_\lambda\to U\}$ in $X_\sub{\'et}$ is a cover if and only if it is jointly ``strongly surjective'', which is equivalent to being jointly surjective at the level of adic points. See \cite[\S2.1]{Huber1996}.} and we therefore denote by $|X|$ the underlying topological space of its associated adic space $X^\sub{ad}$: for example, if $T$ is an affinoid $\bb C$-algebra, then $|\Sp T|$ denotes the topological space of (equivalences classes of) all continuous valuations on $T$, not merely those factoring through a maximal ideal (which correspond to the closed points of the adic space).

We now define (a countable version of) Scholze's pro-\'etale site $X_\sub{pro\'et}$ in several steps:
\begin{itemize}
\item An object of $X_\sub{pro\'et}$ is simply a formal inverse system $\cal U=\projlimf_iU_i$ in $X_\sub{\'et}$ of the form
\[\xymatrix@R=5mm{
\vdots\ar[d]\\
U_3\ar[d]^{\sub{fin.~\'et.~surj.}}\\
U_2\ar[d]^{\sub{fin.~\'et.~surj.}}\\
U_1\ar[d]^{\sub{\'et.}}\\
X
}\] In other words, $\cal U$ is the data of a tower of finite \'etale covers of $U_1$, which is \'etale over $X$. The underlying topological space of $\cal U$ is by definition $|\cal U|:=\projlim_i|U_i|$.
\item Up to isomorphism,\footnote{This means that we are permitted to replace the towers $\projlimf_iU_i$ and $\projlimf_iV_i$ by ``obviously isomorphic'' towers, e.g., by inserting or removing some stages of the tower. To be precise, first let $\op{pro-}X_\sub{\'et}$ denote the usual category of countable inverse systems in $X_\sub{\'et}$: its objects are inverse systems $\projlimf_iU_i$ in $X_\sub{\'et}$, and its morphisms are defined by $\Hom(\projlimf_iU_i,\projlimf_jV_j):=\projlim_j\indlim_i\Hom_{X_\sub{\'et}}(U_i,V_j)$. Then call an object $\cal U$ of $\op{pro-}X_\sub{\'et}$ {\em pro-\'etale} if and only if it is isomorphic in $\op{pro-}X_\sub{\'et}$ to an inverse system $\projlimf_iU_i$ whose transition maps are finite \'etale surjective; and call a morphism $f:\cal U\to\cal V$ {\em pro-\'etale} if and only if there exist isomorphisms $\cal U\cong \projlimf_iU_i$ and $\cal V\cong \projlimf_iV_i$ in $\op{pro-}X_\sub{\'et}$ such that $\projlimf_iU_i$ and $\projlimf_iV_i$ have finite \'etale surjective transition maps and such that the resulting morphism $\projlimf_iU_i\to \projlimf_iV_i$ has the shape described in the main text. Then the category $X_\sub{pro\'et}$ is more correctly defined as the full subcategory of $\op{pro-}X_\sub{\'et}$ consisting of pro-\'etale objects, and covers are defined as in the main text using the more correct definition of a pro-\'etale morphism.} a morphism $f:\cal U\to\cal V$ in $X_\sub{pro\'et}$ is simply a compatible family of morphisms between the towers
\[\xymatrix@R=5mm{
\vdots\ar[d]&&\vdots\ar[d]\\
U_3\ar[d]\ar[rr]^{f_3}&&V_3\ar[d]\\
U_2\ar[d]\ar[rr]^{f_2}&&V_2\ar[d]\\
U_1\ar[dr]\ar[rr]^{f_1}&&V_1\ar[dl]\\
&X&
}\]
\item A morphism $f$ as immediately above is called {\em pro-\'etale} if and only if it satisfies the following additional condition: the induced finite \'etale map \[U_{i+1}\To U_i\times_{V_i}V_{i+1}\] is surjective for each $i\ge 1$. It can be shown that this implies that the induced continuous map of topological spaces $|f|:|\cal U|\to|\cal V|$ is an open mapping \cite[Lem.~3.10(iv)]{Scholze2013}.

Then a collection of morphisms $\{f_\lambda:\cal U_\lambda\to\cal U\}$ in $X_\sub{pro\'et}$ is defined to be a cover if and only if each morphism $f_\lambda$ is pro-\'etale and the collection $\{|f_\lambda|:|\cal U_\lambda|\to|\cal U|\}$ is a pointwise covering of the topological space $|\cal U|$. For the proof that this indeed defines a Grothendieck topology we refer the reader to \cite[Lem.~3.10]{Scholze2013}.
\end{itemize}
This completes the definition of the pro-\'etale site $X_\sub{pro\'et}$.\footnote{The topos of abelian sheaves on $X_\sub{pro\'et}$ is ``algebraic'' in the sense of \cite[Def.~VI.2.3]{SGA_IV_II}; see \cite[Prop.~3.12]{Scholze2013} for this and further properties of the site. In particular, it then follows from \cite[Corol.~VI.5.3]{SGA_IV_II} that if $\cal U\in X_\sub{pro\'et}$ is such that $|\cal U|$ is quasi-compact and quasi-separated, then $H^*_\sub{pro\'et}(\cal U,-)$ commutes with filtered inductive limits of sheaves.}

There is an obvious projection functor of sites \[\nu:X_\sub{pro\'et}\To X_\sub{\'et}\] obtained by pulling back any $U\in X_\sub{\'et}$ to the constant tower $\cdots \to U\to U\to U\to X$ in $X_\sub{pro\'et}$; this satisfies the unsurprising\footnote{Nonetheless, a condition is required: we must assume that the topological space $|\cal U|$ is quasi-compact and quasi-separated; this is satisfied in particular when $\cal U$ is a tower of rigid affinoids.} property that if $\cal F$ is a sheaf on $X_\sub{\'et}$, and $\cal U=\projlimf_i U_i\in X_\sub{pro\'et}$, then \[\nu^*\cal F(\cal U)=\indlim_i\cal F(U_i)\] and more generally \[H^i_\sub{pro\'et}(\cal U,\nu^*\cal F)=\indlim H^i_\sub{\'et}(U_i,\cal F)\] for all $i\ge0$ \cite[Lem.~3.16]{Scholze2013}. For this reason the most interesting sheaves on $X_\sub{pro\'et}$ are not obtained via pullback from $X_\sub{\'et}$, although our first examples of sheaves on $X_\sub{pro\'et}$ are of this form.

The {\em integral} and {\em rational structure sheaves} $\roi_{X_\sub{\'et}}^+$ and $\roi_{X_\sub{\'et}}$ on $X_\sub{\'et}$ are defined by \[\roi_{X_\sub{\'et}}^+(\Sp T):=T^\circ\subset T=:\roi_{X_\sub{\'et}}(\Sp T)\] where $\Sp T\in X_\sub{\'et}$ is any rigid affinoid, and $T^\circ$ denotes the subring of power bounded elements inside $T$. The integral structure sheaf was not substantially studied in the classical theory.\footnote{Unlike the rational structure sheaf, the integral structure sheaf can have non-zero higher cohomology on rigid affinoids.} Pulling back then defines the {\em integral} and {\em rational structure sheaves} $\roi_X^+$ and $\roi_X$ on $X_\sub{pro\'et}$ \[\roi_X^+:=\nu^*\roi_{X_\sub{\'et}}^+\subset\roi_X:=\nu^*\roi_{X_\sub{\'et}},\] which are our first examples of sheaves on $X_\sub{pro\'et}$.


We now describe the finer, local nature of the pro-\'etale site by introducing affinoid perfectoids and stating the fundamental role which they play in the theory.

\begin{definition}
An object $\cal U=\projlimf_iU_i$ in $X_\sub{pro\'et}$ is called {\em affinoid perfectoid} if and only if it satisfies the following two conditions:
\begin{itemize}\itemsep0pt
\item $U_i$ is a rigid affinoid, i.e., $U_i=\Sp T_i$ for some affinoid $\bb C$-algebra $T_i$, for each $i\ge1$;
\item and the $p$-adic completion of the ring $\roi_X^+(\cal U)=\indlim_iT_i^\circ$ is a perfectoid ring.\footnote{We emphasise that, in our current set-up, this perfectoid ring will always be the type considered in \S\ref{subsection_roots_of_unity}: indeed, it is $p$-torsion-free since each $T_i$ is $p$-torsion-free, and it contains a compatible sequence of primitive $p$-power roots of unity since it contains $\roi$.\label{footnote_aff_perf_are_as_usual}}
\end{itemize}
\end{definition}
The following key result makes precise the idea that $X$ looks locally perfectoid in the pro-\'etale topology, and that affinoid perfectoids are small enough for their cohomology to almost vanish, thereby allowing them to be used for almost calculations \`a la \v{C}ech, as we will see further in \S\ref{subsection_CL}.

\begin{proposition}[Scholze]\label{proposition_local_pro_etale}
\begin{enumerate}
\item The affinoid perfectoid objects of $X_\sub{pro\'et}$ form a basis for the site.
\item If $\cal U\in X_\sub{pro\'et}$ is affinoid perfectoid, then $H^*_\sub{pro\'et}(\cal U,\roi_X^+/p)$ is almost zero (i.e., killed by $\frak m$) for~$*>0$.
\end{enumerate}
\end{proposition}
\begin{proof}
These are consequences of the tilting formalism and almost purity theorems developed in \cite{Scholze2012}. See Corol.~4.7 and Lem.~4.10 of \cite{Scholze2013}.
\end{proof}

To complement the previous local result we recall also the key global result about pro-\'etale cohomology, which we will need:

\begin{theorem}[Scholze]\label{theorem_global_cohomol}
If the rigid analytic variety $X$ is moreover proper and smooth over $\bb C$, then the canonical map of $\roi/p\roi$-modules \[H^i_\sub{\'et}(X,\bb Z/p\bb Z)\otimes_{\bb Z/p\bb Z}\roi/p\roi\To H^i_\sub{pro\'et}(X,\roi_X^+/p)\] is an almost isomorphism (i.e., the kernel and cokernel are killed by $\frak m$) for all $i\ge0$.
\end{theorem}
\begin{proof}
See \cite[\S5]{Scholze2013}.
\end{proof}

\subsection{More sheaves on $X_\sub{pro\'et}$}\label{subsection_pro_etale_sheaves}
As indicated by Proposition \ref{proposition_local_pro_etale}(ii) and Theorem \ref{theorem_global_cohomol}, the pro-\'etale sheaf $\roi_X^+/p$ on $X$ enjoys some special properties, and this richness passes to the {\em completed integral structure sheaf} \[\hat\roi_X^+:=\projlim_s\roi_X^+/p^s,\] which is probably the most important sheaf on $X_\sub{pro\'et}$. We stress that it is not known whether $\roi_X^+(\cal U)$ coincides with the $p$-adic completion of $\roi_X^+(\cal U)$ for arbitrary objects $\cal U\in X_\sub{pro\'et}$.

Further sheaves of interest on $X_\sub{pro\'et}$ are collected in the following definition:

\begin{definition}
The {\em tilted integral structure sheaf} \footnote{Usually denoted by $\hat\roi_{X^\flat}^+$ to evoke the idea of it being the completed integral structure sheaf on the tilt $X^\flat$ of $X$.} is \[\roi_X^{+\flat}:=\projlim_\phi\roi_X^+/p,\] where the limit is taken over iterations of the Frobenius map $\phi$ on the sheaf of $\bb F_p$-algebras $\roi_X^+/p$. We will also need Witt vector forms\footnote{\label{footnote_Witt}If $\cal R$ is a sheaf of rings on a site $\cal T$, then $W_r(\cal R)$ and $W(\cal R)$ are the sheaves of rings obtained by applying the Witt vector construction section-wise, i.e., $W_r(\cal R)(U):=W_r(\cal R(U))$ and $W(\cal R)(U):=W(\cal R(U))$ for all $U\in\cal T$.} of the completed and tilted integral structure sheaves \[W_r(\hat\roi_X^+)\qquad\text{and}\qquad W_r(\roi_X^{+\flat}),\] and the {\em infinitesimal period sheaf} \[\bb A_{\sub{inf},X}:=W(\roi_X^{+\flat}).\]
\end{definition}

By repeating Lemma \ref{lemma_witt_alg_1} in terms of presheaves on $X_\sub{pro\'et}$ and then sheafifying, we obtain a canonical isomorphism of pro-\'etale sheaves \[\bb A_{\sub{inf},X}\Isoto \projlim_{r\sub{ wrt }F}W_r(\hat\roi_X^+).\] As in the affine case in \S\ref{subsection_theta} we then denote the resulting projection maps and their Frobenius twists by \[\tilde\theta_r:\bb A_{\sub{inf},X}\To W_r(\hat\roi_X^+)\quad\text{and}\quad \theta_r=\tilde\theta_r\circ\phi^r:\bb A_{\sub{inf},X}\To W_r(\hat\roi_X^+).\] To reduce further analysis of all these sheaves to the affine case of Section \ref{section_perfectoid}, we combine the fact that $X$ is locally perfectoid in the pro-\'etale topology (Proposition \ref{proposition_local_pro_etale}(i)) with the fact that the sections of these sheaves on affinoid perfectoids are ``as expected'':

\begin{lemma}[Scholze]\label{lemma_local_structure_of_sheaves}
Let $\cal U=\projlimf_i U_i$ be an affinoid perfectoid in $X_\sub{pro\'et}$, with associated perfectoid ring $A:=\roi_X^+(\cal U)_p^\comp$. Then \[\hat\roi_X^+(\cal U)=A,\quad W_r(\hat\roi_X^+)(\cal U)=W_r(A),\quad \roi_X^{+\flat}(\cal U)=A^\flat,\quad W_r(\roi_X^{+\flat})(\cal U)=W_r(A^\flat),\quad \bb A_{\sub{inf},X}(\cal U)=W(A^\flat).\] 
On the other hand, for $*>0$ the pro-\'etale cohomology groups \[H^*_\sub{pro\'et}(\cal U,\hat\roi_X^+),\quad H^*_\sub{pro\'et}(\cal U,W_r(\hat\roi_X^+)),\quad H^*_\sub{pro\'et}(\cal U,\roi_X^{+\flat}),\quad H^*_\sub{pro\'et}(\cal U,W_r(\roi_X^{+\flat})),\quad H^*_\sub{pro\'et}(\cal U,\bb A_{\sub{inf},X})\]
are almost zero, i.e., killed respectively by $\frak m$, $W_r(\frak m)$, $\frak m^\flat$, $W_r(\frak m^\flat)$, $W(\frak m^\flat)$.\footnote{\label{footnote_almost_W_r}
Now seems to be an appropriate moment for mentioning some formalism of almost mathematics over Witt rings. By a ``setting for almost mathematics'' we mean a pair $(V,I)$, where $V$ is a ring and $I=I^2\subseteq V$ is an ideal which is an increasing union of principal ideals $\bigcup_\lambda t_\lambda V$ generated by non-zero-divisors $t_\lambda$. Elementary manipulations of Witt vectors [Lem.~10.1 \& Corol.~10.2, BMS] then show that each Teichm\"uller lift $[t_\lambda]\in W_r(V)$ is a non-zero-divisor and that $W_r(I):=\ker(W_r(V)\to W_r(V/I))$ equals the increasing union $\bigcup_\lambda[t_\lambda]W_r(V)$, which moreover coincides with its square; in conclusion, the pair $(W_r(V), W_r(I))$ is also a setting for almost mathematics. We apply this above, and elsewhere, in the cases $(V,I)=(\roi,\frak m)$ and $(\roi^\flat,\frak m^\flat)$.

Upon taking the limit as $r\to\infty$, the inclusion $W(I):=\ker(W(V)\to W(V/I))\supset [I]:=\bigcup_\lambda[t_\lambda]W(V)$ is strict; the pair $(W(V),[I])$ is a setting for almost mathematics, but $(W(V),W(I))$ typically is not. So, strictly, speaking, the almost language should be avoided for the ideal $W(\frak m^\flat)$ above. However, if $V$ is a perfect ring of characteristic $p$ (e.g., $V=\roi^\flat$), then $[I]$ and $W(I)$ coincide after $p$-adic completion (and derived $p$-adic completion) by the argument of the proof of Lemma \ref{lemma_Tors_over_A_inf}; so a map between $p$-adically complete $W(V)$-modules (resp.~derived $p$-adically complete complexes of $W(V)$-modules) has kernel and cokernel (resp.~all cohomology groups of the cone) killed by $W(I)$ if and only if they are killed by $[I]$.
}
\end{lemma}
\begin{proof}
See Lems.~4.10, 5.11 and Thm.~6.5 of \cite{Scholze2013} for the description of the sections. The almost vanishings follow by taking suitable limits in Proposition \ref{proposition_local_pro_etale}(ii).
\end{proof}

\begin{corollary}\label{corollary_non_zero_divisors_on_sheaf}
The maps of pro-\'etale sheaves  $\theta_r,\tilde\theta_r:\bb A_{\sub{inf},X}\to W_r(\hat\roi_X^+)$ are surjective, with kernels generated respectively by the elements $\xi_r,\tilde\xi_r\in\bb A_\sub{inf}=W(\roi^\flat)$ defined in \S\ref{subsection_roots_of_unity}; moreover, these elements (as well as $\mu\in\bb A_\sub{inf}$, also defined in \S\ref{subsection_roots_of_unity}) are non-zero-divisors of the sheaf of rings $\bb A_{\sub{inf},X}$.
\end{corollary}
\begin{proof}
All assertions are local, so by Proposition \ref{proposition_local_pro_etale}(i) it is sufficient to prove the analogous affine assertions after taking sections in any affinoid perfectoid $\cal U\in X_\sub{pro\'et}$; but using the descriptions of the sections given by the previous lemma, these affine assertions were covered in \S\ref{subsection_perfectoid}--\ref{subsection_roots_of_unity}.
\end{proof}

\subsection{Calculating pro-\'etale cohomology}\label{subsection_CL}
This section is devoted to an explanation of how Proposition \ref{proposition_local_pro_etale}(ii) is used in practice to (almost) calculate the pro-\'etale cohomology of our sheaves of interest; this is of course the pro-\'etale analogue of Faltings' purity theorem and techniques which we saw in \S\ref{subsection_faltings_purity}. We assume in this section that our rigid analytic variety $X$ is the generic fibre $\frak X_\bb C$ of a smooth formal scheme $\frak X$ over $\roi$; this will be the set-up of our main results later.

Relatively elementary arguments show that $\frak X$ admits a basis of affine opens $\{\Spf R\}$ where each $R$ is a $p$-adically complete, formally smooth $\roi$-algebra which is moreover small, i.e., formally \'etale over $\roi\pid{T_1^{\pm1},\dots,T_d^{\pm1}}$. Fix such an open $\Spf R\subseteq\frak X$ (as well as a formally \'etale map $\roi\pid{\ul T^{\pm1}}\to R$, sometimes called a ``framing''); the associated generic fibre is the rigid affinoid space $U:=\Sp R[\tfrac1p]\subseteq X$, which is equipped with an \'etale morphism to $\Sp \bb C\pid{\ul T^{\pm1}}$. We will explain how to almost calculate the pro-\'etale cohomology groups $H^*_\sub{pro\'et}(\Sp R[\tfrac1p],?)$ where $?$ is any of the sheaves from Lemma \ref{lemma_local_structure_of_sheaves}.

For each $i\ge 1$, let \[R_i:=R\otimes_{\roi\pid{\ul T^{\pm1}}}\roi\pid{\ul T^{\pm1/p^i}}\] be the finite \'etale $R$-algebra obtained by adjoining $p^i$-roots of $T_1,\dots,T_d$. Then $\Sp R_{i+1}[\tfrac1p]\to\Sp R_i[\tfrac1p]$ is a finite \'etale cover of rigid affinoids for each $i\ge0$, whence it easily follows that \[\cal U:=\projlimf_i\Sp R_i[\tfrac1p] \To U\] is a cover in $X_\sub{pro\'et}$.

In fact, $\Sp R_i[\tfrac1p]\to U$ is a finite Galois cover with Galois group $\mu_{p^i}^d$, where $\ul\zeta=(\zeta_1,\dots,\zeta_d)\in\mu_{p^i}^d$ acts on $R_i$ in the usual way via $\ul\zeta\cdot T_1^{j_1/p^i}\cdots T_d^{j_d/p^i}:=\zeta_1^{j_1}\cdots\zeta_d^{j_d}T_1^{j_1/p^i}\cdots T_d^{j_d/p^i}$, and so for each $s\ge1$ there is an associated Cartan--Leray\footnote{Often called Hochschild--Serre in this setting. Here $H^*_\sub{grp}$ and $R\Gamma_\sub{grp}$ refer to group cohomology for a topological group acting on discrete modules.} spectral sequence \[H^a_\sub{grp}(\mu_{p^i}^d, H^b_\sub{pro\'et}(U_i,\roi_X^+/p^s))\implies H^{a+b}_\sub{pro\'et}(U,\roi_X^+/p^s),\] or writing in a more derived fashion \[R\Gamma_\sub{grp}(\mu_{p^i}^d,R\Gamma_\sub{pro\'et}(U_i,\roi_X^+/p^s))\quis R\Gamma_\sub{pro\'et}(U,\roi_X^+/p^s).\] Taking the colimit over $i$ yields an analogous quasi-isomorphism (and spectral sequence) for the ``$\bb Z_p(1)^d$-Galois cover'' $\cal U\to U$: \[R\Gamma_\sub{grp}(\bb Z_p(1)^d,R\Gamma_\sub{pro\'et}(\cal U,\roi_X^+/p^s))\quis R\Gamma_\sub{pro\'et}(U,\roi_X^+/p^s).\] 

However, $\cal U$ is affinoid perfectoid: indeed, since the power bounded elements in the affinoid $\bb C$-algebra $R_i[\tfrac1p]$ are exactly $R_i$, we must show that $(\indlim_iR_i)_p^\comp=R\hat\otimes_{\roi\pid{\ul T^{\pm1}}}\roi\pid{\ul T^{\pm1/p^\infty}}=:R_\infty$ is a perfectoid ring; but $R_\infty$ is a $p$-adically complete, formally \'etale $\roi\pid{\ul T^{\pm1/p^\infty}}$-algebra, whence it is perfectoid by Example~\ref{example_torus}. Therefore the pro-\'etale cohomology $H^*_\sub{pro\'et}(\cal U,\roi_X^+/p^s)$ almost vanishes for $*>0$ (by Proposition \ref{proposition_local_pro_etale}(ii)) and almost equals $R_\infty/p^sR_\infty$ for $*=0$ (by Lemma \ref{lemma_local_structure_of_sheaves}, using that $\roi_X^+/p^s=\hat\roi_X^+/p^s$); so the edge map associated to the previous line is an almost quasi-isomorphism \[R\Gamma_\sub{grp}(\bb Z_p(1)^d,R_\infty/p^sR_\infty)\xTo{\sub{al.~qu.-iso.}} R\Gamma_\sub{pro\'et}(U,\roi_X^+/p^s)\] (i.e., all cohomology groups of the cone are killed by $\frak m$), where we mention explicitly that $\bb Z_p(1)^d$ is acting on $R_\infty$ as in \S\ref{subsection_faltings_purity}. Finally, taking the derived inverse limit\footnote{This process of taking the inverse limit deserves further explanation. By definition, when $G$ is a topological group and $M$ is a complete topological $G$-module whose topology is defined by a system $\{N\}$ of open sub-$G$-submodules, we define its continuous group cohomology as $R\Gamma_\sub{cont}(G,M):=R\!\op{lim}_NR\Gamma_\sub{grp}(G,M/N)$ and $H^*_\sub{cont}(G,M):=H^*(R\Gamma_\sub{cont}(G,M))$; of course, we may always restrict the limit to any preferred system of open neighbourhoods of $0$ by sub-$G$-modules. In particular, $R\Gamma_\sub{cont}(\bb Z_p(1)^d,R_\infty)=R\!\op{lim}_sR\Gamma_\sub{grp}(\bb Z_p(1)^d,R_\infty/p^sR_\infty)$.

To take the inverse limit of the right, we show that the canonical map $R\Gamma_\sub{pro\'et}(U,\hat\roi_X^+)\to R\!\op{lim}_sR\Gamma_\sub{pro\'et}(U,\roi_X^+/p^s)$ is a quasi-isomorphism. Since the codomain may be rewritten as $R\Gamma_\sub{pro\'et}(U,R\!\op{lim}_s\roi_X^+/p^s)$ by general formalism of derived functors, it is enough to show that the canonical map $\hat\roi_X^+\to R\!\op{lim}_s\roi_X^+/p^s$ is a quasi-isomorphism (note that the topos of pro-\'etale sheaves does not satisfy the necessary Grothendieck axioms to automatically imply that higher derived inverse limits vanish in the case of surjective transition maps!), for which it is enough to show that $R\Gamma_\sub{pro\'et}(\cal V,\hat\roi_X^+)\to R\!\op{lim}_sR\Gamma_\sub{pro\'et}(\cal V,\roi_X^+/p^s)$ is a quasi-isomorphism for every affinoid perfectoid $\cal V\in X_\sub{pro\'et}$; this is what we shall now do. Firstly, it is easily seen to be an almost quasi-isomorphism by Lemma \ref{lemma_local_structure_of_sheaves}, and so in particular the cone is derived $p$-adically complete; since the codomain is derived $p$-adically complete, therefore the domain is also; but the codomain is precisely the derived $p$-adic completion of the domain, and hence the map is a quasi-isomorphism.
}
over $s$ yields an almost quasi-isomorphism \[\boxed{R\Gamma_\sub{cont}(\bb Z_p(1)^d,R_\infty)\xTo{\sub{al.~qu.-iso.}} R\Gamma_\sub{pro\'et}(U,\hat\roi_X^+).}\] 

Arguing by induction and taking inverse limits, these almost descriptions may be extended to the other sheaves in Lemma \ref{lemma_local_structure_of_sheaves}, giving in particular almost (wrt.~$W_r(\frak m)$ and $W(\frak m^\flat)$ respectively) quasi-isomorphisms
\[\boxed{R\Gamma_\sub{cont}(\bb Z_p(1)^d,W_r(R_\infty))\xTo{\sub{al.~qu.-iso.}} R\Gamma_\sub{pro\'et}(U,W_r(\hat\roi_X^+))}\] and
\[\boxed{R\Gamma_\sub{cont}(\bb Z_p(1)^d,W(R_\infty^\flat))\xTo{\sub{al.~qu.-iso.}} R\Gamma_\sub{pro\'et}(U,\bb A_{\sub{inf},X}).}\] These ``Cartan--Leray almost quasi-isomorphisms'' are crucial to all our calculations of pro-\'etale cohomology.

\section{The main construction and theorems}\label{section_main}
In this section we present the main construction and define the new cohomology theory introduced in [BMS], before proving that its main properties, as stated in Theorem \ref{thm:IntroGlobalCompThm}, can be reduced to a certain $p$-adic Cartier isomorphism. We work in the set-up of  \S\ref{subsection_intro} throughout:
\begin{itemize}\itemsep0pt
\item $\bb C$ is a complete, non-archimedean, algebraically closed field of mixed characteristic; ring of integers~$\roi$ with maximal ideal $\frak m$; residue field $k$.
\item We pick a compatible sequence $\zeta_p,\zeta_{p^2},\dots\in\roi$ of $p$-power roots of unity, and define $\mu,\xi,\xi_r,\tilde\xi,\tilde\xi_r\in\bb A_\sub{inf}=W(\roi^\flat)$ as in \S\ref{subsection_roots_of_unity}.
\item $\frak X$ is a smooth formal scheme over $\roi$, which we do not yet assume is proper; its generic fibre, as a rigid analytic variety over $\bb C$, is denoted by $X=\frak X_C$.
\item $\nu:X_\sub{pro\'et}\to \frak X_\sub{Zar}$ is the projection map of sites obtained by pulling back any Zariski open in $\frak X_\sub{Zar}$ to the constant tower in $X_\sub{pro\'et}$ consisting of its generic fibre. That is, $\nu$ is the composition of maps of sites $X_\sub{pro\'et}\to X_\sub{\'et}\to \frak X_\sub{\'et}\to \frak X_\sub{Zar}$, where the first projection map is what was previously denoted by $\nu$ in \S\ref{subsection_pro_etale}.\footnote{We hope that this rechristening of $\nu$ does not lead to confusion, but we are following the (incompatible) notations of \cite{Scholze2013} and [BMS].}
\end{itemize}
The following is the fundamental new object at the heart of our cohomology theory:

\begin{definition}
Applying $\nu:X_\sub{pro\'et}\to\frak X_\sub{Zar}$ to the period sheaf $\bb A_{\sub{inf},X}$ gives a ``nearby cycles period sheaf'' $R\nu_*\bb A_{\sub{inf},X}$, which is a complex of sheaves of $\bb A_\sub{inf}$-modules on $\frak X_\sub{Zar}$; to this we now apply $L\eta_\mu$ to obtain a complex of sheaves of $\bb A_\sub{inf}$-modules on $\frak X_\sub{Zar}$:
\[\bb A\Omega_{\frak X/\roi}:=L\eta_\mu(R\nu_*\bb A_{\sub{inf},X}).\] We will soon equip $\bb A\Omega_{\frak X/\roi}$ with a Frobenius-semi-linear endomorphism $\phi$.
\end{definition}

\begin{remark}\label{remark_Leta_on_sheaves}
The previous definition used the d\'ecalage functor for a complex of sheaves, whereas we only defined it in Definition \ref{definition_decalage} for complexes of modules; here we explain the necessary minor modifications.

Let $\cal T$ be a topos, $A$ a ring, and $f\in A$ a non-zero-divisor. Call a complex $C$ of sheaves of $A$-modules {\em strongly $K$-flat} if and only if 
\begin{itemize}\itemsep0pt
\item $C^i$ is a sheaf of flat $A$-modules for all $i\in\bb Z$,
\item and the direct sum totalisation of the bicomplex $C\otimes_AD$ is acyclic for every acyclic complex $D$ of sheaves of $A$-modules.\footnote{This is not automatic from the first condition since $C$ may be unbounded, and is a standard condition to impose when requiring flatness conditions on unbounded complexes of sheaves.}
\end{itemize}
For any such $C$ we define a new complex of sheaves $\eta_fC$ by \[\cal T\ni U\mapsto (\eta_fC)^i(U):=\{x\in f^iC^i(U):dx\in f^{i+1}C^{i+1}(U)\}.\] Any complex $D$ of sheaves of $A$-modules may be resolved by a strongly $K$-flat complex $C\quis D$ (e.g., see the proof of {\em The Stacks Project}, Tag 077J), and we define $L\eta_fD:=\eta_fC$. This is a well-defined endofunctor of the derived category of sheaves of $A$-modules on $\cal T$.

{\bf Warning:} The d\'ecalage functor does not commute with global sections: there is a natural ``global-to-local'' morphism \[L\eta_fR\Gamma(\cal T,C)\To R\Gamma(\cal T,L\eta_fC),\] but this is not in general a quasi-isomorphism.
\end{remark}

\begin{remark}\label{remark_vague_outline}
Before saying anything precise, we offer some vague descriptions of how $\bb A\Omega_{\frak X/\roi}$ looks and how it can be studied. Ignoring the d\'ecalage functor for the moment, $R\nu_*\bb A_{\sub{inf},X}$ is obtained by sheafifying $\frak X\supseteq \Spf R\mapsto R\Gamma_\sub{pro\'et}(\Sp R[\tfrac1p],\bb A_{\sub{inf},X})$, as $\Spf R$ runs over affine opens of $\frak X$. We may suppose here that $R$ is small and so put ourselves in the situation of \S\ref{subsection_CL}: $R$ is a small, formally smooth $\roi$-algebra corresponding to an affine open $\Spf R\subseteq\frak X$, with associated pro-\'etale cover $\cal U=\projlimf_i\Sp R_i[\tfrac1p]\to \Sp R[\tfrac1p]$, where $\cal U$ is affinoid perfectoid with associated perfectoid ring $R_\infty$. As we saw in \S\ref{subsection_CL} there is an associated Cartan--Leray almost (wrt.~$W(\frak m^\flat)$) quasi-isomorphism \[R\Gamma_\sub{cont}(\bb Z_p(1)^d,W(R_\infty^\flat))\To R\Gamma_\sub{pro\'et}(\Sp R[\tfrac1p],\bb A_{\sub{inf},X}).\] Recalling from \S\ref{subsection_faltings_purity} that the d\'ecalage functor sometimes transforms almost quasi-isomorphisms into actual quasi-isomorphisms, $\bb A\Omega_{\frak X/\roi}$ can therefore be analysed locally through the complexes \[L\eta_\mu R\Gamma_\sub{cont}(\bb Z_p(1)^d,W(R_\infty^\flat)),\] as $\Spf R$ various over small affine opens of $\frak X$.\footnote{The astute reader may notice that in this argument we have just implicitly identified $L\eta_\mu R\Gamma_\sub{pro\'et}(\Sp R[\tfrac1p],\bb A_{\sub{inf},X})$ and $R\Gamma_\sub{Zar}(\Spf R,\bb A\Omega_{\frak X/\roi})$, contrary to the warning of Remark \ref{remark_Leta_on_sheaves}; this is precisely the type of technical obstacle which will need to be overcome in \S\ref{subsection_technical}.
} These complexes will turn out to be relatively explicit and related to de Rham--Witt complexes, Koszul complexes, and $q$-de Rham complexes.
\end{remark}

\begin{remark}[de Rham--Witt complexes]\label{remark_de_Rham_Witt}
Before continuing any further with Section \ref{section_main} the reader should probably first read \S\ref{subsection_definition_Witt}, where the relative de Rham--Witt complex $W_r\Omega_{\frak X/\roi}^\blob$ on $\frak X$ is defined; it provides an explicit complex computing both de Rham and crystalline cohomology.

In the subsequent \S\ref{subsection_constructing_Witt}, which the reader can ignore for the moment, we will explain methods of constructing ``Witt complexes'' over perfectoid rings. In particular, given a commutative algebra object $D\in D(\bb A_\sub{inf})$ equipped with a Frobenius-semi-linear automorphism $\phi_D$ and satisfying certain hypotheses, we will equip the cohomology groups \[\cal W_r^\blob(D):=H^\blob((L\eta_\mu D)/\tilde\xi_r),\quad\text{where}\quad (L\eta_\mu D)/\tilde\xi_r=(L\eta_\mu D)\dotimes_{\bb A_\sub{inf},\tilde\theta_r}W_r(\roi),\] with the structure of a Witt complex for $\roi\to R$ (where $R$ is an $\roi$-algebra depending on $D$); the differential $d:\cal W_r^\blob(D)\to\cal W_r^{\blob+1}(D)$ will be given by the Bockstein $\op{Bock}_{\tilde\xi_r}$. 
\end{remark}

To explain the main theorems we recall from \S\ref{subsection_theta} that there are two ways of specialising from $\bb A_\sub{inf}$ to $W_r(\roi)$
\[\xymatrix{
&\bb A_\sub{inf}\ar@{->>}[dl]_{\theta_r}\ar@{->>}[dr]^{\tilde\theta_r=\theta_r\circ\phi^{-r}}&\\
W_r(\roi)=\bb A_\sub{inf}/\xi_r\bb A_\sub{inf}&&W_r(\roi)=\bb A_\sub{inf}/\tilde\xi_r\bb A_\sub{inf},
}\]
so we use these to form corresponding specialisations of the complex of sheaves of $\bb A_\sub{inf}$-modules $\bb A\Omega_{\frak X/\roi}$:
\[\xymatrix{
&\bb A\Omega_{\frak X/\roi}\ar@{->>}[dl]_{\theta_r}\ar@{->>}[dr]^{\tilde\theta_r=\theta_r\circ\phi^{-r}}&\\
\bb A\Omega_{\frak X/\roi}/\xi_r=\bb A\Omega_{\frak X/\roi}\dotimes_{\bb A_\sub{inf},\theta_r}W_r(\roi)&&\bb A\Omega_{\frak X/\roi}\dotimes_{\bb A_\sub{inf},\tilde\theta_r}W_r(\roi)=\bb A\Omega_{\frak X/\roi}/\tilde\xi_r=:\tilde{W_r\Omega}_{\frak X/\roi}
}\]
The next theorem is the main new calculation at the heart of our results (and is the reason for the chosen notation $\tilde{W_r\Omega}_{\frak X/\roi}$ on the right of the previous line), from which we will deduce all further results, in which $W_r\Omega_{\frak X/\roi}^\blob$ is the relative de Rham--Witt complex of $\frak X$ over $\roi$:

\begin{theorem}[$p$-adic Cartier isomorphism]\label{theorem_p_adic_Cartier}
There are natural\footnote{As written, this isomorphism is natural but not canonical: it depends on the chosen sequence of $p$-power roots of unity. To make it independent of any choices, the right side should be replaced by $\cal H^i(\tilde{W_r\Omega}_{\frak X/\roi})\otimes_{W_r(\roi)}(\ker\tilde\theta_r/(\ker\tilde\theta_r)^2)^{\otimes i}$. Here  $\ker\tilde\theta_r/(\ker\tilde\theta_r)^2=\tilde\xi_r\bb A_\sub{inf}/\tilde\xi_r^2\bb A_\sub{inf}$ is a certain canonical rank-one free $W_r(\roi)$-module, and so we are replacing the right side by a type of Tate twist $\cal H^i(\tilde{W_r\Omega}_{\frak X/\roi})\{i\}$. This dependence arrises as follows: changing the chosen sequence of $p$-power roots of unity changes $\mu$ up to a unit in $\bb A_\sub{inf}$: this does not affect $L\eta_\mu$ (which depends only on the ideal generated by $\mu$), but does affect the forthcoming isomorphism in Remark \ref{remark_more_on_Leta}(a) (see footnote \ref{footnote_Bockstein}).
} isomorphisms of Zariski sheaves of $W_r(\roi)=\bb A_\sub{inf}/\tilde\xi_r\bb A_\sub{inf}$-modules \[C^{-r}_\frak X:W_r\Omega_{\frak X/\roi}^i\Isoto\cal H^i(\tilde{W_r\Omega}_{\frak X/\roi})\] for all $i\ge0$, $r\ge1$, which satisfy the following compatibilities:
\begin{enumerate}\itemsep0pt
\item the restriction map $R:W_{r+1}\Omega_{\frak X/\roi}^i\to W_r\Omega_{\frak X/\roi}^i$ is compatible with the map $\tilde{W_{r+1}\Omega}_{\frak X/\roi}\to \tilde{W_r\Omega}_{\frak X/\roi}$ induced by the inverse Frobenius on $\bb A_{\sub{inf},X}$.
\item the de Rham--Witt differential $d:W_r\Omega_{\frak X/\roi}^i\to W_r\Omega_{\frak X/\roi}^{i+1}$ is compatible with the Bockstein homomorphism $\op{Bock}_{\tilde\xi_r}:\cal H^i(\tilde{W_r\Omega}_{\frak X/\roi})\to\cal H^{i+1}(\tilde{W_r\Omega}_{\frak X/\roi})$.
\end{enumerate}
\end{theorem}
\begin{proof}[Idea of forthcoming proof]
Using the construction of \S\ref{subsection_constructing_Witt} (summarised in the previous remark), we will equip the sections $\cal H^\blob(\tilde{W_r\Omega}_{\frak X/\roi})(\Spf R)$ with the structure of a Witt complex for $\roi\to R$, naturally as $\Spf R$ varies over all small affine opens of $\frak X$, in \S\ref{subsection_reduction_to_tor}. This will give rise to universal (hence natural) morphisms of Witt complexes $W_r\Omega_{R/\roi}^\blob\to \cal H^\blob(\tilde{W_r\Omega}_{\frak X/\roi})(\Spf R)$ which satisfy (i) and (ii) and which will be explicitly checked to be isomorphisms (after $p$-adically completing $W_r\Omega_{R/\roi}^\blob$) by reducing,  via the type of  argument sketched in Remark \ref{remark_vague_outline}, to  group cohomology calculations given in \S\ref{subsection_local_Cartier}. 
\end{proof}

\begin{theorem}[Relative de Rham--Witt comparison]\label{theorem_drw_comparison_main}
There are natural quasi-isomorphisms in the derived category of Zariski sheaves of $W_r(\roi)=\bb A_\sub{inf}/\xi_r\bb A_\sub{inf}$-modules \[W_r\Omega_{\frak X/\roi}^\blob\simeq \bb A\Omega_{\frak X/\roi}/\xi_r,\] for all $r\ge 1$, such that the restriction map $R:W_{r+1}\Omega_{\frak X/\roi}^\blob\to W_r\Omega_{\frak X/\roi}^\blob$ is compatible with the canonical quotient map $\bb A_\sub{inf}/\xi_{r+1}\bb A_\sub{inf}\to \bb A_\sub{inf}/\xi_r\bb A_\sub{inf}$.
\end{theorem}

In a moment we will equip $\bb A\Omega_{\frak X/\roi}$ with a Frobenius and check that Theorem \ref{theorem_p_adic_Cartier} implies Theorem~\ref{theorem_drw_comparison_main}, from which we will then deduce Theorem \ref{thm:IntroGlobalCompThm}; first we require some additional properties of the d\'ecalage functor:

\begin{remark}[Elementary properties of the d\'ecalage functor, I]\label{remark_more_on_Leta}
Let $A$ be a ring and $f\in A$ a non-zero-divisor.
\begin{enumerate}[(a)]
\item (Bockstein construction) One of the most important properties of the d\'ecalage functor is its relation to the Bockstein boundary map. Let $C$ be a complex of $f$-torsion-free $A$-modules. From the definition of $\eta_fC$ it is easy to see that if $f^ix\in(\eta_fC)^i$ is a arbitrary element, then $x\text{ mod }fC^i$ is a cocycle of the complex $C/fC$ (since $d(f^ix)$ is divisible by $f^{i+1}$), and so defines a class $\res x\in H^i(C/fC)$; this defines a map of $A$-modules \[(\eta_fC)^i\To H^i(C/fC),\qquad f^ix\mapsto \res x.\] Next, the Bockstein $\op{Bock}_f:H^\blob(C/fC)\to H^{\blob+1}(C/fC)$ gives the cohomology groups $H^\blob(C/fC)$ the structure of a complex of $A/fA$-modules, and we leave it to the reader as an important exercise to check that the map \[\eta_fC\To [H^\blob(C/fC),\op{Bock}_f],\] given in degree $i$ by the previous line, is actually one of complexes, i.e., that the differential on $\eta_fC$ is compatible with $\op{Bock}_f$. Even more, the reader should check that the induced map \[(\eta_fC)\otimes_AA/fA\To  [H^\blob(C/fC),\op{Bock}_f]\] is a quasi-isomorphism.

More generally, if $D$ is an arbitrary complex of $A$-modules, then this can be rewritten as a natural\footnote{\label{footnote_Bockstein}Continuing the theme of the previous footnote, the left side depends only on the ideal $fA$ while the right side currently depends on the chosen generator $f$; to make the construction and morphism independent of this choice, each cohomology group on the right should be replaced by the twist $H^*(C\dotimes_AA/fA)\otimes_{A/fA}(f^*A/f^{*+1}A)$.} quasi-isomorphism \[(L\eta_fD)\dotimes_AA/fA\quis [H^*(D\dotimes_AA/fA),\op{Bock}_f]\] of complexes of $A/fA$-modules.\footnote{Curiously, this shows that the complex $(L\eta_f D)\dotimes_AA/fA$, which a priori lives only in the derived category of $A/fA$-modules, has a natural representative by an actual complex.}
\item (Multiplicativity) If $g\in A$ is another non-zero-divisor, and $C$ is a complex of $fg$-torsion-free $A$-modules, then \[\eta_g\eta_fC=\eta_{fg}C\subseteq C[\tfrac1{gf}].\] Noting that $\eta_f$ preserves the property the $g$-torsion-freeness, there is no difficulty deriving to obtain a natural equivalence of endofunctors of $D(A)$ \[L\eta_g\circ L\eta_f\simeq L\eta_{gf}.\]
\item (Coconnective complexes) Let $D^{\ge 0}_{f\sub{-tf}}(A)$ be the full subcategory of $D(A)$ consisting of those complexes $D$ which have $H^i(D)=0$ for $i<0$ and $H^0(D)$ is $f$-torsion-free. Any such $D$ admits a quasi-isomorphic replacement $C\quis D$, where $C$ is a cochain complex of $f$-torsion-free $A$-modules supported in positive degree (e.g., if $D$ is bounded then pick a projective resolution $P\quis D$ and set $C:=\tau^{\ge 0}P$). Then \[L\eta_fD=\eta_fC\subseteq C\quis D,\] whence $L\eta_f$ restricts to an endofunctor of $D^{\ge 0}_{f\sub{-tf}}(A)$, and on this subcategory there is a natural transformation $j:L\eta_f\to\op{id}$. In fact, all our applications of the d\'ecalage functor take place in this subcategory.

\item (Functorial bound on torsion) We maintain the hypotheses of (c). Then the morphism $j:L\eta_f D\to D$ induces an isomorphism on $H^0$: indeed, \[H^0(L\eta_f D)=\ker((\eta_fC)^0\xto{d}(\eta_fC)^1)=\ker(C^0\xto{d}C^1)=H^0(D).\] More generally, for any $i\ge0$, the map $j:H^i(L\eta_fD)\to H^i(D)$ has kernel $H^i(L\eta_fD)[f^i]$ and image $f^iH^i(D)$: indeed, the composition \[H^i(D)/H^i(D)[f]\isoto H^i(L\eta_fD)\xto{j} H^i(D),\] where the first isomorphism is Lemma \ref{lemma_calculation_of_cohomol}, is easily seen to be multiplication by $f^i$, whence the assertion follows.

It may be useful to note that this $f$-power-torsion difference between $D$ and its d\'ecalage $L\eta_fD$ can be functorially captured in the derived category, at least after truncation. More precisely, multiplication by $f^i$ defines a map $\tau^{\le i}C\to\tau^{\le i}\eta_fC$, which induces a natural transformation of functors $``{f^i}":\tau^{\le i}\to\tau^{\le i}L\eta_f$ on $D^{\ge 0}_{f\sub{-tf}}(A)$ such that the compositions \[\tau^{\le i}\xto{``f^i"}\tau^{\le i}L\eta_f\xto{j}\tau^{\le i},\quad \tau^{\le i}L\eta_f\xto{j}\tau^{\le i}\xto{``f^i"}\tau^{\le i}L\eta_f\] are both multiplication by $f^i$.
\item (a)--(d) have obvious modifications for complexes of sheaves of $A$-modules on a site.
\end{enumerate}
\end{remark}

As promised, we will now equip $\bb A\Omega_{\frak X/\roi}$ with a Frobenius:

\begin{lemma}\label{lemma_frobenius_on_AOmega}
The complex of sheaves of $\bb A_\sub{inf}$-modules $\bb A\Omega_{\frak X/\roi}$ is equipped with a Frobenius-semi-linear endomorphism $\phi$ which becomes an isomorphism after inverting $\xi$, i.e., \[\phi:\bb A\Omega_{\frak X/\roi}\dotimes_{\bb A_\sub{inf}}\bb A_\sub{inf}[\tfrac1\xi]\quis \bb A\Omega_{\frak X/\roi}\dotimes_{\bb A_\sub{inf}}\bb A_\sub{inf}[\tfrac1{\tilde\xi}]\] (recall that $\tilde\xi=\phi(\xi)$).
\end{lemma}
\begin{proof}
The Frobenius automorphism $\phi$ on the period sheaf $\bb A_{\sub{inf},X}$ induces a Frobenius automorphism $\phi$ on its derived image $C:=R\nu_*\bb A_{\sub{inf},X}$, which by functoriality then induces a quasi-isomorphism of complexes of Zariski sheaves \[\phi:L\eta_\mu C \quis L\eta_{\phi(\mu)}C.\] We follow this map by \[L\eta_{\phi(\mu)}C=L\eta_{\tilde \xi}L\eta_\mu C\To L\eta_\mu C\] to ultimately define the desired Frobenius $\phi:L\eta_\mu C\to L\eta_\mu C$, where it remains to explain the previous line. The equality is a consequence of Remark \ref{remark_more_on_Leta}(b) of the previous remark since $\phi(\mu)=\tilde\xi\mu$; the arrow is a consequence of Remark \ref{remark_more_on_Leta}(c) since $\cal H^0(L\eta_\mu C)$ has no $\tilde\xi$-torsion.\footnote{{\em Proof.} $\cal H^0(C)=\nu_*\bb A_{\sub{inf},X}$ has no $\mu$-torsion since $\bb A_{\sub{inf},X}$ has no $\mu$-torsion by Corollary \ref{corollary_non_zero_divisors_on_sheaf}; thus $\cal H^0(L\eta_\mu C)\isoto\cal H^0(C)$ by Remark \ref{remark_more_on_Leta}(d). But since $\cal H^0(C)$ has no $\mu$-torsion, it also has no $\phi(\mu)=\tilde\xi\mu$-torsion, thus has no $\tilde\xi$-torsion. $\square$} Since the arrow becomes a quasi-isomorphism after inverting $\tilde\xi$, we see that the final Frobenius $\phi:L\eta_\mu C\to L\eta_\mu C$ becomes a quasi-isomorphism after inverting $\xi$.
\end{proof}

\begin{proof}[Proof that Theorem \ref{theorem_p_adic_Cartier} implies Theorem \ref{theorem_drw_comparison_main}]
As in the proof of the previous lemma we write $C:=R\nu_*\bb A_{\sub{inf,X}}$, which we equipped with a Frobenius-semi-linear automorphism $\phi$. Thus we have
\begin{align*}
W_r\Omega_{\frak X/\roi}^\blob&\stackrel{C^{-r}_\frak X}{\cong}[\cal H^\blob(\tilde{W_r\Omega}_{\frak X/\roi}),\op{Bock}_{\tilde\xi_r}]\tag{by Theorem \ref{theorem_p_adic_Cartier}}\\
&=[\cal H^\blob((L\eta_\mu C)/\tilde\xi_r),\op{Bock}_{\tilde\xi_r}]\tag{rewriting for clarify}\\
&\simeq (L\eta_{\tilde \xi_r}L\eta_\mu C)/\tilde\xi_r\tag{by the Bockstein--$L\eta$ relation, i.e., Rmk.~\ref{remark_more_on_Leta}(a)}\\
&=(L\eta_{\tilde\xi_r\mu}C)/\tilde\xi_r\tag{by  Rmk.~\ref{remark_more_on_Leta}(b)}\\
&\stackrel{\phi^{-r}}\quis (L\eta_\mu C)/\xi_r\tag{functoriality and $\phi^{-r}(\tilde\xi_r\mu)=\mu$},
\end{align*}
which proves Theorem \ref{theorem_drw_comparison_main}.
\end{proof}

\lb{
\begin{remark}
Since the relative de Rham--Witt complex computes crystalline cohomology, Theorem \ref{theorem_drw_comparison_main} may be rewritten as \[\bb A\Omega_{\frak X/\roi}\dotimes_{\bb A_\sub{inf}}\bb A_\sub{inf}/\xi_r\cong Ru_*\roi_{\frak X/}\]

Here we use the usual divided powers structure on the ideal $VW_{r-1}(\roi)$, defined by $\gamma_n(V(a)):=\tfrac{p^{n-1}}{n!}V(x^n)$.

We can actually prove the stronger result that \[\bb A\Omega_{\frak X/\roi}\hat\otimes_{\bb A_\sub{inf}}\bb A_\sub{crys}\cong Ru_*\roi_{\frak X/\bb A_\sub{crys}}\]
and we can add the following comparison theorem: $R\Gamma_{\bb A}(\frak X)\dotimes_{\bb A_\sub{inf}}\bb A_\sub{crys}\simeq R\Gamma_\sub{crys}(\frak X\otimes_\roi\roi/p/\bb A_\sub{crys})$.
\end{remark}
}

Now we deduce the beginning of Theorem \ref{thm:IntroGlobalCompThm} from Theorem \ref{theorem_drw_comparison_main}:

\begin{theorem}\label{theorem_main_with_proofs}
If $\frak X$ is moreover proper over $\roi$, then $R\Gamma_\bb A(\frak X):=R\Gamma_\sub{Zar}(\frak X,\bb A\Omega_{\frak X/\roi})$ is a perfect complex of $\bb A_\sub{inf}$-modules with the following specialisations, in which (i) and (ii) are compatible with the Frobenius actions:
\begin{enumerate}
\item \'Etale specialization: $R\Gamma_{\bb A}(\frak X) \dotimes_{\bb A_\sub{inf}} W(\bb C^\flat) \simeq R\Gamma_\sub{\' et}(X,\mathbb{Z}_p) \dotimes_{\mathbb{Z}_p} W(\bb C^\flat)$.
\item Crystalline specialization: $R\Gamma_{\bb A}(\frak X) \dotimes_{\bb A_\sub{inf}} W(k) \simeq R\Gamma_{\mathrm{crys}}(\mathfrak{X}_k/W(k))$.
\item de Rham specialization: $R\Gamma_{\bb A}(\frak X) \dotimes_{\bb A_\sub{inf},\theta} \mathcal{O} \simeq R\Gamma_{\mathrm{dR}}(\mathfrak{X}/\mathcal{O})$.
\end{enumerate}
\end{theorem}
\begin{proof}
We prove the specialisations in reverse order. Firstly, since $R\Gamma_\bb A(\frak X)$ is derived $\xi$-adically complete,\footnote{\label{footnote_complete}{\em ``Proof''}. If $f,g$ are non-zero-divisors of a ring $A$, and $D$ is complex of $A$-modules which is derived $g$-adically complete, then we claim that $L\eta_fD$ is still derived $g$-adically complete: indeed, this follows from the fact that a complex is derived $g$-adically complete if and only if all of its cohomology groups are derived $g$-adically complete, that $H^i(L\eta_f D)\cong H^i(D)/H^i(D)[f]$ for all $i\in\bb Z$ by Lemma \ref{lemma_calculation_of_cohomol}, and that kernels and cokernels of maps between derived $g$-adically complete modules are again derived $g$-adically complete. For a reference on such matters, see {\em The Stacks Project}, Tag 091N.

It is tempting to claim that the previous paragraph remains valid for the complex of sheaves $R\nu_*\bb A_{\sub{inf},X}$ (which is indeed derived $\xi$-adically complete, since $R\nu_*$ preserves the derived $\xi$-adic completeness of the pro-\'etale sheaf $\bb A_{\sub{inf},X}$), which would complete the proof since $R\Gamma_\sub{Zar}(\frak X,-)$ also preserves derived $\xi$-adic completeness, but unfortunately the previous paragraph does not remain valid for complexes of sheaves on a ``non-replete'' site (e.g., the Zariski site). In fact, it seems that the derived $\xi$-adic completeness of $R\Gamma_\bb A(\frak X)$ is not purely formal, and requires the technical lemmas established in \S\ref{subsection_technical}; therefore we have postponed a proof of the completeness to Corollary \ref{corollary_completeness}. $\square$
} general formalism implies that $R\Gamma_\bb A(\frak X)$ is a perfect complex of $\bb A_\sub{inf}$-modules if and only if $R\Gamma_\bb A(\frak X)\dotimes_{\bb A_\sub{inf}}\bb A_\sub{inf}/\xi\bb A_\sub{inf}$ is a perfect complex of $\bb A_\sub{inf}/\xi\bb A_\sub{inf}=\roi$-modules. But Theorem \ref{theorem_drw_comparison_main} in the case $r=1$ implies that \[R\Gamma_\bb A(\frak X)\dotimes_{\bb A_\sub{inf}}\bb A_\sub{inf}/\xi\bb A_\sub{inf}\simeq R\Gamma_\sub{Zar}(\frak X,\Omega_{\frak X/\roi}^\blob)=R\Gamma_\sub{dR}(\frak X/\roi),\] which is indeed a perfect complex.\footnote{{\em Proof}. By derived $p$-adically completeness, it is enough to check that $R\Gamma_{\sub{dR}}(\frak X/\roi)\dotimes_{\roi}\roi/p\roi=R\Gamma_\sub{dR}(\frak X\otimes_{\roi}\roi/p\roi/(\roi/p\roi))$ is a perfect complex of $\roi/p\roi$-modules; this follows from the facts that $\Omega_{\frak X\otimes_{\roi}\roi/p\roi/(\roi/p\roi)}^\blob$ is a perfect complex of $\roi_{\frak X\otimes_{\roi}\roi/p\roi}$-modules by smoothness, and that the structure map $\frak X\otimes_{\roi}\roi/p\roi\to\Spec\roi/p\roi$ is proper, flat, and of finite presentation. $\square$}

It follows that $R\Gamma_\bb A(\frak X)\dotimes_{\bb A_\sub{inf}}W(k)$ is a perfect complex of $W(k)$-modules; since $W(k)$ is $p$-adically complete, any perfect complex over it is derived $p$-adically complete and so
\begin{align*}
R\Gamma_\bb A(\frak X)\dotimes_{\bb A_\sub{inf}}W(k)&\quis \op{Rlim}_r(R\Gamma_\bb A(\frak X)\dotimes_{\bb A_\sub{inf}}W_r(k))\\
&=\op{Rlim}_r(R\Gamma_\bb A(\frak X)\dotimes_{\bb A_\sub{inf}}W_r(\roi)\dotimes_{W_r(\roi)}W_r(k))\\
&\quis \op{Rlim}_r(R\Gamma_\sub{Zar}(\frak X,W_r\Omega_{\frak X/\roi}^\blob\dotimes_{W_r(\roi)}W_r(k)))
\end{align*}
where the final line uses Theorem \ref{theorem_drw_comparison_main}. But the canonical base change map $W_r\Omega_{\frak X/\roi}^\blob\dotimes_{W_r(\roi)}W_r(k)\quis W_r\Omega_{\frak X_k/k}^\blob$ is a quasi-isomorphism for each $r\ge1$ by Remark \ref{remark_de_rham_witt}(vii), and so we deduce that \[R\Gamma_\bb A(\frak X)\dotimes_{\bb A_\sub{inf}}W(k)\quis\op{Rlim}_rR\Gamma_\sub{Zar}(\frak X_k,W_r\Omega_{\frak X_k/k}^\blob)=R\Gamma_\sub{crys}(\frak X_k/W(k)).\]

It remains to prove the \'etale specialisation; we prove the stronger (since $\mu$ becomes invertible in $W(\bb C^\flat)$) result that $R\Gamma_\bb A(\frak X)\dotimes_{\bb A_\sub{inf}}\bb A_\sub{inf}[\tfrac1\mu]\simeq R\Gamma_\sub{\' et}(X,\mathbb{Z}_p) \dotimes_{\mathbb{Z}_p}\bb A_\sub{inf}[\tfrac1\mu]$. Since $L\eta_\mu$ only effects complexes up to $\mu^i$-torsion in degree $i$ (to be precise, use the morphisms $``\mu^i"$ on the truncations of $\bb A\Omega_{\frak X/\roi}\to R\nu_*\bb A_{\sub{inf},X}$, as in Remark \ref{remark_more_on_Leta}), the kernel and cokernel of $H^i_\bb A(\frak X)\to H^i_\sub{Zar}(\frak X,R\nu_*\bb A_{\sub{inf},X})=H^i_\sub{pro\'et}(X,\bb A_{\sub{inf},X})$ are killed by $\mu^i$. The key to the \'etale specialisation is now the fact that the canonical map \[R\Gamma_\sub{\'et}(X,\bb Z_p)\dotimes_{\bb Z_p}\bb A_\sub{inf}\To R\Gamma_\sub{pro\'et}(X,\bb A_{\sub{inf},X})\] has cone killed by $W(\frak m^\flat)\ni\mu$ (this is deduced from Theorem \ref{theorem_global_cohomol} by taking a suitable limit; see \cite[Prf.~of Thm.~8.4]{Scholze2013}); inverting $\mu$ completes the proof.
\end{proof}

We next discuss the rest of Theorem \ref{thm:IntroGlobalCompThm} (using the same enumeration):

\begin{theorem}\label{theorem_main_with_proofs2}
Continuing to assume that $\frak X$ is a proper, smooth formal scheme over $\roi$, then the individual $\bb A_\sub{inf}$-modules $H^i_\bb A(\frak X):=H^i_\sub{Zar}(\frak X,\bb A\Omega_{\frak X/\roi})$ vanish for $i>2\dim\frak X$ and enjoy the following properties:
\begin{enumerate}\setcounter{enumi}{3}
\item $H^i_\bb A(\frak X)$ is a finitely presented $\bb A_\sub{inf}$-module;
\item $H^i_\bb A(\frak X)[\tfrac1p]$ is finite free over $\bb A_\sub{inf}[\tfrac1p]$;
\item $H^i_\bb A(\frak X)$ is equipped with a Frobenius semi-linear endomorphism $\phi$ which becomes an isomorphism after inverting $\xi$ (or any other preferred generator of $\ker\theta$), i.e., $\phi:H^i_\bb A(\frak X)[\tfrac1\xi]\isoto H^i_\bb A(\frak X)[\tfrac1{\tilde\xi}]$.
\item \'Etale: $H^i_\bb A(\frak X)[\tfrac1\mu]\cong H^i_\sub{\'et}(X,\bb Z_p)\otimes_{\bb Z_p}\bb A_\sub{inf}[\tfrac1\mu]$.
\item Crystalline: there is a short exact sequence \[0\To H^i_\bb A(\frak X)\otimes_{\bb A_\sub{inf}}W(k)\to H^i_\sub{crys}(\frak X_k/W(k))\To\Tor_1^{\bb A_\sub{inf}}(H^{i+1}_\bb A(\frak X), W(k))\To 0\]
\item de Rham: there is a short exact sequence \[0\To H^i_\bb A(\frak X)\otimes_{\bb A_\sub{inf},\theta}\roi\to H^i_\sub{dR}(\frak X/\roi)\To H^{i+1}_\bb A(\frak X)[\xi]\To 0\]
\item If $H^i_\sub{crys}(\frak X_k/W(k))$ or $H^i_\sub{crys}(\frak X/\roi)$ is torsion-free, then $H^i_\bb A(\frak X)$ is a finite free $\bb A_\sub{inf}$-module.
\end{enumerate}
\end{theorem}
\begin{proof}
The \'etale and de Rham specialisations, i.e., (vii) and (ix), are immediate from the derived specialisations proved in the previous theorem.


As mentioned at the start of the previous proof, the complex $R\Gamma_\bb A(\frak X)$ is derived $\xi$-adically complete; so to prove that its cohomology vanishes in degree $>2\dim\frak X$, it is enough to note that the same is true of $R\Gamma_\bb A(\frak X)\dotimes_{\bb A_\sub{inf}}\bb A_\sub{inf}/\xi\bb A_\sub{inf}\simeq R\Gamma_\sub{dR}(\frak X/\roi)$ (where we have applied the de Rham comparison of Theorem~\ref{theorem_main_with_proofs}).

(vi) follows from Lemma \ref{lemma_frobenius_on_AOmega} and, similarly to the \'etale specialisation in Theorem \ref{theorem_main_with_proofs}, one can give more precise bounds by observe that $\phi:\bb A\Omega_{\frak X/\roi}\to\bb A\Omega_{\frak X/\roi}$ is invertible on any truncation up to an application of the morphism $``\xi^i"$.

We now prove (iv) and (v) by a descending induction on $i$, noting that they are trivial when $i>2\dim\frak X$. By the inductive hypothesis we may suppose that all cohomology groups of $\tau^{>i}R\Gamma_\bb A(\frak X)$ are finitely presented and become free after inverting $p$, whence they are perfect $\bb A_\sub{inf}$-modules by Theorem~\ref{theorem_modules_over_A_inf}(ii). It follows that the complex of $\bb A_\sub{inf}$-modules $\tau^{>i}R\Gamma_\bb A(\frak X)$ is perfect, which combined with the perfectness of $R\Gamma_\bb A(\frak X)$   implies that $\tau^{\le i}R\Gamma_\bb A(\frak X)$ is also perfect. Thus its top degree cohomology group $H^i(\tau^{\le i}R\Gamma_\bb A(\frak X))=H^i_\bb A(\frak X)$ is the cokernel of a map between projective $\bb A_\sub{inf}$-modules, and so is finitely presented.

To prove (v) we wish to apply Corollary \ref{corollary_Fitting_ideal_trick}, and must therefore check that $H^i_\bb A(\frak X)[\tfrac1{p\mu}]$ is a finite free $\bb A_\sub{inf}[\tfrac1{p\mu}]$-module of the same rank as the $W(k)$-module $M\otimes_{\bb A_\sub{inf}}W(k)$. Part (vii) implies that \[H^i_\bb A(\frak X)[\tfrac1{p\mu}]\cong H^i_\sub{\'et}(X,\bb Q_p)\otimes_{\bb Q_p}\bb A_\sub{inf}[\tfrac1{p\mu}],\] which is finite free over $\bb A_\sub{inf}[\tfrac1{p\mu}]$, while the derived crystalline specialisation of Theorem \ref{theorem_main_with_proofs} implies that \[H^i_\bb A(\frak X)\otimes_{\bb A_\sub{inf}}W(k)[\tfrac1p]\isoto H^i_\sub{crys}(\frak X_k/W(k))[\tfrac1p].\] (There are no higher Tor obstructions since $H^*_\bb A(\frak X)[\tfrac1p]$ is finite free over $\bb A_\sub{inf}[\tfrac1p]$ by the inductive hypothesis for $*>i$.) Therefore we must check that the following equality of dimensions holds: \[\dim_{\bb Q_p}H^i_\sub{\'et}(X,\bb Q_p)=\dim_{W(k)[\tfrac1p]}H^i_\sub{crys}(\frak X_k/W(k))[\tfrac1p]\tag{dim$_\frak X$}.\] This can be proved in varying degrees of generality as follows:
\begin{itemize}\itemsep0pt
\item In the special case that $\frak X$ is obtained by base change from a smooth, proper scheme over the ring of integers of a discretely valued subfield of $\bb C$ (which is perhaps the main case of interest for most readers), then the equality (dim$_\frak X$) is classical (or a consequence of the known Crystalline Comparison Theorem): the crystalline cohomology (with $p$ inverted) of the special fibre identifies with the de Rham cohomology of the generic fibre, which has the same dimension as the $\bb Q_p$-\'etale cohomology by non-canonically embedding into the complex numbers and identifying de Rham cohomology with singular cohomology.
\item Slightly more generally, if $\frak X$ is obtained by base change from a smooth, proper {\em formal} scheme over the ring of integers of a discretely valued subfield of $\bb C$, then (dim$_\frak X$) follows from the rational Hodge--Tate decomposition \cite[Corol.~1.8]{Scholze2013} (which is an easy consequence of the results in the remainder of these notes) and the same identification of crystalline and de Rham cohomology as in the previous case.
\item In the full generality in which we are working (i.e., $\frak X$ is an arbitrary proper, smooth formal scheme over $\roi$), then the equality (dim$_\frak X$) follows from our general Crystalline Comparison Theorem \[H^i_\sub{crys}(\frak X_k/W(k))\otimes_{W(k)}\bb B_\sub{crys}\cong H^i_\sub{\'et}(\frak X_{\res K},\bb Z_p)\otimes_{\bb Z_p}\bb B_\sub{crys}\] (Prop.~13.9 and Thm.~14.5(i) of [BMS]), whose proof we do not cover in these notes.\footnote{Possibly (dim$_\frak X$) can be proved in this case by combining spreading-out arguments of Conrad--Gabber with the relative $p$-adic Hodge theory of \cite[\S8]{Scholze2013}, but we have not seriously considered the problem.}
\end{itemize}

Finally we must prove (viii) and (x): but (viii) follows from the derived form of the crystalline specialisation in Theorem \ref{theorem_main_with_proofs}, part (v), and Lemma \ref{lemma_Tors_over_A_inf}, while (x) follows by combining (viii) or (ix) with Corollary \ref{corollary_condition_for_freeness}.
\end{proof}

This completes the proof of Theorem \ref{theorem_main_with_proofs}, or rather reduces it to the $p$-adic Cartier isomorphism of Theorem \ref{theorem_p_adic_Cartier}. The remainder of these notes is devoted to sketching a proof of this $p$-adic Cartier isomorphism.

\section{Witt complexes}\label{section_constructing_Witt}
This section is devoted to the theory of Witt complexes. We begin by defining Witt complexes and Langer--Zink's relative de Rham--Witt complex, and then in \S\ref{subsection_constructing_Witt} present one of our main constructions: namely equipping certain cohomology groups with the structure of a Witt complex over a perfectoid ring. We apply this construction in \S\ref{subsection_local_Cartier} to the group cohomology of a Laurent polynomial algebra and prove that the result is precisely the relative de Rham--Witt complex itself; this is the key local result from which the $p$-adic Cartier isomorphism will then be deduced in Section \ref{section_Cartier}.

\subsection{Langer--Zink's relative de Rham--Witt complex}\label{subsection_definition_Witt}
We recall the notion of a Witt complex, or $F$-$V$-procomplex, from the work of Langer--Zink \cite{LangerZink2004}.

\begin{definition}\label{definition_Witt}
Let $A\to B$ be a morphism of $\bb Z_{(p)}$-algebras. An associated relative {\em Witt complex}, or {\em $F$-$V$-procomplex}, consists of the following data $(\cal W_r^\blob,R,F,V,\lambda_r)$:
\begin{enumerate}\itemsep0pt
\item a commutative differential graded $W_r(A)$-algebra $\cal W_r^\blob=\bigoplus_{n\ge0} \cal W_r^n$ for each integer $r\ge 1$;
\item morphisms $R:\cal W_{r+1}^\blob\to R_*\cal W_r^\blob$ of differential graded $W_{r+1}(A)$-algebras for $r\ge 1$;
\item morphisms $F:\cal W_{r+1}^\blob\to F_*\cal W_r^\blob$ of graded $W_{r+1}(A)$-algebras for $r\ge 1$;
\item morphisms $V:F_*\cal W_r^\blob\to\cal W_{r+1}^\blob$ of graded $W_{r+1}(A)$-modules for $r\ge 1$;
\item morphisms of $W_r(A)$-algebras $\lambda_r: W_r(B)\to \cal W_r^0$ for each $r\geq 1$ which commute with $R$, $F$, $V$.
\end{enumerate}
such that the following identities hold:
\begin{itemize}\itemsep0pt
\item $R$ commutes with both $F$ and $V$;
\item $FV=p$;
\item $FdV=d$;
\item the Teichm\"uller identity:\footnote{\label{footnote_Teichmuller}The Teichm\"uller identity follows from the other axioms if $\cal W_r^1$ is $p$-torsion-free: \[\hspace{-10mm}p\lambda_r([b])^{p-1}d\lambda_r([b])=d\lambda_r([b]^p)=dF\lambda_r([b])=FdVF\lambda_r([b])=Fd(\lambda_r([b])V(1))=F(V(1))d\lambda_r([b])=pFd\lambda_r([b]).\]} $Fd\lambda_{r+1}([b])=\lambda_r([b])^{p-1}d\lambda_r([b])$ for $b\in B$, $r\ge 1$.
\end{itemize}
\end{definition}

\begin{example}
If $k$ is a perfect field of characteristic $p$ and $R$ is a smooth $k$-algebra (or, in fact, any $k$-algebra, but it is the smooth case that was studied most classically), then the classical de Rham--Witt complex $W_r\Omega^\blob_{R/k}$ of Bloch--Deligne--Illusie, together with its operators $R,F,V$ and the identification $\lambda_r:W_r(R)=W_r\Omega_{R/k}^0$, is a Witt complex for $k\to R$.
\end{example}

There is an obvious definition of morphism between Witt complexes. In particular, it makes sense to ask for an initial object in the category of all Witt complexes for $A\to B$:

\begin{theorem}[Langer--Zink, 2004]\label{thm:dRWExists} There is an initial object $(W_r\Omega_{B/A}^\blob,R,F,V,\lambda_r)$ in the category of Witt complexes for $A\to B$, called the {\em relative de Rham--Witt complex}. (And this agrees with $W_r\Omega_{R/k}^\blob$ of the previous example when $A=k$ and $B=R$).
\end{theorem}

\begin{remark}\label{remark_de_rham_witt}
\begin{enumerate}
\item The reason for the ``relative'' in the definition is that there has been considerable work recently, mostly by Hesselholt, on the {\em absolute de Rham--Witt complex} $W_r\Omega_{B}^\blob``{=W_r\Omega_{B/\bb F_1}^\blob}$''.
\item Given a Witt complex for $A\to B$, each $\cal W_r^\blob$ is in particular a commutative differential graded $W_r(A)$-algebra whose degree zero summand is a $W_r(B)$-algebra (via the structure maps $\lambda_r$). There are therefore natural maps of differential graded $W_r(A)$-algebras $\Omega_{W_r(B)/W_r(A)}^\blob\to \cal W_r^\blob$ for all $r\ge1$ (which are compatible with the restriction maps on each side).

In the case of the relative de Rham--Witt complex itself, each map $\Omega_{W_r(B)/W_r(A)}^\blob\to W_r\Omega_{B/A}^\blob$ is surjective (indeed, the elementary construction of $W_r\Omega_{B/A}^\blob$ is to mod out $\Omega_{W_r(B)/W_r(A)}^\blob$ by the required relations so that the axioms of a Witt complex are satisfied) and is even an isomorphism when $r=1$, i.e., $\Omega_{B/A}^\blob\isoto W_1\Omega_{B/A}^\blob$.
\item If $B$ is smooth over $A$, and $p$ is nilpotent in $A$, then Langer--Zink construct natural comparison quasi-isomorphisms $R\Gamma_\sub{crys}(B/W_r(A))\quis W_r\Omega_{B/A}^\blob$, where the left side is crystalline cohomology with respect to the usual pd-structure on the ideal $VW_{r-1}(A)\subseteq W_r(A)$ (note that the quotient $W_r(A)/VW_{r-1}(A)$ is $A$) defined by the rule $\gamma_n(V(\al)):=\tfrac{p^{n-1}}{n!}V(\al^n)$. This is a generalisation of Illusie's classical comparison quasi-isomorphism $R\Gamma_\sub{crys}(R/W_r(k))\quis W_r\Omega_{R/k}^\blob$.
\item Langer--Zink's proof of the comparison quasi-isomorphism in (iii) uses an explicit description of $W_r\Omega_{B/A}^\blob$ in the case that $B=A[T_1,\dots,T_d]$; in [BMS, \S10.4] we extend their description to $B=A[T_1^{\pm1},\dots,T_d^{\pm1}]$.
\item If $B\to B'$ is an \'etale morphism of $A$-algebras, then $W_r(B)\to W_r(B')$ is known to be \'etale and it can be shown that $W_r\Omega_{B/A}^n\otimes_{W_r(B)}W_r(B')\isoto W_r\Omega_{B'/A}^n$ [BMS, Lem.~10.8]. From these and similar base change results one sees that if $Y$ is any $A$-scheme, then there is a well-defined Zariski (or even \'etale) sheaf $W_r\Omega_{Y/A}^n$ on $Y$ whose sections on any $\Spec B$ are $W_r\Omega_{B/A}^n$.
\item If now $\frak X$ is a $p$-adic formal scheme over $A$, then there is similarly a well-defined Zariski (or \'etale) sheaf $W_r\Omega_{\frak X/A}^n$ whose sections on any $\Spf B$ are the following (identical\footnote{For the elementary proof that the three completions are the same, see Lem.~10.3 and Corol.~10.10 of [BMS].}) $p$-adically complete $W_r(B)$-modules \[(W_r\Omega_{B/A}^n)_p^\comp\qquad (W_r\Omega_{B/A}^n)_{[p]}^\comp\qquad \projlim_sW_r\Omega_{(B/p^sB)/(A/p^sA)}^n\]
\item (Base change) In [BMS, Prop.~10.14] we establish the following important base change property: if $A\to A'$ is a homomorphism between perfectoid rings, and $R$ is a smooth $A$-algebra, then the canonical base change map $W_r\Omega_{R/A}^n\otimes_{W_r(A)}W_r(A')\to W_r\Omega_{R\otimes_AA'/A'}^n$ is an isomorphism; moreover, the $W_r(A)$-modules $W_r\Omega_{R/A}^n$ and $W_r(A')$ are Tor-independent, whence $W_r\Omega_{R/A}^\blob\dotimes_{W_r(A)}W_r(A')\quis W_r\Omega_{R\otimes_AA'/A'}^\blob$.
\end{enumerate}
\end{remark}

In conclusion, in the set-up of Section \ref{section_main}, the relative de Rham--Witt complex $W_r\Omega_{\frak X/\roi}^\blob$ is an explicit complex computing both de Rham and crystalline cohomologies.

\subsection{Constructing Witt complexes}\label{subsection_constructing_Witt}
From now until the end of Section \ref{section_constructing_Witt} we fix the following:
\begin{itemize}
\item $A$ is a perfectoid ring of the type discussed in \S\ref{subsection_roots_of_unity}, i.e., $p$-torsion-free and containing a compatible system $\zeta_p,\zeta_{p^2},\dots$ of primitive $p$-power roots of unity (which we fix); let $\ep\in A^\flat$ and $\mu,\xi,\xi_r,\tilde\xi_r\in W(A^\flat)$ be the elements constructed there.
\item $D$ is a coconnective (i.e., $H^*(D)=0$ for $*<0$), commutative algebra object\footnote{By this we mean that $D$ is a commutative algebra object in the category $D(W(A^\flat))$ in the most naive way: the constructions can be upgraded to the level of $\bb E_\infty$-algebras, but again this is not necessary for our existing results.} in $D(W(A^\flat))$ which is equipped with a $\phi$-semi-linear quasi-isomorphism $\phi_D:D\quis D$ (of algebra objects), and is assumed to satisfy the following hypothesis:
\begin{quote}
($\cal W1$) $H^0(D)$ is $\mu$-torsion-free.
\end{quote}
\end{itemize}
Here we will explain how to functorially construct, from the data $D,\phi_D$, certain Witt complexes over $A$: this will lead to universal maps from de Rham--Witt complexes to cohomology groups of $D$, which will eventually provide the maps in the $p$-adic Cartier isomorphism.

\begin{example}\label{example_D}
The main examples are $A=\roi$ with the following coconnective, commutative algebra objects over $\bb A_\sub{inf}=W(\roi^\flat)$, which will be studied in \S\ref{subsection_local_Cartier} and \S\ref{subsection_reduction_to_tor} respectively:
\begin{enumerate}
\item $R\Gamma(\bb Z^d, W(A^\flat)[U_1^{\pm 1/p^\infty},\cdots,U_d^{\pm 1/p^\infty}])$, or a $p$-adically complete version thereof.
\item $R\Gamma_\sub{pro\'et}(\Sp R[\tfrac1p],\bb A_{\sub{inf},X})$, where $\Spf R$ is a small affine open of a formally smooth $\roi$-scheme with generic fibre $X$.
\end{enumerate}
\end{example} 

We first explain our preliminary construction of a Witt complex from the data $D,\phi_D$, which will then be refined. In this construction, indeed throughout the rest of the section, it is important to recall from Section \ref{section_perfectoid} the isomorphisms $\tilde\theta_r:W(A^\flat)/\tilde\xi_r\isoto W_r(A)$, which we often implicitly view as an identification. In particular, for each $r\ge1$, we may form the coconnective\footnote{From assumption ($\cal W1$) and the existence of $\phi_D$, it follows that $H^0(D)$ has no $\phi^r(\mu)=\tilde\xi_r\mu$-torsion, hence no $\tilde\xi_r$-torsion; so $D/\tilde\xi_r$ is still coconnective.} derived algebra object \[D/\tilde\xi_r:=D\dotimes_{W(A^\flat)} W_r(A^\flat)/\tilde\xi_r=D\dotimes_{W(A^\flat),\tilde\theta_r} W_r(A)\] over $W(A^\flat)/\tilde\xi_r=W_r(A)$, and take its cohomology
\[
\cal W_r^\blob(D)_\sub{pre}:=H^\blob(D\dotimes_{W(A^\flat)}W(A^\flat)/\tilde\xi_r)
\]
to form a graded $W_r(A)$-algebra. Equipping these cohomology groups with the Bockstein differential $\op{Bock}_{\tilde\xi_r}:\cal W_r^n(D)_\sub{pre}\to\cal W_r^{n+1}(D)_\sub{pre}$ associated to the distinguished triangle
\[
D\dotimes_{W(A^\flat)} W(A^\flat)/\tilde\xi_r\xTo{\tilde\xi_r}D\dotimes_{W(A^\flat)}W(A^\flat)/\tilde\xi_r^2\To D\dotimes_{W(A^\flat)}W(A^\flat)/\tilde\xi_r
\]
makes $\cal W_r^\blob(D)_\sub{pre}$ into a differential graded $W_r(A)$-algebra.

Next let
\[\begin{aligned}
R'&: \cal W_{r+1}^\blob(D)_\sub{pre}\to\cal W_r^\blob(D)_\sub{pre}\\
F&: \cal W_{r+1}^\blob(D)_\sub{pre}\to\cal W_r^\blob(D)_\sub{pre}\\
V&: \cal W_r^\blob(D)_\sub{pre}\to\cal W_{r+1}^\blob(D)_\sub{pre}
\end{aligned}\]
be the maps on cohomology induced respectively by
\[\begin{aligned}
D\dotimes_{W(A^\flat)}W(A^\flat)/\tilde\xi_{r+1}&\xTo{\phi_D^{-1}\otimes\phi^{-1}} D\dotimes_{W(A^\flat)}W(A^\flat)/\tilde\xi_r\\
D\dotimes_{W(A^\flat)}W(A^\flat)/\tilde\xi_{r+1}&\xTo{\op{id}\otimes\mathrm{can.\ proj.}} D\dotimes_{W(A^\flat)}W(A^\flat)/\tilde\xi_r\\
D\dotimes_{W(A^\flat)}W(A^\flat)/\tilde\xi_r&\xTo{\op{id}\otimes\varphi^{r+1}(\xi)} D\dotimes_{W(A^\flat)}W(A^\flat)/\tilde\xi_{r+1},
\end{aligned}\]
which are compatible with the usual Witt vector maps $R,F,V$ on $W_r(A)=W(A^\flat)/\tilde \xi_r$ thanks to the second set of diagrams in Lemma \ref{lemma_theta_r_diagrams}.

As we will see in the proof of part (ii) of the next result, $R'$ must be replaced by\footnote{The reader should use the identities of \S\ref{subsection_roots_of_unity} to calculate that $\tilde\theta_r(\xi)=\tfrac{[\zeta_{p^r}]-1}{[\zeta_{p^{r+1}}]-1}\in W_r(A)$.}
\[
R:=\tilde\theta_r(\xi)^n R': \cal W_{r+1}^n(D)_\sub{pre}\to\cal W_r^n(D)_\sub{pre}\]
if we are to satisfy the axioms of a Witt complex.

\begin{proposition}\label{proposition_first_construction}
The data $(\cal W_r^\blob(D)_\sub{pre},R,F,V)$ satisfies all those axioms appearing in the definition of a Witt complex (Def.~\ref{definition_Witt}) which only refer to $R,F,V$ (i.e., which do not involve the additional ring $B$ or the structure maps $\lambda_r$). More precisely:
\begin{enumerate}\itemsep0pt
\item $\cal W_r^\blob(D)_\sub{pre}$ is a commutative\footnote{\label{footnote_p=2}Unfortunately this is not strictly true: if $p=2$ then the condition that $x^2=0$ for $x\in\cal W_r^\sub{odd}(D)_\sub{pre}$ need not be true; but this will be fixed when we improve the construction.} differential graded $W_r(A)$-algebra for each $r\ge1$.
\item $R'$ is a homomorphism of graded rings, and $R$ is a homomorphism of differential graded rings;
\item $V$ is additive, commutes with $R'$ and $R$, and is $F$-inverse-semi-linear (i.e., $V(F(x)y)=xV(y)$);
\item $F$ is a homomorphism of graded rings and commutes with both $R'$ and $R$;
\item $FdV=d$;
\item $FV$ is multiplication by $p$.
\end{enumerate}
\end{proposition}

\begin{proof}
Part (i) is a formal consequence of $D$ being a commutative algebra object of $D(W(A^\flat))$.

(ii): $R'$ is a homomorphism of graded rings by functoriality; the same is true of $R$ since it is twisted by increasing powers of an element. Moreover, the commutativity of
\[\xymatrix{
0\ar[r] & W(A^\flat)/\tilde\xi_{r+1}\ar[r]^{\tilde\xi_{r+1}}\ar[d]_{\xi \varphi^{-1}} & W(A^\flat)/\tilde\xi_{r+1}^2\ar[r]\ar[d] & W(A^\flat)/\tilde\xi_{r+1}\ar[r]\ar[d] & 0\\
0\ar[r] & W(A^\flat)/\tilde\xi_r\ar[r]^{\tilde\xi_r} & W(A^\flat)/\tilde\xi_r^2\ar[r] & W(A^\flat)/\tilde\xi_r\ar[r] & 0
}\]
and functoriality of the resulting Bocksteins implies that
\[\xymatrix{
\cal W_{r+1}^n(D)_\sub{pre}\ar[r]^d\ar[d]_{R'}&\cal W_{r+1}^{n+1}(D)_\sub{pre}\ar[d]^{\tilde\theta_r(\xi)R'} \\
\cal W_r^n(D)_\sub{pre}\ar[r]_d&\cal W_r^{n+1}(D)_\sub{pre} \\
}\]
commutes; hence the definition of $R$ was exactly designed to arrange that it commute with $d$.

(iii): $V$ is clearly additive, and it commutes with $R'$ since it already did so before taking cohomology. Secondly, the $F$-inverse-semi-linearity of $V$ follows by passing to cohomology in the following commutative diagram:
\[\xymatrix@C=2cm{
D/\tilde\xi_{r+1}\dotimes D/\tilde\xi_{r+1}\ar[r]^{\sub{mult}} & D/\tilde\xi_{r+1}\\
D/\tilde\xi_{r+1}\dotimes D/\tilde\xi_{r}\ar[u]^{\mathrm{id}\otimes \varphi^{r+1}(\xi)}\ar[d]_{\mathrm{can.\ proj.}\otimes\op{id}}&\\
D/\tilde\xi_{r}\dotimes D/\tilde\xi_{r}\ar[r]^{\sub{mult}} & D/\tilde\xi_r\ar[uu]_{\varphi^{r+1}(\xi)}
}\]
It now easily follows that $V$ also commutes with $R'$.

(iv): $F$ is a graded ring homomorphism, and it commutes with $R^\prime$ by definition, and then easily also with $R$.

(v): This follows by tensoring the commutative diagram below with $D$ over $W(A^\flat)$, and looking at the associated boundary maps on cohomology:

\[\xymatrix{
0\ar[r] & W(A^\flat)/\tilde\xi_r\ar[r]\ar@{=}[d] & W(A^\flat)/\tilde\xi_r^2\ar[r]\ar^{\varphi^{r+1}(\xi)}[d] & W(A^\flat)/\tilde\xi_r\ar^{\varphi^{r+1}(\xi)}[d]\ar[r] & 0\\
0\ar[r] & W(A^\flat)/\tilde\xi_r\ar[r] & W(A^\flat)/\tilde\xi_r \tilde\xi_{r+1} \ar[r] & W(A^\flat)/\tilde\xi_{r+1}\ar[r] & 0 \\
0\ar[r] & W(A^\flat)/\tilde\xi_{r+1}\ar[r]\ar^{\mathrm{can.\ proj.}}[u] & W(A^\flat)/\tilde\xi_{r+1}^2\ar[r]\ar[u] & W(A^\flat)/\tilde\xi_{r+1}\ar[r]\ar@{=}[u] & 0
}\]

(vi): This follows from the fact that $\tilde\theta_r(\varphi^{r+1}(\xi))=p$ for all $r\geq 1$ (which is true since $\theta_r(\phi(\xi))=\theta_r(\phi(\xi))=F(\theta_{r+1}(\xi))=FV(1)=p$, where the third equality uses the second diagram of Lemma \ref{lemma_theta_r_diagrams}).
\end{proof}

Unfortunately, there are various heuristic and precise reasons\footnote{For example, suppose that $B$ is an $A$-algebra and that we are given structure maps $\lambda_r:W_r(B)\to\cal W_r^0(D)$ under which $(\cal W_r^\blob(D),R,F,V,\lambda_r)$ becomes a Witt complex for $A\to B$, thereby resulting in a universal map of Witt complexes $\lambda_r^\blob:W_r\Omega_{B/A}^\blob\to W_r^\blob(D)$; then from the surjectivity of the restriction maps for $W_r\Omega^\blob_{B/A}$ and the definition of the restriction map $R$ for $\cal W_r^\blob(D)_\sub{pre}$, we see that
\[
\Im\lambda_r^n\subseteq \bigcap_{s\ge 1}\Im(\cal W_{r+s}^n(D)_\sub{pre}\xto{R^s}\cal W_r^n(D)_\sub{pre})\subseteq \bigcap_{s\ge 1}\tilde\theta_r(\xi_s)^n\cal W_r^n(D)_\sub{pre}=\bigcap_{s\ge 1}\big(\tfrac{[\zeta_{p^r}]-1}{[\zeta_{p^s}]-1}\big)^n\cal W_r^n(D)_\sub{pre}.
\]
The far right side contains, and often equals in realistic situations, $([\zeta_{p^r}]-1)^n\cal W_r^n(D)_\sub{pre}$, which motivates our replacement.}
that $\cal W_r^\blob(D)_\sub{pre}$ is ``too large'' to underlie an interesting Witt complex over $A$, and so we replace it by \[\cal W_r^n(D):=([\zeta_{p^r}]-1)^n\cal W_r^n(D)_\sub{pre}\subseteq\cal W_r^n(D)_\sub{pre}.\]

\begin{lemma}\label{lemma_improved}
The $W_r(A)$-submodules $\cal W_r^n(D)\subseteq \cal W_r^n(D)_\sub{pre}$ define sub differential graded algebras of $\cal W_r^\blob(D)_\sub{pre}$, for each $r\ge1$, which are closed under the maps $R,F,V$ (and hence Proposition \ref{proposition_first_construction} clearly remains valid for the data $(\cal W_r^\blob(D),R,F,V)$ ).
\end{lemma}
\begin{proof}
This is a consequence of the following simple identities, where $x\in \cal W_{r+1}^n(D)_\sub{pre}$ and $y\in \cal W_r^n(D)_\sub{pre}$:
\[R(([\zeta_{p^{r+1}}]-1)^nx)=\left(\tfrac{[\zeta_{p^{r+1}}]-1}{[\zeta_{p^r}]-1}\right)^n([\zeta_{p^{r+1}}]-1)^nR'(x)=([\zeta_{p^r}]-1)^nR'(x)\]
\[F(([\zeta_{p^{r+1}}]-1)^nx)=(F[\zeta_{p^{r+1}}]-1)^nF(x)=([\zeta_{p^r}]-1)^nF(x)\]
\[V(([\zeta_{p^r}]-1)^ny)=V(F([\zeta_{p^{r+1}}]-1)^ny)=([\zeta_{p^{r+1}}]-1)^nV(y)\]
Note that the first identity crucially used the definition of the restriction map $R$ as a multiple of $R'$.
\end{proof}

Next we relate the groups $\cal W_r^n(D)$ to the cohomology of the d\'ecalage $L\eta_\mu D$ of $D$. From the earlier assumption ($\cal W1$) and Remark \ref{remark_more_on_Leta}(c) there is a canonical map $L\eta_\mu D\to D$, and by imposing the following two additional assumptions on $D$ we will show in Lemma \ref{lemma_pre_Witt_image} that the resulting map on cohomology \[H^n(L\eta_\mu D\dotimes_{W(A^\flat)}W(A^\flat)/\tilde\xi_r)\To H^n(D\dotimes_{W(A^\flat)}W(A^\flat)/\tilde\xi_r)=\cal W_r^n(D)_\sub{pre}\] is injective and has image exactly $\cal W_r^n(D)$.

From now on we assume that $D$ satisfies the following assumptions (in addition to ($\cal W1$)):

\begin{quote}
($\cal W2$) The cohomology groups $H^*(L\eta_{\mu}D \dotimes_{W(A^\flat)}W(A^\flat)/\tilde \xi_r)$ are $p$-torsion-free for all $r\ge 0$.\\
($\cal W3$) The canonical base change map $L\eta_\mu D\dotimes_{W(A^\flat)}W(A^\flat)/\tilde\xi_r\to L\eta_{[\zeta_{p^r}]-1}(D\dotimes_{W(A^\flat)}W(A^\flat)/\tilde\xi_r)$ is a quasi-isomorphism for all $r\ge1$.
\end{quote}

\begin{remark}[Elementary properties of the d\'ecalage functor, II -- base change]\label{remark_base_change} We explain the base change map of assumption ($\cal W3$). If $\al:R\to S$ is a ring homomorphism, $f\in R$ is a non-zero-divisor whose image $\al(f)\in S$ is still a non-zero-divisor, and $C\in D(R)$, then there is a canonical base change map \[(L\eta_fC)\dotimes_RS\To L\eta_{\al(f)}(C\dotimes_RS)\] in $D(S)$ which the reader will construct without difficulty. This base change map is not a quasi-isomorphism in general,\footnote{On the other hand, if $C\in D(S)$ then the canonical restriction map $L\eta_f(C|_A)\to L\eta_{\al(f)}(C)|_D$ in $D(A)$, which the reader will also easily construct, is always a quasi-isomorphism.} but it is in the following cases:
\begin{enumerate}
\item When $R\to S$ is flat. {\em Proof}: Easy. $\square$
\item When $S=R/gR$ for some non-zero-divisor $g\in R$ (i.e., $f,g$ is a regular sequence in $R$) and the cohomology groups of $C\dotimes_RR/fR$ are assumed to be $g$-torsion-free.\footnote{This was erroneously asserted to be true in the official announcement without the $g$-torsion-freeness assumption.}
{\em Proof}: Since the base change map is always a quasi-isomorphism after inverting $f$, it is equivalent to establish the quasi-isomorphism after applying $-\dotimes_{R/gR}R/(f,g)$, after which the base change map becomes the canonical map
\[[H^\blob(C\dotimes_RR/fR),\op{Bock}_f]\dotimes_{R/gR}R/(f,g)\To [H^\blob(C\dotimes_RR/gR \dotimes_{R/gR}R/(f,g)),\op{Bock}_{f\sub{ mod }gR}]\]
by Remark \ref{remark_more_on_Leta}(a). But our assumption implies that the left tensor product $\dotimes_{R/gR}$ is equivalently underived, and that hence it is enough to check that the canonical map $H^n(C\dotimes_RR/fR)\otimes_{R/gR}R/(f,g)\to H^n(C\dotimes_RR/(f,g))$ is an isomorphism for all $n\in\bb Z$; but this is again true because of the $g$-torsion-freeness assumption. $\square$
\end{enumerate}

In the particular case of ($\cal W3$), we are base changing along $\tilde\theta_r:W(A^\flat)\to W(A^\flat)/\tilde\xi_r=W_r(A)$, noting that $\tilde\theta_r(\mu)=[\zeta_{p^r}]-1\in W_r(A)$ is a non-zero-divisor by Remark \ref{remark_non_primitive}. There is no a priori reason to expect hypothesis ($\cal W3$) to be satisfied in practice, but it will be in our cases of interest.\footnote{\label{footnote_condition_for_W3}Note in particular that ($\cal W3$) is satisfied if the cohomology groups of $D\dotimes_{W(A^\flat)}W(A^\flat)/\mu$ are $p$-torsion-free; this follows from Remark \ref{remark_base_change}(ii) since $\tilde\xi_r\equiv p^r$ mod $\mu W(A^\flat)$.}
\end{remark}

\begin{lemma}\label{lemma_pre_Witt_image}
The aforementioned map on cohomology \[H^n(L\eta_\mu D\dotimes_{W(A^\flat)}W(A^\flat)/\tilde\xi_r)\To H^n(D\dotimes_{W(A^\flat)}W(A^\flat)/\tilde\xi_r)=\cal W_r^n(D)_\sub{pre}\tag{\dag}\] is injective with image $\cal W_r^n(D)=([\zeta_{p^r}]-1)^n\cal W_r^n(D)_\sub{pre}$, for all $r\ge1$ and $n\ge0$.
\end{lemma}
\begin{proof}
The canonical map $L\eta_\mu D\to D$ induces maps on cohomology whose kernels and cokernels are killed by powers of $\mu$, by Remark \ref{remark_more_on_Leta}(d); hence the map (\dag) of $W_r(A)$-modules has kernel and cokernel killed by a power of $\tilde\theta_r(\mu)=[\zeta_{p^r}]-1$. But $[\zeta_{p^r}]-1$ divides $p^r$ by Remark \ref{remark_non_primitive}, so from assumption ($\cal W2)$ we deduce that map (\dag) is injective for every $r\ge1$ and $n\ge0$.

Regarding its image, simply note that (\dag) factors as \[H^n(L\eta_\mu D\dotimes_{W(A^\flat)}W(A^\flat)/\tilde\xi_r)\Isoto H^n(L\eta_{[\zeta_{p^r}]-1}(D\dotimes_{W(A^\flat)}W(A^\flat)/\tilde\xi_r))\To\cal W_r^n(D)_\sub{pre},\] where the first map is the base change isomorphism of assumption ($\cal W3$), and the second map has image $([\zeta_{p^r}]-1)^nW_r^n(D)_\sub{pre}$ by Remark \ref{remark_more_on_Leta}(d).
\end{proof}

We summarise our construction of Witt complexes by stating the following theorem:

\begin{theorem}\label{theorem_witt_main}
Let $A$ and $D,\phi_D$ be as at the start of \S\ref{subsection_constructing_Witt}, and assume that $D$ satisfies assumptions ($\cal W1$)--($\cal W3$). Suppose moreover that $B$ is an $A$-algebra equipped with $W_r(A)$-algebra homomorphisms $\lambda_r:W_r(B)\to H^0(D\dotimes_{W(A^\flat)}W(A^\flat)/\tilde\xi_r)$ making the following diagrams commute for all $r\ge1$:
\[
\xymatrix{
W_{r+1}(B) \ar[d]_R \ar[r]^{\lambda_{r+1}} & H^0(D/\tilde\xi_{r+1})\ar[d]^{\varphi_D^{-1}}\\
W_r(B) \ar[r]^{\lambda_r} & H^0(D/\tilde\xi_r)
}
\quad
\xymatrix{
W_{r+1}(B) \ar[d]_F \ar[r]^{\lambda_{r+1}} & H^0(D/\tilde\xi_{r+1})\ar[d]^{\mathrm{can.\ proj.}}\\
W_r(B) \ar[r]^{\lambda_r} & H^0(D/\tilde\xi_r)
}
\quad
\xymatrix{
W_{r+1}(B) \ar[r]^{\lambda_{r+1}} & H^0(D/\tilde\xi_{r+1})\\
W_r(B) \ar[u]^{V} \ar[r]^{\lambda_r} & H^0(D/\tilde\xi_r)\ar[u]_{\times\varphi^{r+1}(\xi)}
}
\]
Then the cohomology groups $\cal W_r^*(D)=H^*(L\eta_\mu D\dotimes_{W(A^\flat)}W(A^\flat)/\tilde\xi_r)$ may be equipped with the structure of a Witt complex for $A\to B$, and consequently there are associated universal maps of Witt complexes \[\lambda_r^\blob:W_r\Omega_{B/A}^\blob\To\cal W_r^\blob(D)\] (which are functorial with respect to $D,\phi_D$ and $B,\lambda_r$ in the obvious sense).
\comment{
More precisely,
\begin{itemize}
\item the Witt complex differential $\cal W_r^*(D)\to\cal W_r^{*+1}(D)$ is the Bockstein $\op{Bock}_{\tilde\xi_r}$;
\item and the Frobenius
\end{itemize}
}
\end{theorem}
\begin{proof}

Combining the hypotheses of the theorem with Lemma \ref{lemma_improved}, we see that $\cal W_r^*(D)$ satisfies all axioms for a Witt complex for $A\to B$, except perhaps for the following two: that $x^2=0$ for $x\in\cal W_r^\sub{odd}(D)_\sub{pre}$ when $p=2$; and the Teichm\"uller identity. But these follow from the other axioms since $\cal W_r^*(D)$ is assumed to be $p$-torsion-free.\footnote{$2x^2=0$ so $x^2=0$, c.f., footnote \ref{footnote_p=2}. For the Teichm\"uller identity see footnote \ref{footnote_Teichmuller}.}
\end{proof}

\begin{remark}[$p$-completions]\label{remark_p_completion}
In our cases of interest the complex $D$ will sometimes be derived $p$-adically complete, whence the complexes $L\eta_\mu D$ and $L\eta_\mu\dotimes_{W(A^\flat)}W(A^\flat)/\tilde\xi_r$ are also derived $p$-adically complete (by footnote  \ref{footnote_complete}); then each cohomology group $\cal W_r^n(D)$ is both $p$-torsion-free (by assumption ($\cal W2$)) and derived $p$-adically complete, hence $p$-adically complete in the underived sense. So, in this case, the associated universal maps $W_r\Omega_{B/A}^n\to\cal W_r^n(D)$ of the previous theorem factor through the $p$-adic completion $(W_r\Omega_{B/A}^n)_p^\comp$ which was discussed in Remark \ref{remark_de_rham_witt}(vi).
\end{remark}

\subsection{The de Rham--Witt complex of a torus as group cohomology}\label{subsection_local_Cartier}
We continue to let $A$ be a fixed perfectoid ring as at the start of  \S\ref{subsection_constructing_Witt}, and we fix $d\ge0$ and set \[D^\sub{grp}=D^\sub{grp}_{A,d}:=R\Gamma_\sub{grp}(\bb Z^d, W(A^\flat)[U_1^{\pm 1/p^\infty},\cdots,U_d^{\pm 1/p^\infty}]),\] where the $i^\sub{th}$-generator $\gamma_i\in\bb Z^d$ acts on $W(A^\flat)[\ul U^{\pm 1/p^\infty}]$ via the $W(A^\flat)$-algebra homomorphism \[\gamma_iU_j^k:=\begin{cases} [\ep^k]U_j^k& i=j \\ U_j^k & i\neq j\end{cases}\] (here $k\in\bb Z[\tfrac1p]$, and $\ep^k\in A^\flat$ is well-defined since $A^\flat$ is a perfect ring). Here we will apply the construction of \S\ref{subsection_constructing_Witt} to $D^\sub{grp}$ to build a Witt complex $\cal W_r^\blob(D^\sub{grp})$ for $A\to A[\ul T^{\pm1}]$, and show that the resulting universal maps $\lambda_r^\blob:W_r\Omega^\blob_{A[\ul T^{\pm1}]/A}\to \cal W_r^\blob(D^\sub{grp})$ are in fact isomorphisms. This is the key local result from which the $p$-adic Cartier isomorphism will be deduced in Section \ref{section_Cartier}.

In order to apply Theorem \ref{theorem_witt_main} to $D^\sub{grp}$ we must first check that all necessary hypotheses are fulfilled; we begin with the basic assumptions:

\begin{lemma}\label{lemma_Dgrp_1}
$D^\sub{grp}$ is a coconnective algebra object in $D(W(A^\flat))$ which is equipped with a $\phi$-semi-linear quasi-isomorphism $\phi_\sub{grp}:D^\sub{grp}\quis D^\sub{grp}$ and satisfies assumptions ($\cal W1$)--($\cal W3$). 
\end{lemma}
\begin{proof}
Certainly $D^\sub{grp}$ is a coconnective, commutative algebra object in $D(W(A^\flat))$, and it is equipped with a $\phi$-semi-linear quasi-isomorphism $\phi_\sub{grp}:D^\sub{grp}\quis D^\sub{grp}$ induced by the obvious Frobenius automorphism on $W(A^\flat)[\ul U^{\pm 1/p^\infty}]$ (acting on the coefficients as the Witt vector Frobenius $\phi$ and sending $U_i^k$ to $U_i^{pk}$ for all $k\in\bb Z[\tfrac1p]$ and $i=1,\dots,d$). Also, $H^0(D^\sub{grp})$ is $\mu$-torsion-free since $\mu$ is a non-zero-divisor of $W(A^\flat)$ by Proposition \ref{proposition_roots_of_unity}. Therefore $D^\sub{grp}$ satisfies the hypotheses from the start of \S\ref{subsection_constructing_Witt}, including ($\cal W1$).

Next we show that the cohomology groups of $L\eta_\mu D\dotimes_{W(A^\flat)}W(A^\flat)/\tilde\xi_r$ and $D\dotimes_{W(A^\flat)}W(A^\flat)/\mu$ are $p$-torsion-free, i.e., that hypotheses ($\cal W2$) and ($\cal W3$) (by footnote \ref{footnote_condition_for_W3}) are satisfied. This is a straightforward calculation of group cohomology in terms of Koszul complexes, in the same style as the proof of Theorem \ref{theorem_integral_purity}. Indeed, there is a $\bb Z^d$-equivariant decomposition of $W(A^\flat)$-modules \[W(A^\flat)[\ul U^{\pm1/p^\infty}]=\bigoplus_{k_1,\dots,k_d\in\bb Z[\tfrac1p]}W(A^\flat)U_1^{k_1}\cdots U_d^{k_d},\] where the generator $\gamma_i\in\bb Z^d$ acts on the rank-one free $W(A^\flat)$-module $W(A^\flat)U_1^{k_1}\cdots U_d^{k_d}$ as multiplication by $[\ep^{k_i}]$. By the standard group cohomology calculation of $R\Gamma_\sub{grp}(\bb Z^d,W(A^\flat)U_1^{k_1}\cdots U_d^{k_d})$ as a Koszul complex, this shows that \[R\Gamma_\sub{grp}(\bb Z^d,W(A^\flat)[\ul U^{\pm1/p^\infty}])\simeq \bigoplus_{k_1,\dots,k_d\in\bb Z[\tfrac1p]}K_{W(A^\flat)}([\ep^{k_1}]-1,\dots,[\ep^{k_d}]-1).\] It is now sufficient to show that the cohomology groups of $\eta_\mu K\otimes_{W(A^\flat)}W(A^\flat)/\tilde\xi_r$ and $K\otimes_{W(A^\flat)}W(A^\flat)/\mu$ are $p$-torsion-free, where $K$ runs over the Koszul complexes appearing in the sum.

Since it is important for the forthcoming cohomology calculations, we explicitly point out now that, if $k,k'\in \bb Z[\tfrac1p]$, then $[\ep^k]-1$ divides $[\ep^{k'}]-1$ if and only if $\nu_p(k)\le \nu_p(k')$.

We will first prove that the cohomology of $K\otimes_{W(A^\flat)}W(A^\flat)/\mu$ is $p$-torsion-free. Lemma \ref{lemma_on_Koszul_2} implies that there is an isomorphism of $W(A^\flat)$-modules $H^n(K)\cong W(A^\flat)/([\ep^k]-1)^{\binom{d-1}{n-1}}$, where $k=p^{-\min_{1\le i\le d}\nu_p(k_i)}$ and we have used that $[\ep^k]-1$ is a non-zero-divisor of $W(A^\flat)$ (so that the torsion term of that lemma vanishes). But $W(A^\flat)/([\ep^k]-1)$ is $p$-torsion-free since $p,[\ep^k]-1$ is a regular sequence\footnote{\label{footnote_reg_seq}{\em Proof}. We will show that $\ep^k-1$ is a non-zero-divisor of $A^\flat$. If $x\in A^\flat=\projlim_{x\mapsto x^p}A$ satisfies $\ep^kx=x$, then $\zeta^{k/p^i}x^{(i)}=x^{(i)}$ for all $i\ge0$, and so $x^{(i)}=0$ for $i\gg0$ since then $\zeta^{k/p^i}-1$ is a non-zero-divisor of $A$, just as at the end of the proof of Proposition \ref{proposition_roots_of_unity}. $\square$} of $W(A^\flat)$, and so both $H^{n+1}(K)[\mu]$ and $H^n(K)/\mu\cong W(A^\flat)/([\ep^{\min\{k,0\}}]-1)^{\binom{d-1}{n-1}}$ are $p$-torsion-free; therefore $H^n(K\otimes_{W(A^\flat)}W(A^\flat)/\mu)$ is $p$-torsion-free.

Next we prove that the cohomology of $\eta_\mu K\otimes_{W(A^\flat)}W(A^\flat)/\tilde\xi_r$ is $p$-torsion-free. Lemma \ref{lemma_on_Koszul_1} implies that $\eta_\mu K\cong K_{W(A^\flat)}(([\ep^{k_1}]-1)/\mu,\dots,([\ep^{k_d}]-1)/\mu)$ if $k_i\in\bb Z$ for all $i$, and that $\eta_\mu K$ is acyclic otherwise. Evidently we may henceforth assume we are in the first case; then $\tilde\theta_r$ induces an identification of complexes of $W(A^\flat)/\tilde\xi_r=W_r(A)$-modules \[\eta_\mu K\otimes_{W(A^\flat)}W(A^\flat)/\tilde\xi_r\cong K_{W_r(A)}\left(\frac{[\zeta^{k_1/p^r}]-1}{[\zeta_{p^r}]-1},\dots,\frac{[\zeta^{k_d/p^r}]-1}{[\zeta_{p^r}]-1}\right),\] and it remains to prove that the Koszul complex on the right has $p$-torsion-free cohomology. But Lemma~\ref{lemma_on_Koszul_2} implies that each cohomology group of this Koszul complex is isomorphic to a direct sum of copies of 
\[W_r(A)\left[\tfrac{[\zeta_{p^j}]-1}{[\zeta_{p^r}]-1}\right]\qquad\text{and}\qquad W_r(A)/\tfrac{[\zeta_{p^{j}}]-1}{[\zeta_{p^r}]-1},\] where $j:=-\min_{1\le i\le d}\nu_p(k_i/p^r)\le r$. The left module is $p$-torsion-free since $W_r(A)$ is $p$-torsion-free, while the right module (which $=W_r(A)$ if $j\le0$, so we suppose $1\le j\le r$) can be easily shown to be isomorphic to $W_{r-j}(A)$ via $F^j:W_r(A)\to W_{r-j}(A)$ [Corol.~3.18, BMS], which is again $p$-torsion-free.
\end{proof}

%

Next we prove the existence of suitable structure maps:

\begin{lemma}\label{lemma_lambda_grp}
There exists a unique collection of $W_r(A)$-algebra homomorphisms $\lambda_{r,\sub{grp}}:W_r(A[\ul T^{\pm1}])\to H^0(D^\sub{grp}/\tilde\xi_r)$, for $r\ge1$, making the diagrams of Theorem~\ref{theorem_witt_main} commute and satisfying $\lambda_{r,\sub{grp}}([T_i])=U_i$ for $i=1,\dots,d$.
\end{lemma}
\begin{proof}
The maps $\tilde\theta_r$ induce identifications $W(A^\flat)[\ul U^{\pm1/p^\infty}]/\tilde\xi_r=W_r(A)[\ul U^{\pm1/p^\infty}]$ and thus $H^0(D^\sub{grp}/\tilde\xi_r)=W_r(A)[\ul U^{\pm1}]^{\bb Z^d}$, where the latter term is the fixed points for $\bb Z^d$ acting on $W_r(A)[\ul U^{\pm1}]$ via \[\gamma_iU_j^k:=\begin{cases} [\zeta^{k/p^r}]U_j^k& i=j \\ U_j^k & i\neq j\end{cases}\] (where the notation $\zeta^{k/p^r}$ was explained at the start of the proof of Theorem \ref{theorem_integral_purity}). Under this identification of $H^0(D^\sub{grp}/\tilde\xi_r)$, it is easy to see that the maps $\phi_\sub{grp}^{-1}$, ``canonical projection'', and $\times\phi_\sub{grp}^{r+1}(\xi)$ in the diagrams of Theorem \ref{theorem_witt_main} are given respectively by:
\begin{itemize}
\item the ring homomorphism $R:W_{r+1}(A)[\ul U^{\pm 1/p^\infty}]\to W_r(A)[\ul U^{\pm 1/p^\infty}]$ which acts as the Witt vector Restriction map on the coefficients and satisfies $R(U_i^k)=U_i^{k/p}$ for all $k\in\bb Z[\tfrac1p]$ and $i=1,\dots,d$;
\item the ring homomorphism $F:W_{r+1}(A)[\ul U^{\pm 1/p^\infty}]\to W_r(A)[\ul U^{\pm 1/p^\infty}]$ which acts as the Witt vector Frobenius on the coefficients and fixes the variables;
\item the additive map $V:W_r(A)[\ul U^{\pm 1/p^\infty}]\to W_{r+1}(A)[\ul U^{\pm 1/p^\infty}]$ which is defined by $V(\al U_1^{k_1}\cdots U_d^{k_d}):=V(\al)U_1^{k_1}\cdots U_d^{k_d}$ for all $\al\in W_r(A)$ and $k_1,\dots,k_d\in\bb Z[\tfrac1p]$.
\end{itemize}
Therefore the proof will be complete if we show that there is a unique collection of $W_r(A)$-algebra homomorphisms $\lambda_{r,\sub{grp}}:W_r(A[\ul T^{\pm1}])\to W_r(A)[\ul U^{\pm1/p^\infty}]$ commuting with $R,F,V$ on each side and satisfying $\lambda_{r,\sub{grp}}([T_i])=U_i$ for $i=1,\dots,d$.

To prove this, we first use the standard isomorphism of $W_r(A)$-algebras\footnote{This isomorphism is proved by localising the analogous assertion for $A[\ul T^{1/p^\infty}]$, which is an easy consequence of \cite[Corol.~2.4]{LangerZink2004}. The cited result also implies that $W_r(A[\ul T^{\pm1}])$ is generated as a $W_r(A)$-module by the elements $V^j([T_i^k])$, for $k\in\bb Z$, $j\ge0$, $i=1,\dots,d$, which proves the uniqueness of the maps $\lambda_{r,\sub{grp}}$.} \[W_r(A)[\ul U^{\pm1/p^\infty}]\isoto W_r(A[\ul T^{\pm1/p^\infty}]),\qquad U_i^k\mapsto [T_i^k]\qquad (k\in\bb Z[\tfrac1p])\] to define a modified isomorphism \[\tau_r:W_r(A)[\ul U^{\pm1/p^\infty}]\isoto W_r(A[\ul T^{\pm1/p^\infty}]),\qquad U_i^k\mapsto [T_i^{k/p^r}]\qquad (k\in\bb Z[\tfrac1p]),\] noting that the new maps $\tau_r$ respect $R,F,V$ on each side (the reader should check this by explicit calculation). Therefore the collection of maps \[\lambda_{r,\sub{grp}}:W_r(A[\ul T^{\pm1}])\into W_r(A[\ul T^{\pm1/p^\infty}])\stackrel{\tau_r^{-1}}\Isoto W_r(A)[\ul U^{\pm1/p^\infty}]\] satisfies the desired conditions (and their uniqueness was explained in the previous footnote).
%
%
\end{proof}

The previous two lemmas show that all hypotheses of Theorem  \ref{theorem_witt_main} are satisfied, and so there are associated universal maps of Witt complexes \[\lambda_{r,\sub{grp}}^\blob:W_r\Omega^\blob_{A[\ul T^{\pm1}]/A}\To \cal W_r^\blob(D^\sub{grp}).\] As already explained, the key local result underlying the forthcoming proof of the $p$-adic Cartier isomorphism will be the fact that these are isomorphisms:

\begin{theorem}\label{theorem_de_rham_witt_group_cohomology}
The map $\lambda_{r,\sub{grp}}^n:W_r\Omega^n_{A[\ul T^{\pm1}]/A}\to \cal W_r^n(D^\sub{grp})$ is an isomorphism for each $r\ge1$, $n\ge0$.
\end{theorem}
\begin{proof}
We will content ourselves here with proving that $\lambda_{r,\sub{grp},\kappa}^n:=\lambda_{r,\sub{grp}}^n\otimes_{W_r(A)}W_r(\kappa)$ is an isomorphism,\footnote{To then deduce that $\lambda_{r,\sub{grp}}^n$ itself is an isomorphism, one applies a form of Nakayama's lemma exploiting the fact that (the non-finitely generated $W_r(A)$-modules) $W_r\Omega^n_{A[\ul T^{\pm1}]/A}$ and $\cal W_r^n(D^\sub{grp})$ admit compatible direct sum decompositions into certain finitely generated $W_r(A)$-modules for which Nakayama's lemma is valid; see [Lem.~11.14, BMS] for the details.} where $\kappa:=A/\sqrt{pA}$ is the perfect ring obtained by modding out $A$ by its ideal of $p$-adically topologically nilpotent elements. Recalling from Remark \ref{remark_de_rham_witt}(vii) that the canonical base change map $W_r\Omega^n_{A[\ul T^{\pm1}]/A}\otimes_{W_r(A)}W_r(\kappa)\to W_r\Omega^n_{\kappa[\ul T^{\pm1}]/\kappa}$ is an isomorphism, this means showing that $\lambda_{r,\sub{grp},\kappa}^n$ induces an isomorphism $W_r\Omega^n_{\kappa[\ul T^{\pm1}]/\kappa}\isoto \cal W_r^n(D^\sub{grp})\otimes_{W_r(A)}W_r(\kappa)$; this will turn out to be exactly Illusie--Raynaud's Cartier isomorphism for the classical de Rham--Witt complex.


We now begin the proof that $\lambda_{r,\sub{grp},\kappa}^n$ is an isomorphism. By the K\"unneth formula and the standard calculation of group cohomology of an infinite cyclic group, we may represent $D^\sub{grp}$ by the particular complex of $W(A^\flat)$-modules \[D^\sub{grp}=\bigotimes_{i=1}^d\left[W(A^\flat)[U_i^{\pm1/p^\infty}]\xTo{\gamma_i-1} W(A^\flat)[U_i^{\pm1/p^\infty}]\right],\] where each length two complex is \[W(A^\flat)[U_i^{\pm1/p^\infty}]\xTo{\gamma_i-1} W(A^\flat)[U_i^{\pm1/p^\infty}],\quad U_i^k\mapsto ([\ep^k]-1)U_i^k\qquad (k\in\bb Z[\tfrac1p]).\]  (Note: although we previously used $D^\sub{grp}$ to denote $R\Gamma(\bb Z^d,W(A^\flat)[\ul U^{\pm1/p^\infty}])$ in a derived sense, in the rest of this proof we have this particular honest complex of flat $W(A^\flat)$-modules in mind when writing $D^\sub{grp}$.) This length two complex obviously receives a injective map, given by the identity in degree $0$ and by multiplication by $\mu$ in degree $1$, from \[D_{\sub{int},i}^\sub{grp}:=\left[W(A^\flat)[U_i^{\pm1}]\to W(A^\flat)[U_i^{\pm1}]\right],\quad U_i^k\mapsto \tfrac{[\ep^k]-1}{\mu}U^k_i\quad (k\in\bb Z),\] and tensoring over $i=1,\dots,d$ defines a split injection of complexes of $W(A^\flat)$-modules\footnote{The complex $D^\sub{grp}_\sub{int}$ (resp.~$D_{\sub{int},i}^\sub{grp}$) is in fact the ``$q$-de Rham complex'' $[\ep]\text-\Omega^\blob_{W(A^\flat)[\ul U^{\pm1}]/W(A^\flat)}$ (resp.~$[\ep]\text-\Omega^\blob_{W(A^\flat)[U_i^{\pm1}]/W(A^\flat)}$) of $W(A^\flat)[\ul U^{\pm1}]$ (resp.~$W(A^\flat)[U_i^{\pm1}]$) associated to the element $q=[\ep]\in W(A^\flat)$.} \[D_\sub{int}^\sub{grp}:=\bigotimes_{i=1}^nD_{\sub{int},i}^\sub{grp}\To D^\sub{grp}.\] The content of the second sentence of the final paragraph of the proof of Lemma \ref{lemma_Dgrp_1} was exactly that this inclusion has image in $\eta_\mu D^\sub{grp}$ and that the induced map $\frak q:D^\sub{grp}_\sub{int}\into \eta_\mu D^\sub{grp}$ is a quasi-isomorphism.

The next important observation (which is most natural from the point of view of $q$-de Rham complexes) is that there is an identification $D^\sub{grp}_\sub{int}\otimes_{W(A^\flat)}W(\kappa)=\Omega^\blob_{W(k)[\ul U^{\pm1}]/W(\kappa)}$: indeed, the canonical projection $A^\flat\to A/pA\to \kappa$ sends  $\ep$ to $1$, and so the projection $W(A^\flat)\to W(\kappa)$ sends $([\ep^k]-1)/\mu=1+[\ep]+\cdots+[\ep]^{k-1}$ to $k$, whence \[D^\sub{grp}_\sub{int}\otimes_{W(A^\flat)}W(\kappa)=\bigotimes_{i=1}^n\left[W(\kappa)[U_i^{\pm1}]\xTo{U_i^k\mapsto kU_i^k} W(\kappa)[U_i^{\pm1}]\right]=\Omega^\blob_{W(\kappa)[\ul U^{\pm1}]/W(\kappa)}.\] The final identification here is most natural after inserting a dummy basis element $\dlog U_i$ in degree one of each two term complex.

Base changing the Bockstein construction\footnote{If $\al:R\to S$ is a ring homomorphism, $f\in R$ is a non-zero-divisor whose image $\al(f)\in S$ is still a non-zero-divisor, and $C\in D(R)$, then there is a base change map $[H^\blob(C\dotimes_RR/fR),\op{Bock}_f]\otimes_{R/fR}S/\al(f)S\to[H^\blob(C\dotimes_RS/\al(f)S),\op{Bock}_{\al(f)}]$ of complexes of $S/\al(f)S$-modules; it is an isomorphism if the $R/fR$-modules $H^*(D\dotimes_RR/fR)$ are Tor-independent from $S/\al(f)S$, as the reader will easily prove (c.f., Remark \ref{remark_base_change}(ii)).

Here we are applying this base change along the canonical map $W(A^\flat)\to W(\kappa)$, which sends $\tilde\xi_r$ to $p^r$, and the complex $\eta_fD_\sub{grp}$. The Tor-independence condition is satisfied in this case since the $W_r(A)$-modules $\cal W_r^*(D_\sub{grp})$ are Tor-independent from $W_r(k)$: indeed, the proof of Lemma \ref{lemma_Dgrp_1} shows that the cohomology groups of $\eta_\mu D_\sub{grp}/\tilde\xi_r$ are direct sums of $W_r(A)$-modules of the form \[W_r(A), \qquad W_r(A)\left[\tfrac{[\zeta_{p^j}]-1}{[\zeta_{p^r}]-1}\right],\qquad\text{and}\qquad W_r(A)/\tfrac{[\zeta_{p^{j}}]-1}{[\zeta_{p^r}]-1}, \qquad1\le j<r,\] which are Tor-independent from $W_r(\kappa)$ by Lems.~3.13 \& 3.18(iii) and Rmk.~3.19 of [BMS].
} 
along $W(A^\flat)\to W(\kappa)$ therefore yields isomorphisms of complexes of $W_r(\kappa)$-modules \[\hspace{-7mm}\cal W_r^\blob(D^\sub{grp})\otimes_{W_r(A)}W_r(\kappa)\isoto[H^\blob(\eta_\mu D^\sub{grp}\dotimes_{W(A^\flat)}W(\kappa)/p^r),\op{Bock}_{p^r}]\stackrel{\frak q\simeq}\leftarrow[H^\blob(\Omega^\blob_{W(\kappa)[\ul U^{\pm1}]/W(\kappa)}\otimes_{W(\kappa),\tilde\theta_r}W_r(\kappa)),\op{Bock}_{p^r}]\] But the complex on the right (hence on the left) identifies with $W_r\Omega^\blob_{\kappa[\ul T^{\pm1}]/\kappa}$ by the de Rham--Witt Cartier isomorphism of Illusie--Raynaud \cite[\S III.1]{IllusieRaynaud1983}, and the resulting map \[W_r\Omega^\blob_{A[\ul T^{\pm1}]/A}\otimes_{W_r(A)}W_r(\kappa)\xTo{can.~map} W_r\Omega^\blob_{\kappa[\ul T^{\pm1}]/\kappa}\cong \cal W_r^\blob(D^\sub{grp})\otimes_{W_r(A)}W_r(\kappa)\] is precisely $\lambda_{r,\sub{grp},\kappa}^\blob$: this is proved by observing that the above isomorphisms (including the de Rham--Witt Cartier isomorphism) are all compatible with multiplicative structure, whence it suffices to check in degree $0$, which is not hard (see [Thm.~11.13, BMS] for a few more details). As we commented at the beginning of the proof, the canonical base change map of relative de Rham--Witt complexes in the previous line is an isomorphism, and so in conclusion $\lambda_{r,\sub{grp},\kappa}^\blob$ is an isomorphism.
\qedhere

\end{proof}

\section{The proof of the $p$-adic Cartier isomorphism}\label{section_Cartier}
This section is devoted to a detailed sketch of the $p$-adic Cartier isomorphism stated in Theorem \ref{theorem_p_adic_Cartier}. We adopt the set-up from the start of Section \ref{section_main}, namely 
\begin{itemize}\itemsep0pt
\item $\bb C$ is a complete, non-archimedean, algebraically closed field of mixed characteristic; ring of integers~$\roi$ with maximal ideal $\frak m$; residue field $k$.
\item We pick a compatible sequence $\zeta_p,\zeta_{p^2},\dots\in\roi$ of $p$-power roots of unity, and define $\mu,\xi,\xi_r,\tilde\xi,\tilde\xi_r\in\bb A_\sub{inf}=W(\roi^\flat)$ as in \S\ref{subsection_roots_of_unity}.
\item $\frak X$ will denote various smooth formal schemes over $\roi$.
\end{itemize}


\subsection{Technical lemmas: base change and global-to-local isomorphisms}\label{subsection_technical}
Here in \S\ref{subsection_technical} we state, and sketch the proofs of, certain technical lemmas which need to be established as part of the proof of the $p$-adic Cartier isomorphism. We adopt the following local set-up: let $R$ be a $p$-adically complete, formally smooth $\roi$-algebra and $\frak X:=\Spf R$, with associated generic fibre being the rigid affinoid $X=\Sp R[\tfrac 1p]$. We will often impose the extra condition that $R$ is {\em small}, i.e., that there exists a formally \'etale map (a ``framing'') $\roi\pid{\ul T^{\pm1}}=\roi\pid{T_1^{\pm1},\dots,T_d^{\pm1}}\to R$; we stress however that we are careful to formulate certain results (e.g., Lemma \ref{lemma_technical_1}) without reference to any such framing (its existence will simply be required in the course of the proof).

Firstly, as explained at the end of Remark \ref{remark_Leta_on_sheaves} (taking $\cal T=\frak X_\sub{Zar}$ and $C=R\nu_*\bb A_{\sub{inf},X}$), there is a natural global-to-local morphism $L\eta_\mu R\Gamma_\sub{Zar}(\frak X,R\nu_*\bb A_{\sub{inf},X})\to R\Gamma_\sub{Zar}(\frak X,\bb A\Omega_{\frak X/\roi})$ of complexes of $\bb A_\sub{inf}$-modules; this may be rewritten as \[\bb A\Omega_{R/\roi}^\sub{pro\'et}:=L\eta_\mu R\Gamma_\sub{pro\'et}(X,\bb A_{\sub{inf},X})\To R\Gamma_\sub{Zar}(\frak X,\bb A\Omega_{\frak X/\roi}).\tag{t1}\] There is an analogous global-to-local morphism of complexes of $W_r(\roi)$-modules \[\tilde{W_r\Omega}_{R/\roi}^\sub{pro\'et}:=L\eta_{[\zeta_{p^r}]-1} R\Gamma_\sub{pro\'et}(X,W_r(\hat\roi_X^+))\To R\Gamma_\sub{Zar}(\frak X,L\eta_{[\zeta_{p^r}]-1}R\nu_*W_r(\hat\roi_X^+)).\tag{t2}\] Thirdly, recalling Corollary \ref{corollary_non_zero_divisors_on_sheaf} that $\tilde\theta_r:\bb A_{\sub{inf},X}/\tilde\xi_r\isoto W_r(\hat\roi_X^+)$ (which we continue to often implicitly view as an identification), there is a base change morphism (see Remark \ref{remark_base_change}) of complexes of $W_r(\roi)$-modules \[\bb A\Omega_{R/\roi}^\sub{pro\'et}/\tilde\xi_r=L\eta_\mu R\Gamma_\sub{pro\'et}(X,\bb A_{\sub{inf},X})\dotimes_{\bb A_\sub{inf}}\bb A_\sub{inf}/\tilde \xi_r\bb A_\sub{inf} \To L\eta_{[\zeta_{p^r}]-1} R\Gamma_\sub{pro\'et}(\frak X,W_r(\hat\roi_X^+))=\tilde{W_r\Omega}_{R/\roi}^\sub{pro\'et}.\tag{t3}\] As we have commented earlier, global-to-local and base change morphisms associated to the d\'ecalage functor are not in general quasi-isomorphisms; remarkably, they are in our setting:

\begin{lemma}\label{lemma_technical_1}
If $R$ is small then maps (t1), (t2), and (t3) are quasi-isomorphisms and, moreover:
\begin{enumerate}
\item the cohomology groups of $\tilde{W_r\Omega}_{R/\roi}^\sub{pro\'et}$ are $p$-torsion-free;
\item if $R'$ is a $p$-adically complete, formally \'etale $R$-algebra, then the canonical base change map $\tilde{W_r\Omega}_{R/\roi}^\sub{pro\'et}\hat\dotimes_{W_r(R)}W_r(R')\to \tilde{W_r\Omega}_{R'/\roi}^\sub{pro\'et}$ is a quasi-isomorphism.
\end{enumerate}
\end{lemma}

The key to proving Lemma \ref{lemma_technical_1}, and to performing necessary auxiliary calculations, is the Cartan--Leray almost quasi-isomorphisms of \S\ref{subsection_CL}, for which we must assume that $R$ is small and fix a framing $\roi\pid{\ul T^{\pm1}}\to R$; set $R_\infty:=R\hat\otimes_{\roi\pid{\ul T^{\pm1}}}\roi\pid{\ul T^{\pm1/p^\infty}}$ as in \S\ref{subsection_CL}. Then, as explained in \S\ref{subsection_CL} and repeated in Remark \ref{remark_vague_outline}, there are Cartan--Leray almost (wrt.~$W(\frak m^\flat)$ and $W_r(\frak m)$ respectively) quasi-isomorphisms of complexes of $\bb A_\sub{inf}$- and $W_r(\roi)$-modules respectively \[R\Gamma_\sub{cont}(\bb Z_p(1)^d,W(R_\infty^\flat))\To R\Gamma_\sub{pro\'et}(X,\bb A_{\sub{inf},X})\] and \[R\Gamma_\sub{cont}(\bb Z_p(1)^d,W_r(R_\infty))\To R\Gamma_\sub{pro\'et}(X,W_r(\hat\roi_X^+).\] Applying $L\eta_\mu$ (resp.~$L\eta_{[\zeta_{p^r}]-1}$) obtains \[\bb A\Omega_{R/\roi}^\square:=L\eta_\mu R\Gamma_\sub{cont}(\bb Z_p(1)^d,W(R_\infty^\flat))\To L\eta_\mu R\Gamma_\sub{pro\'et}(X,\bb A_{\sub{inf},X})=\bb A\Omega_{R/\roi}^\sub{pro\'et}\tag{t4}\] and \[\tilde{W_r\Omega}_{R/\roi}^\square:=L\eta_{[\zeta_{p^r}]-1}R\Gamma_\sub{cont}(\bb Z_p(1)^d,W_r(R_\infty))\To L\eta_{[\zeta_{p^r}]-1}R\Gamma_\sub{pro\'et}(X,W_r(\hat\roi_X^+))=\tilde{W_r\Omega}_{R/\roi}^\sub{pro\'et}\tag{t5}\] (The squares $\square$ remind us that the objects depend on the chosen framing.) The second technical lemma, stating that the d\'ecalage functor has transformed the almost quasi-isomorphisms into actual quasi-isomorphisms, and hence reminiscent of Theorem \ref{theorem_integral_purity}, is:

\begin{lemma}\label{lemma_technical_2}
(t4) and (t5) are quasi-isomorphisms.
\end{lemma}

We now sketch a proof of the previous two technical lemmas. The arguments are of a similar flavour to what we have already seen in \S\ref{subsection_faltings_purity} and \S\ref{subsection_local_Cartier}, so we will not provide all the details; see [\S9, BMS] for further details. For the overall logic of the proof, it will be helpful to draw the following commutative diagram of the maps of interest:
\[\xymatrix@C=0.7cm@R=1.5cm{
R\Gamma_\sub{Zar}(\frak X,\tilde{W_r\Omega}_{\frak X/\roi})=R\Gamma_\sub{Zar}(\frak X,\bb A\Omega_{\frak X/\roi})/\tilde\xi_r \ar[r]^-{(t7)} & R\Gamma_\sub{Zar}(\frak X,L\eta_{[\zeta_{p^r}]-1}R\nu_* W_r(\hat\roi_X^+))\\
\bb A\Omega_{R/\roi}^\sub{pro\'et}/\tilde\xi_r=L\eta_\mu R\Gamma_\sub{pro\'et}(X, \bb A_{\sub{inf},X})/\tilde\xi_r \ar[r]^-{(t3)}\ar[u]^{(t1)\op{mod}\tilde\xi_r} & L\eta_{[\zeta_{p^r}]-1}R\Gamma_\sub{pro\'et}(X, W_r(\hat\roi_X^+))=\tilde{W_r\Omega}_{R/\roi}^\sub{pro\'et}\ar[u]_{(t2)}\\
\bb A\Omega_{R/\roi}^\square/\tilde\xi_r=L\eta_\mu R\Gamma_\sub{cont}(\bb Z_p(1)^d, W(R_\infty^\flat))/\tilde\xi_r \ar[r]^-{(t6)} \ar[u]^{(t4)\op{mod}\tilde\xi_r}& L\eta_{[\zeta_{p^r}]-1}R\Gamma_\sub{cont}(\bb Z_p(1)^d, W_r(R_\infty))=\tilde{W_r\Omega}_{R/\roi}^\square\ar[u]_{(t5)}
}\]
The new maps, namely (t6) and (t7), are simply the base change maps associated to the identifications $\tilde\theta_r:W(R_\infty^\flat)/\tilde\xi_r\isoto W_r(R_\infty)$ and $\tilde\theta_r:\bb A_{\sub{inf},X}/\tilde\xi_r\isoto W_r(\hat\roi_X^+)$. In particular, the diagram commutes by the naturality of global-to-local and base-change maps. We will show that (t1)--(t7) are quasi-isomorphisms. We begin by proving the following, which is [Lem.~9.7(i), BMS]:

\begin{lemma}\label{lemma_technical_3}
$R\Gamma_\sub{cont}(\bb Z_p(1)^d,W_r(R_\infty))$ is quasi-isomorphic to the derived $p$-adic completion of a direct sum of Koszul complexes $K_{W_r(\roi)}([\zeta^{k_1}]-1,\dots,[\zeta^{k_d}]-1)$, for varying $k_i\in\bb Z[\tfrac1p]$ .
\end{lemma}
\begin{proof}
Since Witt vectors preserve \'etale morphisms [BMS, Thm.~10.4], the maps $W_r(\roi\pid{\ul T^{\pm1/p^\infty}}/p^n)\to W_r(R_\infty/p^n)$ induced by the framing are \'etale for all $n\ge1$, whence the same is true of the maps $W_r(\roi\pid{\ul T^{\pm1/p^\infty}})/p^n\to W_r(R_\infty)/p^n$ (since the systems of ideals $(p^nW_r(B))_{n\ge1}$ and $(W_r(p^nB))_{n\ge1}$ are intertwined for any ring $B$; for a proof see, e.g., [BMS, Lem.~10.3]). In particular, these latter maps are flat for all $n\ge1$ and therefore \[R\Gamma_\sub{cont}(\bb Z_p(1)^d,W_r(R_\infty))\simeq R\Gamma_\sub{cont}(\bb Z_p(1)^d,W_r(\roi\pid{\ul T^{\pm1/p^\infty}}))\hat\dotimes_{W_r(\roi\pid{\ul T^{\pm1/p^\infty}})}W_r(R_\infty).\] 

Next we note that $R\Gamma_\sub{cont}(\bb Z_p(1)^d,W_r(\roi\pid{\ul T^{\pm1/p^\infty}}))$ identifies with the derived $p$-adic completion of the complex $R\Gamma(\bb Z^d,W_r(\roi)[\ul U^{\pm1/p^\infty}])$ which was studied in \S\ref{subsection_local_Cartier}: we will explain and prove this identification carefully towards the end of the proof of Proposition \ref{proposition_final}, so do not say more here. Also, an easy modification of the first half of the proof of Lemma \ref{lemma_Dgrp_1} shows that \[R\Gamma(\bb Z^d,W_r(\roi)[\ul U^{\pm1/p^\infty}])\simeq \bigoplus_{k_1,\dots,k_d\in \bb Z[\tfrac1p]\cap[0,1)}K_{W_r(\roi)}([\zeta^{k_1}]-1,\dots,[\zeta^{k_d}]-1)\otimes_{W_r(\roi)}W_r(\roi)[\ul U^{\pm1/p^\infty}].\]
Assembling these identities shows that $R\Gamma_\sub{cont}(\bb Z_p(1)^d,W_r(R_\infty))$ is quasi-isomorphic to the derived $p$-adic completion of \[\bigoplus_{k_1,\dots,k_d\in \bb Z[\tfrac1p]\cap[0,1)}K_{W_r(\roi)}([\zeta^{k_1}]-1,\dots,[\zeta^{k_d}]-1)\otimes_{W_r(\roi)}W_r(R_\infty).\] The proof is completed in a similar way to that of Theorem \ref{theorem_integral_purity}, namely by arguing that $W_r(R_\infty)$ is the $p$-adic completion of a free $W_r(\roi)$-module, which we leave to the reader.
\end{proof}

\begin{proof}[Proof that (t5) is a quasi-isom.]
Using Lemma \ref{lemma_on_Koszul_2} to calculate the cohomology of the Koszul complexes in Lemma \ref{lemma_technical_3} (and footnote \ref{footnote_completing_sums} to exchage cohomology and $p$-adic completions), it follows that each cohomology group of $R\Gamma_\sub{cont}(\bb Z_p(1)^d,W_r(R_\infty))$ is isomorphic to the $p$-adic completion of a direct sum of copies of \[W_r(\roi),\qquad W_r(\roi)[[\zeta_{p^j}]-1],\qquad W_r(\roi)/([\zeta_{p^j}]-1),\qquad j\ge1,\] each of which is ``good'' in the sense of Lemma \ref{lemma_technical_0} (wrt.\ $A=W_r(\roi)$, $\frak M=W_r(\frak m)$, and $f=[\zeta_{p^r}]-1$)) by [Corol.~3.29, BMS]. So all cohomology groups $R\Gamma_\sub{cont}(\bb Z_p(1)^d,W_r(R_\infty))$ are good, whence Lemma \ref{lemma_technical_0} implies that (t5) is a quasi-isomorphism.
\end{proof}


\begin{proof}[Proof of Lemma \ref{lemma_technical_1}(i)]
Since $L\eta_{[\zeta_{p^r]}-1}$ commutes with derived $p$-adic completion by Remark \ref{remark_non_primitive}, Lemmas \ref{lemma_technical_3} and \ref{lemma_on_Koszul_1} imply that $\tilde{W_r\Omega}_R^\square=L\eta_{[\zeta_{p^r}]-1}R\Gamma_\sub{cont}(\bb Z_p(1)^d,W_r(R_\infty))$ is quasi-isomorphic to the derived $p$-adic completion of a direct sum of Koszul complexes \[K_{W_r(\roi)}\left(\frac{[\zeta_{p^{j_1}}]-1}{[\zeta_{p^r}]-1},\dots,\frac{[\zeta_{p^{j_d}}]-1}{[\zeta_{p^r}]-1}\right),\] for varying $j_1,\dots,j_d\le r$. The calculation at the end of the proof of Lemma \ref{lemma_Dgrp_1} therefore shows that the cohomology groups of $\tilde{W_r\Omega}_R^\square$ are $p$-torsion-free. Combining this with quasi-isomorphism (t5) proves Lemma \ref{lemma_technical_1}(i).
\end{proof}


\begin{proof}[Proof of Lemma \ref{lemma_technical_1}(ii) and that (t2) is a quasi-isom.]
Let $R'$ be a $p$-adically complete, formally \'etale $R$-algebra, and write $R'_\infty:=R'\hat\otimes_{\roi\pid{\ul T^{\pm1}}}\roi\pid{\ul T^{\pm1/p^\infty}}$. Just as at the start of the proof of Lemma \ref{lemma_technical_3}, the maps $W_r(R_\infty)/p^n\to W_r(R_\infty')/p^n$ are flat for all $n\ge1$, whence the canonical map \[\tilde{W_r\Omega}_{R/\roi}^\square\hat\dotimes_{W_r(R)}W_r(R')\To \tilde{W_r\Omega}_{R'/\roi}^\square\] is a quasi-isomorphism. The same is therefore true after replacing $^\square$ by $^\sub{pro\'et}$ (since (t5) is a quasi-isomorphism for both $R$ and $R'$), and this proves Lemma \ref{lemma_technical_1}(ii). This is a strong enough coherence result to show that $\tilde{W_r\Omega}_{R/\roi}^\sub{pro\'et}\hat\dotimes_{W_r(R)}W_r(\roi_\frak X)\to L\eta_{[\zeta_{p^r}]-1}R\nu_*W_r(\hat\roi_X^+)$ is a quasi-isomorphism of complexes of $W_r(\roi_\frak X)$-modules, and it follows that (t2) is a quasi-isomorphism (see [Corol.~9.11, BMS] for further details).
\end{proof}


\begin{proof}[Proof that (t6) is a quasi-isom.]
According to footnote \ref{footnote_condition_for_W3}, it is enough to prove that the cohomology of the complex $R\Gamma_\sub{cont}(\bb Z_p(1)^d, W(R_\infty^\flat))\dotimes_{\bb A_\sub{inf}}\bb A_\sub{inf}/\mu \bb A_\sub{inf}=R\Gamma_\sub{cont}(\bb Z_p(1)^d, W(R_\infty^\flat)/\mu)$ is $p$-torsion-free. We claim first that the map $W(\roi\pid{\ul T^{\pm1/p^\infty}}^\flat)/\pid{\mu,p^r}\to W(R_\infty^\flat)/\pid{\mu,p^r}$ is \'etale for all $r\ge1$; indeed, $p^r\equiv \tilde\xi_r$ mod $\mu\bb A_\sub{inf}$ and so $\tilde\theta_r$ identifies this map with $W_r(\roi\pid{\ul T^{\pm1/p^\infty}})/([\zeta_{p^r}]-1)\to W_r(R_\infty)/([\zeta_{p^r}]-1)$. But $[\zeta_{p^r}]-1$ divides $p^r$ and, just as at the start of the proof of Lemma \ref{lemma_technical_3}, the map $W_r(\roi\pid{\ul T^{\pm1/p^\infty}})/p^r\to W_r(R_\infty)/p^r$ is \'etale; this proves the claim. 

The claim reduces the proof to showing that the cohomology of $R\Gamma_\sub{cont}(\bb Z_p(1)^d, W(\roi\pid{\ul T^{\pm1/p^\infty}}^\flat)/\mu)$ is $p$-torsion-free. To show this we first observe that there is an isomorphism of $\bb A_\sub{inf}/\mu\bb A_\sub{inf}$-algebras \[\bb A_\sub{inf}/\mu\bb A_\sub{inf}\pid{\ul U^{\pm1/p^\infty}}\Isoto W(\roi\pid{\ul T^{\pm1/p^\infty}}^\flat)/\mu,\quad U_i^k\mapsto [(T_i^k,T_i^{k/p},T_i^{k/p^2},\dots)]\quad (k\in\bb Z[\tfrac1p]),\] which is proved by quotienting the ``standard isomorphism'' in the proof of Lemma \ref{lemma_lambda_grp} by $[\zeta_{p^r}]-1$ and then taking $\projlim_{r\sub{ wrt }F}$. By the same type of Koszul decomposition argument which has been made several times, it now follows that $R\Gamma_\sub{cont}(\bb Z_p(1)^d, W(\roi\pid{\ul T^{\pm1/p^\infty}}^\flat)/\mu)$ is quasi-isomorphic to the derived $p$-adic completion of \[\bigoplus_{k_1,\dots,k_d\in\bb Z[\tfrac1p]}K_{\bb A_\sub{inf}/\mu\bb A_\sub{inf}}([\ep^{k_1}]-1,\dots,[\ep^{k_2}]-1).\] The cohomology of each of these Koszul complexes is, by Lemma \ref{lemma_on_Koszul_2}, a finite direct sum of copies of \[(\bb A_\sub{inf}/\mu\bb A_\sub{inf})[[\ep^k]-1]\qquad\text{and}\qquad \bb A_\sub{inf}/([\ep^k]-1)\bb A_\sub{inf}\] for various $k\in\bb Z[\tfrac1p]$. But these are $p$-torsion-free since $p,[\ep^k]-1$ is a regular sequence of $\bb A_\sub{inf}$ (see footnote \ref{footnote_reg_seq}) for any $k\in\bb Z[\tfrac1p]$  (including $k=1$, to treat the left term).
\end{proof}


\begin{proof}[Proof that (t4) is a quasi-isom.]
Proving that (t4) is a quasi-isomorphism was done in [BMS] via a subtle generalisation of the ``good'' cohomology groups argument of Lemma \ref{lemma_technical_0}, which required calculating $R\Gamma_\sub{cont}(\bb Z_p(1)^d, W(\roi\pid{\ul T^{\pm1/p^\infty}}^\flat))$ in terms of Koszul complexes\footnote{Here we explain why the analogous calculations we have already seen do not generalise to this case. Although there is an identification $\bb A_\sub{inf}\langle \ul U^{\pm1/p^\infty}\rangle\isoto W(\roi\pid{\ul T^{\pm1/p^\infty}}^\flat)$, the convergence of the power series on the left is with respect to the $\pid{p,\xi}$-adic topology. But neither $R\Gamma_\sub{cont}(\bb Z_p(1)^d,\cdot)$ nor $L\eta_\mu$ commute with derived $\pid{p,\xi}$-adic completion!} (see Lems.~9.12--9.13 and the first paragraph of Prop.~9.14). Here we will offer a simpler argument which was presented first in \cite[Rem.~7.11]{Bhatt2016}.

We need the following strengthening of Lemma \ref{lemma_technical_0}: ``Let $\frak M\subseteq A$ be an ideal of a ring and $f\in \frak M$ a non-zero-divisor; if $\cal C\to \cal D$ is a morphism of complexes of $A$-modules whose cone is killed by $\frak M$, and all cohomology groups of $C\dotimes_AA/fA$ contain no non-zero elements killed by $\frak M^2$, then $L\eta_f C\to L\eta_fD$ is a quasi-isomorphism.'' This follows from the proof of \cite[Lem.~6.14]{Bhatt2016} and exploits the relation between $L\eta$ and the Bockstein construction.

Applying this in the case $A=\bb A_\sub{inf}$, $f=\mu$, and $\frak M=W(\frak m^\flat)$, the proof immediately reduces to showing that the cohomology of $R\Gamma_\sub{cont}(\bb Z_p(1)^d,W(R^\flat_\infty)/\mu)$ contains no non-zero elements killed by $W(\frak m^\flat)^2$. But the decomposition from the previous proof showed that each of these cohomology groups was the $p$-adic completion of a direct sum of copies of the $p$-torsion-free modules \[(\bb A_\sub{inf}/\mu\bb A_\sub{inf})[[\ep^k]-1]\qquad\text{and}\qquad \bb A_\sub{inf}/([\ep^k]-1)\bb A_\sub{inf}\] for various $k\in\bb Z[\tfrac1p]$; so it is enough to show for any $k\in\bb Z[\tfrac1p]$ (including $k=1$, to treat the left term) that $\bb A_\sub{inf}/([\ep^k]-1)\bb A_\sub{inf}$ contains no non-zero elements killed by $W(\frak m^\flat)^2$. But the maps $\tilde\theta_r$ induce an isomorphism $\bb A_\sub{inf}/([\ep^k]-1)\bb A_\sub{inf}\isoto \projlim_{r\sub{ wrt }F}W_r(\roi)/([\zeta^{k/p^r}]-1)$, and each $W_r(\roi)/([\zeta^{k/p^r}]-1)$ contains no non-zero elements killed by $W_r(\frak m)^2=W_r(\frak m)$ (recall that $W_r(\frak m)$ is an ideal for almost mathematics, c.f.,  footnote \ref{footnote_almost_W_r}), as we already saw above in the proof that (t5) is a quasi-isomorphism.
\end{proof}

\begin{proof}[Proof that (t1), (t3), and (t7) are quasi-isoms.]
Since we now know that (t4) is a quasi-isomorphism, the commutativity of the diagram implies that (t3) is a quasi-isomorphism. Using this quasi-isomorphism, and by taking $\op{Rlim}_{r\sub{ wrt }F}$ of the quasi-isomorphisms (t2), it can be shown that (t1) is a quasi-isomorphism [BMS, Prop.~9.14]. Finally, the commutativity of the diagram implies that (t7) is also a quasi-isomorphism.
\end{proof}

This finishes the proofs of the technical lemmas, but we note in addition the following consequence which was needed at the start of the proof of Theorem \ref{theorem_main_with_proofs}:

\begin{corollary}\label{corollary_completeness}
If $\frak X$ is a smooth formal scheme over $\roi$, then the complex of $\bb A_\sub{inf}$-modules $R\Gamma_\sub{Zar}(\frak X,\bb A\Omega_{\frak X/\roi})$ is derived $\xi$-adically complete.
\end{corollary}
\begin{proof}
By picking a cover of $\frak X$ by small opens, we may suppose that $\frak X=\Spf R$ is a small affine as above. Then the complex $R\Gamma_\sub{cont}(\bb Z_p(1)^d,W(R^\flat_\infty))$ is derived $\xi$-adically complete since $W(R^\flat_\infty)$ is $\xi$-adically complete,\footnote{If $A$ is any perfectoid ring then $W(A^\flat)$ is $\ker\theta$-adically complete.} whence $\bb A\Omega_{R/\roi}^\square$ is derived $\xi$-adically complete since $L\eta_\mu$ preserves the completeness by footnote \ref{footnote_complete}. Now quasi-isomorphisms (t1) and (t4) complete the proof.
\end{proof}

\subsection{Reduction to a torus and to Theorem \ref{theorem_de_rham_witt_group_cohomology}}\label{subsection_reduction_to_tor}
We continue to suppose that $R$ is a $p$-adically complete, formally smooth $\roi$-algebra, with notation $\frak X=\Spf R$ and $X=\Sp R[\tfrac1p]$ as in \S\ref{subsection_technical}. We wish to apply the construction of \S\ref{subsection_constructing_Witt} (with base perfectoid ring $A=\roi$) to \[D^\sub{pro\'et}_{R/\roi}:=R\Gamma_\sub{pro\'et}(X,\bb A_{\sub{inf},X}),\] and must therefore check that the necessary hypotheses are fulfilled:

\begin{lemma}
$D^\sub{pro\'et}_{R/\roi}$ is a coconnective algebra object in $D(\bb A_\sub{inf})$ which is equipped with a $\phi$-semi-linear quasi-isomorphism $\phi_\sub{pro\'et}:D^\sub{pro\'et}_{R/\roi}\quis D^\sub{pro\'et}_{R/\roi}$. If $R$ is small, then it moreover satisfies assumptions ($\cal W1$)--($\cal W3$) from \S\ref{subsection_constructing_Witt} and there exist $W_r(\roi)$-algebra homomorphisms $\lambda_{r,\sub{pro\'et}}:W_r(R)\to H^0(D^\sub{pro\'et}_{R/\roi}/\tilde\xi_r)$ (natural in $R$) making the diagrams of Theorem \ref{theorem_witt_main} commute.
\end{lemma}
\begin{proof}
$D^\sub{pro\'et}_{R/\roi}$ is clearly a coconnective algebra object in $D(\bb A_\sub{inf})$, and it is equipped with a $\phi$-semi-linear quasi-isomorphism $\phi_\sub{pro\'et}$ induced by the Frobenius automorphism of $\bb A_{\sub{inf},X}$, similarly to Lemma~\ref{lemma_frobenius_on_AOmega}.

Moreover, $H^0(D^\sub{pro\'et}_{R/\roi})=H^0_\sub{pro\'et}(X,\bb A_{\sub{inf},X})$ is $\mu$-torsion-free, since $\bb A_{\sub{inf},X}$ is a $\mu$-torsion-free sheaf on $X_\sub{pro\'et}$ by Corollary \ref{corollary_non_zero_divisors_on_sheaf}; this proves that assumption ($\cal W1$) holds. It remains to check ($\cal W2$) and ($\cal W3$), as well as prove the existence of the maps $\lambda_r$; for this we must now assume that $R$ is small (but we do not fix any framing). Hypotheses ($\cal W2$) and ($\cal W3$) are then exactly the $p$-torsion-freeness and quasi-isomorphism (t3) of Lemma~\ref{lemma_technical_1}.

Finally, the canonical maps of Zariski sheaves of rings $\roi_\frak X\to \nu_*\hat\roi_X^+\to R\nu_*\hat\roi_*^+$ on $\frak X$ induce analogous maps on Witt vectors (see footnote \ref{footnote_Witt}), namely $W_r(\roi_\frak X)\to \nu_* W_r(\hat\roi_X^+)\to R\nu_* W_r(\hat\roi_*^+)$, which are compatible with $R,F,V$ on each term. Applying $H^0(\frak X,-)$ to the composition then yields the following arrow which is also compatible with $R,F,V$:
\[\lambda_{r,\sub{pro\'et}}:W_r(R)=H^0_\sub{Zar}(\frak X,W_r(\roi_\frak X))\To H^0_\sub{pro\'et}(X, W_r(\hat \roi_X^+))\stackrel{\tilde\theta_r}\cong H^0(D^\sub{pro\'et}_{R/\roi}/\tilde\xi_r).\] The isomorphism $\tilde\theta_r$ is compatible with $R,F,V$ on the left according to a sheaf version of the second set of diagrams in Corollary \ref{lemma_theta_r_diagrams}; therefore, overall, these maps $\lambda_{r,\sub{pro\'et}}$ make the diagrams of Theorem \ref{theorem_witt_main} commute, and they are clearly natural in $R$, as desired.
\end{proof}

Continuing to assume that $R$ is small, the previous lemma states that all hypotheses of Theorem \ref{theorem_witt_main} are satisfied for $D^\sub{pro\'et}_{R/\roi}$, and so there are associated universal maps of Witt complexes, natural in $R$, \[\lambda_{r,\sub{pro\'et}}^\blob:W_r\Omega_{R/\roi}^\blob\To\cal W_r^\blob(D^\sub{pro\'et}_{R/\roi})=H^\blob(\bb A\Omega_{R/\roi}^\sub{pro\'et}/\tilde\xi_r).\] By Remark \ref{remark_p_completion}, these factor through the $p$-adic completion of the left side, i.e., \[\hat\lambda_{r,\sub{pro\'et}}^n:(W_r\Omega_{R/\roi}^n)_p^\comp\To H^n(\bb A\Omega_{R/\roi}^\sub{pro\'et}/\tilde\xi_r).\] The $p$-adic Cartier isomorphism will follow from showing that these maps are isomorphisms:

\begin{lemma}\label{lemma_reduction_to_torus}
The following implications hold:
\[\begin{array}{c}
\text{$\hat\lambda_{r,\sub{pro\'et}}^n$ is an isomorphism when $R=\roi\pid{T_1^{\pm1},\dots,T_d^{\pm1}}$.} \\ \Downarrow\\ \text{$\hat\lambda_{r,\sub{pro\'et}}^n$ is an isomorphism for any every small, formally smooth $\roi$-algebra $R$.}\\\Downarrow\\\text{The $p$-adic Cartier isomorphism (Theorem \ref{theorem_p_adic_Cartier}) is true.}
\end{array}\]
\end{lemma}
\begin{proof}
The first implication is a consequence of the domain and codomain of $\hat\lambda_{r,\sub{pro\'et}}^n$ behaving well under formally \'etale base change, according to Remark \ref{remark_de_rham_witt}(v)\&(vi) and Lemma \ref{lemma_technical_1}(ii).

For the second implication it is convenient to briefly change the point of view and notation, by fixing a smooth formal scheme $\frak X$ over $\roi$ and letting $\Spf R\subseteq \frak X$ denote any small affine open. We then consider the composition
\[\xymatrix@C=2cm{
H^n(\bb A\Omega_{R/\roi}^\sub{pro\'et}/\tilde\xi_r)
\ar[r]^\cong_-{(t1)\sub{ mod }\tilde \xi_r}&
H^n_\sub{Zar}(\Spf R,\tilde{W_r\Omega}_{\frak X/\roi})
\ar[r]_{\sub{edge map}}&
\cal H^n(\tilde{W_r\Omega}_{\frak X/\roi})(\Spf R)
}\]
and note that the edge map is an isomorphism by the coherence result of Lemma \ref{lemma_technical_1}(ii).\footnote{Here we are of course using the trivial identification $\tilde{W_r\Omega}_{\frak X/\roi}|_{\Spf R}=\tilde{W_r\Omega}_{\Spf R/\roi}$ in order to appeal to the affine results in \S\ref{subsection_technical}.} Since $(W_r\Omega_{R/\roi}^n)_p^\comp=W_r\Omega^n_{\frak X/\roi}(\Spf R)$ (Remark \ref{remark_de_rham_witt}(vi)), the middle assumption therefore leads to isomorphisms \[W_r\Omega^n_{\frak X/\roi}(\Spf R)\Isoto \cal H^n(\tilde{W_r\Omega}_{\frak X/\roi})(\Spf R)\] naturally as $\Spf R\subseteq\frak X$ varies over all small affine opens; that proves the $p$-adic Cartier isomorphism.
\end{proof}

To complete the proof of the $p$-adic Cartier isomorphism we must prove the top statement in Lemma~\ref{lemma_reduction_to_torus}, namely the following:

\begin{proposition}\label{proposition_final}
The universal maps \[\hat\lambda_{r,\sub{pro\'et}}^n:(W_r\Omega_{R/\roi}^n)_p^\comp\To H^n(\bb A\Omega_{R/\roi}^\sub{pro\'et}/\tilde\xi_r)\] are isomorphisms in the special case that $R:=\roi\pid{T_1^{\pm1},\dots,T_d^{\pm1}}$.
\end{proposition}
\begin{proof}
The proof will consist merely of assembling results we have already established: indeed, the technical lemmas of Section \ref{subsection_technical} imply that $H^n(\bb A\Omega_{R/\roi}^\sub{pro\'et}/\tilde\xi_r)$ can be calculated in terms of group cohomology, which we identified with the de Rham--Witt complex in Theorem~\ref{theorem_de_rham_witt_group_cohomology}.

Note first that the map\footnote{This map is injective and identifies the right with the $\pid{p,\xi}$-adic completion of the left, i.e., with $\bb A_\sub{inf}\langle \ul U^{\pm1/p^\infty}\rangle$, but we do not need this.} $\bb A_\sub{inf}[\ul U^{\pm1/p^\infty}]\to W(\roi\langle \ul T^{\pm1/p^\infty}\rangle^\flat)$, $U_i^k\mapsto [(T_i^k,T_i^{k/p},T_i^{k/p^2},\dots)]$, when base changed along $\tilde\theta_r$, yields an inclusion $W_r(\roi)[\ul U^{\pm1/p^\infty}]\into W_r(\roi\langle \ul T^{\pm1/p^\infty}\rangle)$, $U_i^k\mapsto T_i^{k/p^r}$ which identifies the right with the $p$-adic completion of the left, i.e., with $W(A^\flat)\langle \ul U^{\pm1/p^\infty}\rangle$; indeed, this follows easily from the ``standard/modified isomorphisms'' which appeared in the proof of Lemma \ref{lemma_lambda_grp}. The map $\bb A_\sub{inf}[\ul U^{\pm1/p^\infty}]\to W(\roi\langle \ul T^{\pm1/p^\infty}\rangle^\flat)$ is obviously also compatible with the actions of the groups $\bb Z^d\subseteq \bb Z_p(1)^d$ (induced by our fixed choice of $p$-power roots of unity) on the left (from \S\ref{subsection_local_Cartier}) and right, thereby inducing the first of the following maps:
\[\xymatrix@R=2mm{
R\Gamma(\bb Z^d,\bb A_\sub{inf}[\ul U^{\pm1/p^\infty}])\ar[r] & R\Gamma_\sub{cont}(\bb Z_p(1)^d,W(\roi\langle \ul T^{\pm1/p^\infty}\rangle^\flat))\ar[r] & R\Gamma_\sub{pro\'et}(\Sp R[\tfrac1p],\bb A_{\sub{inf},X})\\
D^\sub{grp}_{\roi,d}\ar@{=}[u]&D^\sub{cont}\ar@{=}[u]^{\sub{def.}}&D^\sub{pro\'et}_{R/\roi}\ar@{=}[u]
}\]
Here $D^\sub{grp}:=D^\sub{grp}_{\roi,d}$ was the object of study of \S\ref{subsection_local_Cartier}, and the second map is the Cartan--Leray almost quasi-isomorphism which has already appeared, for example just after the statement of Lemma \ref{lemma_technical_1}. Both maps in the previous line are morphisms of commutative algebra objects in $D(\bb A_\sub{inf})$, compatible with the Frobenius on each object (in particular, with $\phi_\sub{grp}$ on the left and $\phi_\sub{pro\'et}$ on the right).

Moreover, we claim that the composition makes the following diagram of structure maps commute for each $r\ge0$:
\[\xymatrix{
H^0(D^\sub{grp}/\tilde\xi_r)\ar[r] & H^0(D_\sub{pro\'et}/\tilde\xi_r)\\
W_r(\roi[\ul T^{\pm1}])\ar@{^(->}[r]\ar[u]^{\lambda_{r,\sub{grp}}} & W_r(\roi\langle\ul T^{\pm1}\rangle)\ar[u]_{\lambda_{r,\sub{pro\'et}}}
}\]
The proof of this compatibility is a straightforward chase through the definitions of the structure maps $\lambda_{r,\sub{grp}}$ and $\lambda_{r,\sub{pro\'et}}$. We first identify the top row via $\tilde\theta_r$ with the composition of the top row of the following diagram:
\[\xymatrix@C=1.5cm{
W_r(\roi)[\ul U^{\pm1/p^\infty}]^{\bb Z^d}\ar[r]^{U_i^k\mapsto T_i^{k/p^r}} & W_r(\roi\langle \ul T^{\pm1/p^\infty}\rangle)^{\bb Z_p(1)^d}\ar[r] & H^0_\sub{pro\'et}(\Spf R,W_r(\hat\roi_X^+))\\
W_r(\roi[\ul T^{\pm1}])\ar@{^(->}[rr]\ar[u]^{\lambda_{r,\sub{grp}}} && W_r(\roi\langle\ul T^{\pm1}\rangle)\ar[u]_{\lambda_{r,\sub{pro\'et}}}\ar@{_(->}[ul]
}\]
The diagonal arrow here is the obvious inclusion (it is actually an isomorphism); since the Cartan--Leray map (i.e., top right horizontal arrow) is one of $W_r(\roi\pid{T^{\pm1}})$-algebras and $\lambda_{r,\sub{pro\'et}}$ was defined to be precisely the algebra structure map, the resulting triangle commutes. Commutativity of the remaining trapezium is tautological: the definition of $\lambda_{r,\sub{grp}}$ in the proof of Lemma \ref{lemma_lambda_grp} was exactly to make this diagram (or, more precisely, the analogous diagram with $W_r(\roi[\ul T^{\pm1/p^\infty}])$ instead of $W_r(\roi\pid{\ul T^{\pm1/p^\infty}})$) commute.



By the naturality of Theorem \ref{theorem_witt_main}, the following diagram therefore commutes:
\comment{
\[\xymatrix@C=5mm{
H^n(L\eta_{[\zeta_{p^r}]-1}(D^\sub{grp}\dotimes_{W(A^\flat)}W(A^\flat)/\tilde\xi_r))\ar[r]&
H^n(L\eta_{[\zeta_{p^r}]-1}(D^\sub{cont}\dotimes_{W(A^\flat)}W(A^\flat)/\tilde\xi_r))\ar[r]&
H^n(L\eta_{[\zeta_{p^r}]-1}(D^\sub{pro\'et}_{R/\roi}\dotimes_{W(A^\flat)}W(A^\flat)/\tilde\xi_r))\\
\cal W_r^n(D^\sub{grp})\ar[u]_\cong&&\cal W_r^n(D^\sub{pro\'et}_{R/\roi})\ar[u]^\cong\\
W_r\Omega^n_{\roi[\ul T^{\pm1}]/\roi}\ar[u]^{\lambda^n_{r,\sub{grp}}}\ar@{^(->}[rr]&&W_r\Omega^n_{\roi\pid{\ul T^{\pm1}}/\roi}\ar[u]_{\lambda^n_{r,\sub{pro\'et}}}
}\]
}

\[\xymatrix{
\cal W_r^n(D^\sub{grp})\ar[r]&\cal W_r^n(D^\sub{pro\'et}_{R/\roi})\\
W_r\Omega^n_{\roi[\ul T^{\pm1}]/\roi}\ar[u]^{\lambda^n_{r,\sub{grp}}}\ar[r]&W_r\Omega^n_{\roi\pid{\ul T^{\pm1}}/\roi}\ar[u]_{\lambda^n_{r,\sub{pro\'et}}}
}\]
The bottom horizontal arrow here becomes an isomorphism after $p$-adic completion,\footnote{By Remark \ref{remark_de_rham_witt}(vi), the $p$-adic completions may be identified respectively with $\projlim_sW_r\Omega^n_{(\roi[\ul T^{\pm1}]/p^s)/(\roi/p^s\roi)}$ and $\projlim_sW_r\Omega^n_{(\roi\pid{\ul T^{\pm1}}/p^s)/(\roi/p^s\roi)}$, which are clearly the same.} and $\lambda^n_{r,\sub{grp}}$ was proved to be an isomorphism in Theorem \ref{theorem_de_rham_witt_group_cohomology}; so to complete the proof it remains to show that the top horizontal arrow identifies $\cal W_r^n(D^\sub{pro\'et}_{R/\roi})$ with the $p$-adic completion of $\cal W_r^n(D^\sub{grp})$. But the top horizontal arrow is precisely $H^n$ of the composition
\[L\eta_{[\zeta_{p^r}]-1}(D^\sub{grp}/\tilde\xi_r)\To
L\eta_{[\zeta_{p^r}]-1}(D^\sub{cont}/\tilde\xi_r)\To
L\eta_{[\zeta_{p^r}]-1}(D^\sub{pro\'et}_{R/\roi}/\tilde\xi_r),
\]
where the second arrow is the quasi-isomorphism (t5) of Lemma \ref{lemma_technical_2}. Meanwhile, the first arrow identifies the middle term with the derived $p$-adic completion of the left: indeed, $L\eta_{[\zeta_{p^r}]-1}$ commutes with $p$-adic completion by Remark \ref{remark_non_primitive}, so it is enough to check that $D^\sub{cont}/\tilde\xi_r=R\Gamma_\sub{cont}(\bb Z_p(1)^d,W_r(\roi\langle \ul T^{\pm1/p^\infty}\rangle)$ is the derived $p$-adic completion of $D^\sub{grp}/\tilde\xi_r=R\Gamma(\bb Z^d,W_r(\roi)[\ul U^{\pm1/p^\infty}])$; but this follows from $W_r(\roi\langle \ul T^{\pm1/p^\infty}\rangle)$ being the $p$-adic completion of $W_r(\roi[\ul T^{\pm1/p^\infty}])$. So, finally, recall that the cohomology groups of $L\eta_{[\zeta_{p^r}]-1}(D^\sub{grp}/\tilde\xi_r)$ are $p$-torsion-free (since $D^\sub{grp}$ satisfies ($\cal W2$) and ($\cal W3$)), whence $H^n$ of its derived $p$-adic completion is the same as the naive $p$-adic completion of its $H^n$.
\end{proof}


This completes the proof of the $p$-adic Cartier isomorphism and these notes.

\begin{appendix}
\section{$\bb A_\sub{inf}$ and its modules}\label{appendix_A_inf}
The base ring for the cohomology theory constructed in [BMS] is Fontaine's infinitesimal period ring $\bb A_\sub{inf}:=W(\roi^\flat)$, where $\roi$ is the ring of integers of a complete, non-archimedean, algebraically closed field $\bb C$ of mixed characteristic. Since $\roi$ is a perfectoid ring (Example \ref{example_O}), the general theory developed in Section \ref{section_perfectoid} (including \S\ref{subsection_roots_of_unity}) applies in particular to $\roi$. Our goal here is firstly to present a few results which are particular to $\roi$ in order to familiarise the reader, who may be encountering these objects for the first time, with $\roi$ and $\bb A_\sub{inf}$; then we will explain some of the finer theory of modules over $\bb A_\sub{inf}$.

We begin by recalling from \cite[\S3]{Scholze2012} that $\roi^\flat$ is the ring of integers of $\bb C^\flat:=\Frac\roi^\flat$ (footnote \ref{footnote_int_dom} shows that $\roi^\flat$ is an integral domain), which is a non-archimedean, algebraically closed field of characteristic $p>0$, with the same residue field $k$ as $\roi$. The absolute value on $\bb C^\flat$ is given by multiplicatively extending the absolute value on  on $\roi^\flat$ given by \[\roi^\flat=\projlim_{x\mapsto x^p}\roi\xTo{x\mapsto x^{(0)}} \roi\xTo{|\cdot|}\bb R_{\ge0},\] where the first arrow uses the convention introduced just before Lemma \ref{lemma_theta}, and the second arrow is the absolute value on $\roi$. The reader may wish to check that this is indeed an absolute value, i.e., satisfies the ultrametric inequality, that $\bb C^\flat$ is complete under it, and that the ring of integers is exactly $\roi^\flat$. The existence of the canonical projection $\roi^\flat\to\roi/p\roi$ implies that $\roi^\flat$ and $\roi$ have the same residue field. Hensel's lemma shows that $\bb C^\flat$ is algebraically closed.\footnote{We sketch the proof here, which is obtained by reversing the roles of $\roi$ and $\roi^\flat$ in \cite[Prop.~3.8]{Scholze2012}. Let $p^\flat:=(p,p^{1/p},p^{1/p^2},\dots)\in A^\flat$, whose absolute value $|p^\flat|=|p|$ we may normalise to $p^{-1}$ for simplicity of notation. It is sufficient to prove the following, which allows a root to any given polynomial to be built by successive approximation: If $f(X)\in\roi^\flat[X]$ is a monic irreducible polynomial of degree $d\ge1$, and $\al\in\roi^\flat$ satisfies $|f(\al)|\le p^{-n}$ for some $n\ge0$, then there exists $\ep\in\roi$ satisfying $|\ep|\le p^{-n/d}$ and $|f(\al+\ep)|\le p^{-(n+1)}$. Well, given such $f(X)$ and $\al$, use the fact that $\bb C^\flat$ and $\bb C$ have the same value group (this is easy to prove), which is divisible since $\bb C$ is algebraically closed, to find $c\in\roi^\flat$ such that $c^{-d}f(\al)$ is a unit of $\roi^\flat$. Then $g(X):=c^{-d}f(\al+cX)$ is a monic irreducible polynomial in $\bb C^\flat[X]$ whose constant coefficient lies in $\roi^\flat$ (even $\roi^{\flat\times}$); a standard consequence of Hensel's lemma is then that $g(X)\in\roi^\flat[X]$. Next observe that the canonical projection $\roi^\flat\to\roi/p\roi$ has kernel $p^\flat\roi^\flat$ ({\em Proof}. Either argue using valuations, or extract a more general result from the proof of Lemma \ref{lemma_inj_of_Frob}.), whence every monic polynomial in $\roi^\flat/p^\flat\roi^\flat$ has a root. So, by lifting a root we find $\beta\in\roi^\flat$ satisfying $g(\beta)\in p^\flat\roi^\flat$; this implies that $f(\al+c\beta)\in f(\beta) p^{\flat}\roi^\flat$, and so $\ep:=c\beta$ has the desired property. $\square$
}

Now we turn to $\bb A_\sub{inf}$. Let $t\in\bb A_\sub{inf}$ be any element whose image in $\bb A_\sub{inf}/p\bb A_\sub{inf}=\roi^\flat$ belongs to $\frak m^\flat\setminus\{0\}$; examples include $t=[\pi]$, where $m^\flat\setminus\{0\}$, and $t=\xi$, where $\xi$ is any generator of $\ker\theta$. Then $p,t$ is a regular sequence, and $\bb A_\sub{inf}$ is a $\pid{p,t}$-adically complete local ring whose maximal ideal equals the radical of $\pid{p,t}$; in short, $\bb A_\sub{inf}$ ``appears two-dimensional and Cohen--Macaulay''.

In fact, as we will explain the result of this appendix, modules (more precisely, finitely presented modules which become free after inverting $p$) over $\bb A_\sub{inf}$ even behave as though $\bb A_\sub{inf}$ were a two-dimensional, regular local ring.\footnote{However, $\bb A_\sub{inf}$ is not Noetherian, it is usually not coherent \cite{Kedlaya2016}, and the presence of certain infinitely generated, non-topologically-closed ideals implies that it has Krull dimension $\ge 3$...} Further details may be found in [BMS, \S4.2].


\begin{remark}
In light of the goal, it is sensible to recall the structure of modules over any two-dimensional regular local ring $\Lambda$, such as $\Lambda=\roi_K[[T]]$ where $\roi_K$ is a discrete valuation ring. Let $\pi,t\in\Lambda$ be a system of local parameters and $\frak m=\pid{\pi,t}$ its maximal ideal.
\begin{enumerate}
\item Most importantly, any vector bundle on the punctured spectrum $\Spec \Lambda\setminus\{\frak m\}$ extends uniquely to a vector bundle on $\Spec\Lambda$.
\item Finitely generated modules over $\Lambda$ are perfect, i.e., admit finite length resolutions by finite free $\Lambda$-modules. ({\em Proof.} Immediate from the regularity of $\Lambda$.)
\item If $M$ is any finitely generated $\Lambda$-module, then there is a functorial short exact sequence \[0\To M_\sub{tor}\To M\To M_\sub{free}\To \res M\To 0\] of $\Lambda$-modules, where $M_\sub{tor}$ is torsion, $M_\sub{free}$ is finite free, and $\res M$ is killed by a power of $\frak m$. 

{\em Proof}. $M_\sub{tor}$ is by definition the torsion submodule of $M$, whence $M/M_\sub{tor}$ restricts  to a torsion-free coherent sheaf on the punctured spectrum $\Spec \Lambda\setminus\{\frak m\}$; but the punctured spectrum is a regular one-dimensional scheme, so this torsion-free coherent sheaf is necessary a vector bundle, and so extends to a vector bundle on $\Spec \Lambda$ by (i); this vector bundle corresponds to a finite free $\Lambda$-module $M_\sub{free}$ which contains $M/M_\sub{tor}$, with the ensuing quotient $\res M$ being supported at the closed point of $\Spec\Lambda$. $\square$
\item Finite projective modules over $\Lambda[\tfrac1\pi]$ are finite free.

{\em Proof}. Let $N$ be a finite projective $\Lambda[\tfrac1\pi]$-module, and pick a finitely generated $\Lambda$-module $N'\subseteq N$ satisfying $N'[\tfrac1\pi]=N$. Then $N'_\frak p$ is a projective module over $\Lambda_\frak p$ for every non-maximal prime ideal $\frak p\subseteq\Lambda$: indeed either $\pi\notin\frak p$, in which case $N'_\frak p$ is a localisation of the projective module $N$, or $\frak p=\pid\pi$, in which case $\Lambda_\frak p$ is a discrete valuation ring and it is sufficient to note that $N_\frak p'$ obviously has no $\pi$-torsion. This means that $N'$ restricts to a vector bundle on the punctured spectrum, whose unique extension to a finite free $\Lambda$-module $N''$ satisfies $N''[\tfrac1\pi]=N$. $\square$
\end{enumerate}
\end{remark}

\begin{theorem}\label{theorem_modules_over_A_inf}
\begin{enumerate}\itemsep0pt
\item (Kedlaya) Any vector bundle on the punctured spectrum \[\Spec \bb A_\sub{inf}\setminus\{\text{the max.~ideal of }\bb A_\sub{inf}\}\] extends uniquely to a vector bundle on $\Spec \bb A_\sub{inf}$.
\item If $M$ is a finitely presented $\bb A_\sub{inf}$-module such that $M[\tfrac1p]$ is finite free over $\bb A_\sub{inf}[\tfrac1p]$, then $M$ is perfect (again, this means that $M$ admits a finite length resolution by finite free $\bb A_\sub{inf}$-modules).
\item If $M$ is a finitely presented $\bb A_\sub{inf}$-module such that $M[\tfrac1p]$ is finite free over $\bb A_\sub{inf}[\tfrac1p]$, then there is functorial short exact sequence of $\bb A_\sub{inf}$-modules \[0\To M_\sub{tor}\To M\To M_\sub{free}\To \res M\To 0\] such that: $M_\sub{tor}$ is a perfect $\bb A_\sub{inf}$-module killed by a power of $p$; $M_\sub{free}$ is a finite free $\bb A_\sub{inf}$-module; and $\res M$ is a perfect $\bb A_\sub{inf}$-module killed by a power of the ideal $\pid{p,t}$.
\item Finite projective modules over $\bb A_\sub{inf}[\tfrac1p]$ are finite free.
\end{enumerate}
\end{theorem}
\begin{proof}
We have nothing to say about (i) here, and refer instead to [BMS, Lem.~4.6]. We will also only briefly comment on the remaining parts of the theorem, since these self-contained results may be easily read in [BMS, \S4.2].

(ii): By clearing denominators in a basis for $M[\tfrac1p]$ to construct a finite free $\bb A_\sub{inf}$-module $M'\subseteq M$ satisfying $M'[\tfrac1p]=M[\tfrac1p]$, we may reduce to the case that $M$ is killed by a power of $p$, i.e., $M$ is a $\bb A_\sub{inf}/p^r\bb A_\sub{inf}=W_r(\roi^\flat)$-module  for some $r\gg0$. By an induction on $r$, using that $W_r(\roi^\flat)$ can be shown to be coherent, one can reduce to the case $r=1$, in which case it easily follows from the classification of finitely presented modules over the valuation ring $\roi^\flat$: they have the shape $(\roi^\flat)^n\oplus\roi^\flat/a_1\roi^\flat\oplus\cdots\oplus \roi^\flat/a_m\roi^\flat$, for some $n\ge1$ and $a_i\in\roi^\flat$, and so in particular are perfect.

(iii): This is proved similarly to the analogous assertion in the previous remark.

(iv): This is proved exactly as the analogous assertion in the previous remark, once it is checked that the localisation $\bb A_{\sub{inf},\pid{p}}$ is a discrete valuation ring.
\end{proof}

\begin{corollary}\label{corollary_condition_for_freeness}
Let $M$ be a finitely presented $\bb A_\sub{inf}$-module such that $M[\tfrac1p]$ is finite free over $\bb A_\sub{inf}[\tfrac1p]$. If either $M\otimes_{\bb A_\sub{inf}}W(k)$ or $M\otimes_{\bb A_\sub{inf}}\roi$ is $p$-torsion-free (equivalently, finite free over $W(k)$ or $\roi$ respectively), then $M$ is a finite free $\bb A_\sub{inf}$-module.
\end{corollary}
\begin{proof}
It follows easily from the hypothesis that the map $M\to M_\sub{free}$ in Theorem \ref{theorem_modules_over_A_inf}(iii) becomes an isomorphism after tensoring by $W(k)$ or $\roi$; hence $M[\tfrac1p]$ and $M\otimes_{\bb A_\sub{inf}}k$ have the same rank over $\bb A_\sub{inf}[\tfrac1p]$ and $k$ respectively. But an easy Fitting ideal argument shows that if $N$ is a finitely presented module over a local integral domain $R$ satisfying $\dim_{\Frac R}(N\otimes_R\Frac R)=\dim_{k(R)}(N\otimes_Rk(R))$, then $N$ is finite free over $R$.
\end{proof}

To state and prove the next corollary we use the elements $\xi,\xi_r,\mu\in\bb A_\sub{inf}$ constructed in \S\ref{subsection_roots_of_unity}:

\begin{corollary}\label{corollary_Fitting_ideal_trick}
Let $M$ be a finitely presented $\bb A_\sub{inf}$-module, and assume:
\begin{itemize}
\item $M[\tfrac1{p\mu}]$ is a finite free $\bb A_\sub{inf}[\tfrac1{p\mu}]$-module of the same rank as the $W(k)$-module $M\otimes_{\bb A_\sub{inf}}W(k)$.
\item There exists a Frobenius-semi-linear endomorphism of $M$ which becomes an isomorphism after inverting $\xi$.
\end{itemize}
Then $M[\tfrac1p]$ is finite free over $\bb A_\sub{inf}[\tfrac1p]$.
\end{corollary}
\begin{proof}
We must show that each Fitting ideal of the $\bb A_\sub{inf}[\tfrac1p]$-module $M[\tfrac1p]$ is either $0$ or $\bb A_\sub{inf}[\tfrac1p]$; indeed, this means exactly that $M[\tfrac1p]$ is finite projective over $\bb A_\sub{inf}[\tfrac1p]$, which is sufficient by Theorem \ref{theorem_modules_over_A_inf}(iv). Since Fitting ideals behave well under base change, it is equivalent to prove that the first non-zero Fitting ideal $J\subseteq\bb A_\sub{inf}$ of $M$ contains a power of $p$. Again using that Fitting ideals base change well, our hypotheses imply that $J\bb A_\sub{inf}[\tfrac1{p\mu}]=\bb A_\sub{inf}[\tfrac1{p\mu}]$ and $JW(k)\neq 0$; that is, $J$ contains a power of $p\mu$ and $J+W(\frak m^\flat)$ contains a power of $p$, where $W(\frak m^\flat):=\ker(\bb A_\sub{inf}\to W(k))$. Because of the existence of the Frobenius on $M$, we also know that $J$ and $\phi(J)$ are equal up to a power of $\phi(\xi)$. In conclusion, we may pick $N\gg$ such that
\begin{enumerate}
\item $(p\mu)^N\in J$;
\item $p^N\in J+W(\frak m)$;
\item $\phi(\xi)^N\phi(J)\subseteq J$ and $\phi(\xi)^NJ\subseteq \phi(J)$.
\end{enumerate}

Since $W(\frak m)$ is the $p$-adic completion of the ideal generated by $\phi^{-r}(\mu^N)$, for all $r\ge 0$,\footnote{\label{last_footnote}By an easy induction using that $p$ is a non-zero-divisor in $\bb A_\sub{inf}/W(\frak m)$, this follows from the fact that the maximal ideal of $\bb A_\sub{inf}/p\bb A_\sub{inf}=\roi^\flat$ is generated by the elements $\phi^{-r}(\ep)-1$, for all $r\ge0$.} observation (ii) lets us write $p^N=\al+\beta\phi^{-r}(\mu^N)+\beta'p^{N+1}$ for some $\al\in J$ and $\beta,\beta'\in \bb A_\sub{inf}$, and $r\ge 0$. Since $1-\beta'p$ is invertible, we may easily suppose that $\beta'=0$, i.e., $p^N=\al+\beta\phi^{-r}(\mu^N)$. Multiplying through by $p^N\xi_r^N$ gives $\xi_r^Np^{2N}=p^N\xi_r^N\al+\beta p^N\mu^N$, which belongs to $J$ by (i) and (ii).

We claim, for any $a,i\ge1$, that
\begin{center}
$\xi_r^ap^\sub{some power}\in J\implies \xi_r^{a-1}p^\sub{some other power}\in J$.
\end{center}
A trivial induction then shows that $J$ contains a power of $p$, thereby completing the proof, and so it remains only to prove this claim. Suppose $\xi_r^ap^b\in J$ for some $a,b,i\ge 1$. Then $\phi^r(\xi_r)^ap^b\in\phi^r(J)$, and so $\phi^r(\xi_r)^{a+N}p^b\in J$ (since an easy generalisation of (iii) implies that $\phi^r(\xi_r)^N\phi^r(J)\subseteq J$). But $\phi^r(\xi_r)\equiv p^r$ mod $\xi_r$, so we may write $\phi^r(\xi_r)^{a+N}=p^{r(a+N)}+\al \xi_r$ for some $\al\in\bb A_\sub{inf}$ and thus deduce that $J\ni (p^{r(a+N)}+\al\xi_r)p^b=p^{r(a+N)+b}+\al \xi_rp^b$. Now multiply through by $\xi_r^{a-1}$ and use the supposition to obtain $\xi_r^{a-1}p^{r(a+N)+b}\in J$, as required.
\end{proof}

We will also need the following to eliminate the appearance of higher Tors in the crystalline specialisation of the $\bb A_\sub{inf}$-cohomology theory:

\begin{lemma}\label{lemma_Tors_over_A_inf}
Let $M$ be an $\bb A_\sub{inf}$-module such that $M[\tfrac1p]$ is flat over $\bb A_\sub{inf}[\tfrac1p]$. Then $\Tor_*^{\bb A_\sub{inf}}(M,W(k))=0$ for $*>1$.
\end{lemma}
\begin{proof}
Let $[\frak m^\flat]\subseteq W(\frak m^\flat)$ be the ideal of $\bb A_\sub{inf}$ which is generated by Teichm\"uller lifts of elements of $\frak m^\flat$. We first observe that $\bb A_\sub{inf}/[\frak m^\flat]$ is $p$-torsion-free and has Tor-dimension $=1$ over $\bb A_\sub{inf}$: indeed, $[\frak m^\flat]$ is the increasing union of the ideals $[\pi]\bb A_\sub{inf}$, for $\pi\in\frak m^\flat\setminus\{0\}$, and the claims are true for $\bb A_\sub{inf}/[\pi]\bb A_\sub{inf}$ since $p,[\pi]$ is a regular sequence of $\bb A_\sub{inf}$.

Next, since $W_r(\frak m^\flat)$ is generated by the analogous Teichm\"uller lifts in $W_r(\roi)=\bb A_\sub{inf}/p^r\bb A_\sub{inf}$, for any $r\ge1$ (c.f., footnote \ref{footnote_almost_W_r}), the quotient $W(\frak m^\flat)/[\frak m^\flat]$ is $p$-divisible. Combined with the previous observation, it follows that $W(\frak m^\flat)/[\frak m^\flat]$ is uniquely $p$-divisible, i.e., an $\bb A_\sub{inf}[\tfrac1p]$-module, whence \[\Tor_*^{\bb A_\sub{inf}}(W(\frak m^\flat)/[\frak m^\flat],M)=\Tor_*^{\bb A_\sub{inf}[\tfrac1p]}(W(\frak m^\flat)/[\frak m^\flat],M[\tfrac1p]),\] which vanishes for $*>0$ by the hypothesis on $M$. Combining this with the short exact sequence \[0\to W(\frak m^\flat)/[\frak m^\flat]\to \bb A_\sub{inf}/[\frak m^\flat]\to \bb A_\sub{inf}/W(\frak m^\flat)=W(k)\to 0\] and the initial observation about the Tor-dimension of the middle term completes the proof.
\end{proof}

\section{Two lemmas on Koszul complexes}
Let $R$ be a ring, and $g_1,\dots,g_d\in R$. The associated Koszul complex will be denoted by $K_R(g_1,\dots,g_d)=\bigotimes_{i=1}^d K_R(g_i)$, where $K_R(g_i):=[R\xto{g_i}R]$. Here we state two useful lemmas concerning such complexes, the second of which describes the behaviour of the d\'ecalage functor.

\begin{lemma}\label{lemma_on_Koszul_2}
Let $g\in R$ be an element which divides $g_1,\dots,g_d$, and such that $g_i$ divides $g$ for some $i$. Then there are isomorphisms of $R$-modules \[H^n(K_R(g_1,\dots,g_d))\cong R[g]^{\binom{d-1}n}\oplus R/gR^{\binom{d-1}{n-1}}\] for all $n\ge0$.
\end{lemma}
\begin{proof}\
[Lem.~7.10, BMS].
\end{proof}

\begin{lemma}\label{lemma_on_Koszul_1}
Let $f\in R$ be a non-zero-divisor such that, for each $i$, either $f$ divides $g_i$ or $g_i$ divides $f$. Then:
\begin{itemize}\itemsep0pt
\item If $f$ divides $g_i$ for all $i$, then $\eta_f K_R(g_1,\dots,g_d)\cong K_R(g_1/f,\dots,g_d/f)$. 
\item If $g_i$ divides $f$ for some $i$, then $\eta_f K_R(g_1,\dots,g_d)$ is acyclic.
\end{itemize}
\end{lemma}
\begin{proof}\
[Lem.~7.9, BMS].
\end{proof}

\end{appendix}

\bibliographystyle{acm}
\bibliography{../../Bibliography}
\end{document}